\documentclass[onefignum,onetabnum]{siamart190516}

\newcommand{\vlong}[1]{}

\newcommand{\yh}[1]{#1}

\usepackage{lipsum}
\usepackage{amsfonts}
\usepackage{graphicx}
\usepackage{epstopdf}
\usepackage{algorithmic}
\usepackage{url}

\ifpdf
  \DeclareGraphicsExtensions{.eps,.pdf,.png,.jpg}
\else
  \DeclareGraphicsExtensions{.eps}
\fi

\graphicspath{{../include-figure/}}
\newtheorem{conj}{Conjecture}[section]

\newsiamremark{remark}{Remark}
\newsiamremark{hypothesis}{Hypothesis}
\crefname{hypothesis}{Hypothesis}{Hypotheses}
\newsiamthm{claim}{Claim}
\usepackage{xspace,color}

\headers{Asymptotic Convergence Speed of AA and NGMRES}{Hans De Sterck and Yunhui He}

\title{
On the Asymptotic Linear Convergence Speed of Anderson Acceleration, Nesterov Acceleration, and Nonlinear GMRES\thanks{To appear in SIAM Journal on Scientific Computing. Submitted June 20, 2020. Accepted October 27, 2020.
\funding{This work was funded in part by NSERC of Canada (RGPIN-2019-04155).}}}

\author{Hans De Sterck\thanks{Department of Applied Mathematics
University of Waterloo, 200 University Ave W, Waterloo, ON N2L 3G1, Canada
  (\email{hdesterck@uwaterloo.ca}, \email{yunhui.he@uwaterloo.ca}).}
  \and Yunhui He\footnotemark[2]}

\usepackage{amsopn}

\begin{document}

\maketitle

\begin{abstract}
We consider nonlinear convergence acceleration methods for fixed-point iteration $x_{k+1}=q(x_k)$, including Anderson
acceleration (AA), nonlinear GMRES (NGMRES), and Nesterov-type acceleration (corresponding to AA with window size one).
We focus on fixed-point methods that converge asymptotically linearly with convergence factor $\rho<1$
and that solve an underlying fully smooth and non-convex optimization problem.
It is often observed that AA and NGMRES substantially improve the asymptotic convergence behavior of the
fixed-point iteration, but this improvement has not been quantified theoretically. We investigate this problem
under simplified conditions.
First, we consider stationary versions of AA and NGMRES, and determine coefficients that result
in optimal asymptotic convergence factors, given knowledge of the spectrum of $q'(x)$ at the fixed point $x^*$.
This allows us to understand and quantify the asymptotic convergence improvement that can be provided by
nonlinear convergence acceleration, viewing $x_{k+1}=q(x_k)$ as a nonlinear preconditioner for AA and NGMRES.
Second, for the case of infinite window size, we consider linear asymptotic convergence bounds for GMRES
applied to the fixed-point iteration linearized about $x^*$. Since AA and NGMRES are equivalent to GMRES
in the linear case, one may expect the GMRES convergence factors to be relevant for AA and NGMRES as
$x_k \rightarrow x^*$.
Our results are illustrated numerically for a class of test problems from canonical tensor decomposition, comparing
steepest descent and alternating least squares (ALS) as the fixed-point iterations that are accelerated by AA and NGMRES.
Our numerical tests show that both approaches allow us to estimate asymptotic convergence speed for nonstationary
AA and NGMRES with finite window size.
\end{abstract}




\section{Introduction}
\label{sec:intro}

This paper concerns convergence acceleration methods for nonlinear fixed-point iterations of the type
\begin{equation}\label{eq:fixed-point}
  x_{k+1}=q(x_{k}) \qquad k=0,1,2,\ldots,
\end{equation}
which may be employed to solve nonlinear equation systems
\begin{equation}\label{eq:nonl}
  g(x)=0,
\end{equation}
or other scientific computing problems such as nonlinear integral equations or optimization problems.
In particular, we consider the case where there is an underlying optimization problem associated with Equations \cref{eq:fixed-point,eq:nonl} to find a local minimum $x^*$ of
\begin{equation}
  \min_x f(x),
\label{eq:minf}
\end{equation}
where we assume in this paper that $f(x)$ may be nonconvex and is twice continuously differentiable, such that local minima
satisfy
\begin{align}
	g(x):= \nabla f(x) = 0,
\label{eq:g}
\end{align}
and the Hessian of $f(x)$, denoted by $H(x)$, exists.
We will also assume that $q(x)$ is continuously differentiable such that its Jacobian, $q'(x)$, exists.

\subsection{Nonlinear acceleration methods}
In this paper we consider nonlinear acceleration methods of two types \yh{with window size $m$}:
\begin{align}
  x_{k+1}&= q(x_k) + \sum_{i=1}^{\min(k,m)}\beta_{i}^{(k)}(q(x_k)-q(x_{k-i})) &\qquad k=0,1,2,\ldots,
\label{eq:AA}\\
%
  x_{k+1}&= q(x_k) + \sum_{i=0}^{\min(k,m)}\beta_{i}^{(k)}(q(x_k)-x_{k-i})  &\qquad k=0,1,2,\ldots.
\label{eq:NGMRES}
\end{align}
%
When the coefficients $\beta_{i}^{(k)}$ are determined by solving a least-squares
problem in every step $k$ that minimizes a linearized residual in the new iterate $x_{k+1}$, method
\cref{eq:AA} is known as \emph{Anderson acceleration (AA)} \cite{anderson1965iterative}, and method
\cref{eq:NGMRES} is known as the \emph{nonlinear generalized minimal residual (NGMRES)} iteration
\cite{washio1997krylov,oosterlee2000krylov}.
Specifically, AA($m$), with window size $m$, solves in every iteration the linear least-squares problem
\begin{align}
	\min_{\{\beta_i^{(k)} \}} \| r(x_k) + \sum_{i=1}^{\min(k,m)} \beta_i^{(k)} ( r(x_k) - r(x_{k-i})) \|_2^2,
\label{eq:Andersonbetas}
\end{align}
of size up to $m \times m$, with the residuals $r(x)$ of the fixed-point iteration defined by
\begin{align}
	r(x)=x-q(x).
\label{eq:resid}
\end{align}
NGMRES($m$) usually solves the linear least-squares problem
\begin{align}
	\min_{\{\beta_i^{(k)} \}} \| g(q(x_k)) + \sum_{i=0}^{\min(k,m)} \beta_i^{(k)} ( g(q(x_k)) - g(x_{k-i})) \|_2^2,
\label{eq:NGMRESbetas}
\end{align}
in each iteration, minimizing the linearized residual $g(x)$ of nonlinear equation \cref{eq:nonl} evaluated at the accelerated iterate $x_{k+1}$
\cite{washio1997krylov,oosterlee2000krylov,sterck2012nonlinear}.

Anderson acceleration dates back to the 1960s \cite{anderson1965iterative} and has over the years seen
substantial use in computational science. It has gained significant new interest over the past decade
\cite{fang2009two,walker2011anderson,brune2015composing,toth2015,evans2020proof}, both in terms of theoretical
developments and applications. The closely related nonlinear GMRES method was developed more
recently \cite{washio1997krylov,oosterlee2000krylov} and has also been used in various applications
\cite{sterck2012nonlinear,brune2015composing}.
Both NGMRES and AA are often combined with globalization methods
to safeguard against erratic convergence away from a fixed point, e.g., by using damping or restarting mechanisms
\cite{washio1997krylov,oosterlee2000krylov,brune2015composing}, or by using line search strategies in the case of optimization problems
\cite{sterck2012nonlinear,sterck2013steepest}.
Note also that, besides AA and NGMRES, several other methods can be used
as nonlinear convergence accelerators for fixed-point iterations, including nonlinear conjugate gradients, LBFGS, and algebraic multigrid,
see \cite{sterck2013adaptive,sterck2015nonlinearly,brune2015composing,de2016nonlinearly,de2018nonlinearly}.

When $m=1$ in \cref{eq:AA} and the iteration is applied to convex optimization problems
with steepest-descent (SD) fixed-point iteration \cref{eq:fixed-point} and specific choices
for $\beta_{1}^{(k)}$, \cref{eq:AA} is known as \emph{Nesterov's accelerated gradient descent} method \cite{nesterov1983method,NesterovBook},
which guarantees optimal convergence with sublinear rate $O(1/k^2)$ for convex functions $f(x)$ with $L$-Lipschitz-continuous gradients.
\yh{If, in addition, $f(x)$ is strongly convex with strong convexity constant $\mu$, Nesterov's method improves the linear convergence factor of steepest descent with step size $1/L$ from $1-\mu/L$ to $1-\sqrt{\mu/L}$.}
Nesterov acceleration has been extended to nonconvex functions and to accelerating other iterative optimization methods than SD, see, e.g., \cite{mitchell2020nesterov,ang2019accelerating}, where the $\beta_1^{(k)}$ are not pre-determined as in \cite{nesterov1983method,NesterovBook} nor determined
by least-squares problem (\ref{eq:Andersonbetas}), but are determined heuristically combined with restart.

When the coefficients $\beta_{i}^{(k)}$ in methods \cref{eq:AA,eq:NGMRES} are fixed independent of
$k$, the iterations are known as ($m$+1)-step stationary iterative methods \cite{MR703121,Varga1986,Kerkhoven1992}.
Specifically, we consider stationary AA with window size $m$, denoted by sAA($m$),
\begin{align}
  x_{k+1}&= q(x_k) + \sum_{i=1}^{m}\beta_{i} \, (q(x_k)-q(x_{k-i})) \qquad k=m,m+1,\ldots,
\label{eq:sAA}
\end{align}
and stationary NGMRES with window size $m$, denoted by sNGMRES($m$),
\begin{align}
  x_{k+1}&= q(x_k) + \sum_{i=0}^{m}\beta_{i} \, (q(x_k)-x_{k-i})  \qquad k=m,m+1,\ldots.
\label{eq:sNGMRES}
\end{align}
The linear asymptotic convergence factor of an ($m+1$)-step stationary iterative method at fixed point
$x^*$ is defined by its linear root-convergence factor:
\begin{equation}\label{eq:rhodef}
  \rho = \sup_{x_0,x_1,\cdots,x_{m}} \left( {\limsup\limits_{k\rightarrow\infty}} \, \|x_k-x^{*}\|^{1/k} \right),
\end{equation}
where the starting values $x_0,x_1,\cdots,x_{m}$ are restricted to values for which convergence to $x^*$
takes place \cite{Varga1986}.
Asymptotic convergence factors were analyzed in \cite{Kerkhoven1992} for two stationary
nonlinear acceleration methods -- Chebyshev acceleration and stationary second-order Richardson
iteration.

\begin{figure}
\centering
\includegraphics[width=0.45\linewidth]{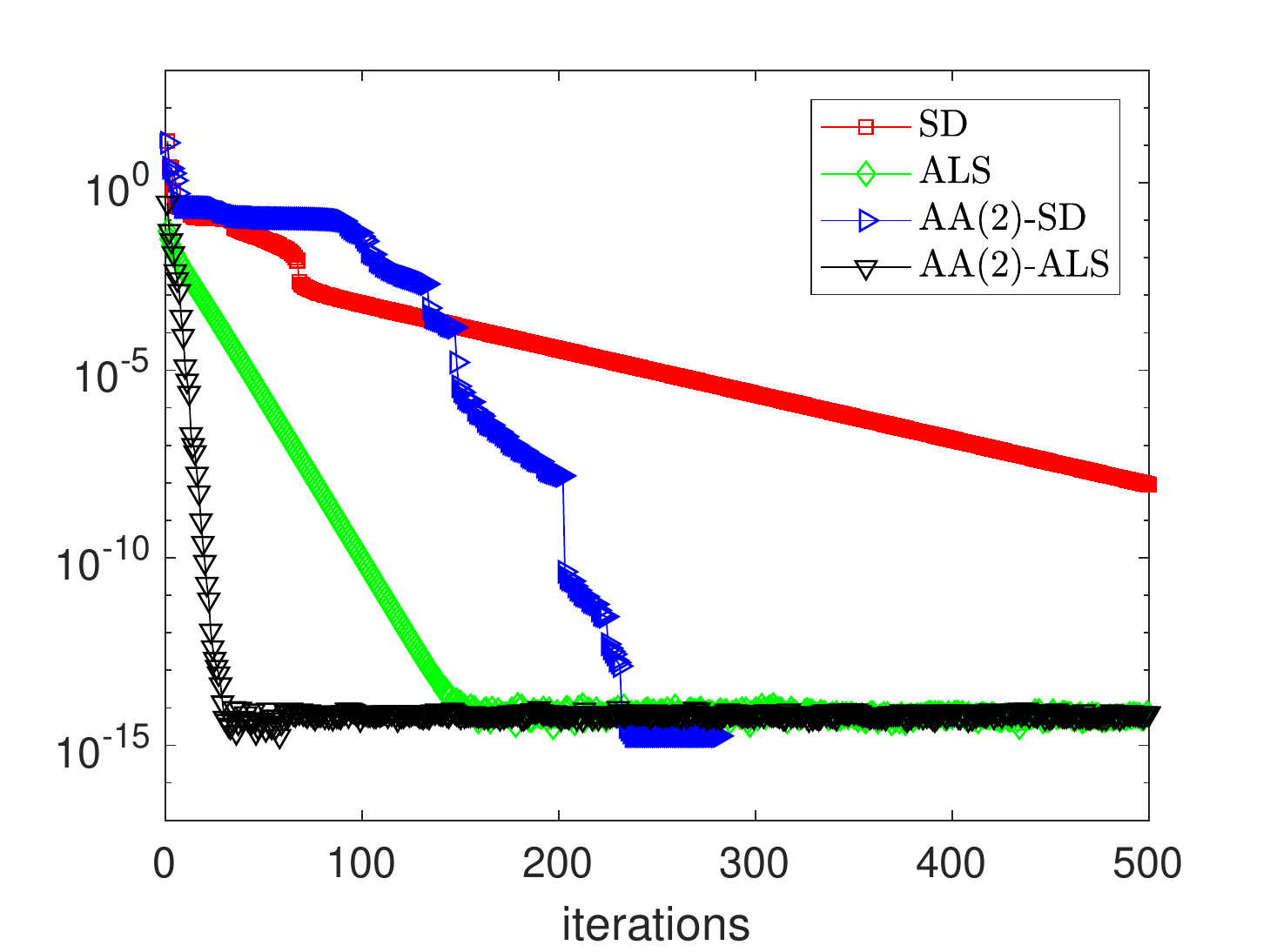}
\caption{Convergence comparison among SD, ALS, AA(2)-SD and AA(2)-ALS for a mildly ill-conditioned canonical tensor decomposition problem with collinearity parameter $c=0.5$. The vertical axis represents $f(x_k)-f(x^*)$, the convergence towards the minimum value of $f(x)$.} \label{compare-SD-ALS-AA-plot}
\end{figure}

When (nonstationary) AA or NGMRES are applied to fixed-point iteration \cref{eq:fixed-point} with differentiable $q(x)$,
this often results in dramatically improved asymptotic convergence behavior near the fixed point $x^*$,
compared to the linear asymptotic convergence factor of fixed-point iteration (\ref{eq:fixed-point}) as
determined by the spectral radius of $q'(x)$ evaluated at $x^*$.
However, there are no known theoretical results to quantify or predict this convergence improvement.
For example, Fig.\ \ref{compare-SD-ALS-AA-plot} shows typical convergence plots for a smooth nonconvex
optimization problem (\ref{eq:minf}) that represents the approximation of a three-dimensional tensor by a low-rank
canonical tensor decomposition \cite{kolda2009tensor,acar2011scalable,sterck2012nonlinear}
(see Sections \ref{sec:back} and \ref{sec:num-sec} for problem description and parameters).
Here optimality equation (\ref{eq:g}) can be solved by fixed-point methods (\ref{eq:fixed-point}) that implement
steepest descent (SD), with fixed-point function $q_{SD}(x)$, or an alternating least-squares (ALS) approach, with $q_{ALS}(x)$. ALS is a form of block coordinate descent or block nonlinear Gauss-Seidel \cite{kolda2009tensor,acar2011scalable}.
Anderson acceleration (in its NGMRES form) was first applied to the problem of canonical tensor decomposition in 2012 in \cite{sterck2012nonlinear},
accelerating the convergence of SD and ALS.
As is well-known, asymptotic convergence of steepest descent is linear with a convergence factor that is
increasingly poor for more ill-conditioned problems \cite{Luenberger} (the problem of Fig.\ \ref{compare-SD-ALS-AA-plot} is mildly ill-conditioned, see Section \ref{sec:num-sec}).
Fig.\ \ref{compare-SD-ALS-AA-plot} shows that ALS converges with a much improved convergence factor relative to SD.
We then apply AA(2) to $q_{SD}(x)$ and to $q_{ALS}(x)$,
and in both cases we see that the convergence is substantially improved.
\yh{In this paper we use canonical tensor decomposition as a test problem to illustrate our findings, because it
is an important problem where dramatic acceleration by AA or NGMRES as in Fig.\ \ref{compare-SD-ALS-AA-plot} has
long been observed but has not yet been explained, and because the ALS iteration exhibits a complex Jacobian
spectrum and, thus, serves well to illustrate an important case in our analysis.}

In this paper, we are interested in quantifying the asymptotic convergence improvement near the fixed
point $x^*$ that is provided by AA, Nesterov and NGMRES compared to the linear asymptotic convergence
factor of fixed-point iteration (\ref{eq:fixed-point}).
To prepare for this endeavour, it is useful to first recall the linear case.

\subsection{The linear case: preconditioned GMRES}

It is well-known that the AA and NGMRES methods of Eqs.\ \cref{eq:AA,eq:NGMRES}
with window size $m=\infty$ are essentially equivalent in the linear case to the well-known
GMRES iterative algorithm for solving $Ax=b$, with $A \in \mathbb{R}^{n \times n}$,
see \cite{washio1997krylov,oosterlee2000krylov,fang2009two,walker2011anderson}.

Specifically, applying NGMRES iteration (\ref{eq:NGMRES}) to fixed-point method (\ref{eq:fixed-point}) reduces
to preconditioned GMRES for $Ax=b$ when using fixed-point function
\begin{align}
q(x)=(I-P\,A)x+P b.
\label{eq:qprec}
\end{align}
Here, $P$ is the preconditioning matrix and fixed-point iteration (\ref{eq:fixed-point})
corresponds to solving the left-preconditioned system $PAx=Pb$ \cite{washio1997krylov,oosterlee2000krylov,sterck2012nonlinear,sterck2013steepest}.
For example, with $P=L^{-1}$, where $L$ is the lower triangular part of $A$, the fixed-point iteration is the Gauss-Seidel iteration and
we obtain GMRES preconditioned by Gauss-Seidel.
When $P=\alpha I$, for some constant $\alpha$, the fixed-point iteration is known as Richardson iteration,
and when $A$ is symmetric positive definite this corresponds to a steepest descent iteration with step length $\alpha$
for minimizing $f(x)=x^T A x/2-b^T x$ \cite{sterck2013steepest};
in other words, preconditioning AA or NGMRES by SD for an optimization problem
corresponds to using the identity preconditioner for GMRES in the linear case \cite{sterck2013steepest}.

It is well-known that the asymptotic convergence of preconditioned GMRES is
determined by matrix properties of $I - q'(x^*)= PA$, including
the condition number, field of values \cite{trefethen2005spectra}, and eigenvalue clustering of the matrix.
Let $r_k=Pb-PAx_k$ be the $k$th residual of the preconditioned GMRES iteration in the linear case.
GMRES minimizes $\|r_k\|$ over an expanding subspace, guaranteeing non-increasing residual norms.
For important classes of matrices $A\in\mathbb{R}^{n \times n}$ and preconditioners $P$
it can be shown that linear convergence bounds
for the preconditioned GMRES residual reduction exist, where the following holds for any initial residual $r_0$:
\begin{align}
\frac{\|r_k\|}{\|r_0\|} \le c \rho^k,
\label{eq:linconv}
\end{align}
with constants $0<c$ and $0<\rho<1$.
For example, for any matrix $PA$ for which 0 does not belong to the field of values of $PA$, a convergence bound of type
(\ref{eq:linconv}) can be computed where the value of $\rho<1$ depends on simple properties of the field of values of $PA$,
and $c<10$ \cite{beckermann2005some};
$\rho$ in (\ref{eq:linconv}) is called an asymptotic convergence factor \cite{beckermann2005some}.
Choosing a suitable, problem-dependent preconditioner $P$ may result in much improved asymptotic convergence factors compared to non-preconditioned GMRES. 

It is important to note, however, that there are also matrices $PA$ and initial residuals $r_0$ for which the
GMRES residual $r_k$ in iteration $k$ remains large until $k$ reaches $n$ and GMRES reaches the exact
solution (in exact arithmetic) \cite{greenbaum1996any}. Moreover, such matrices and initial residuals can be constructed
for any choice of the $n$ eigenvalues of $PA$. So whether or not a useful linear convergence bound of type
(\ref{eq:linconv}) exists (i.e., with $c$ and $\rho$ small enough for the bound to predict
residual reduction that is at least linear for $k \ll n$), depends not only
on the eigenvalue spectrum of $PA$, but also on the angles between the eigenvectors of $PA$. For example,
when $PA$ is normal, the pathological behavior from \cite{greenbaum1996any} does not occur, and, as already
mentioned above, the same is true when 0 does not belong to the field of values of $PA$.

Finally, note also that, just like fixed-point iteration (\ref{eq:fixed-point}) with $q(x)$ given by (\ref{eq:qprec}) is called
a preconditioning iteration for GMRES, fixed-point iteration (\ref{eq:fixed-point}) with nonlinear functions
$q_{SD}(x)$ or $q_{ALS}(x)$ can be viewed as nonlinear preconditioning iterations for NGMRES or AA
\cite{sterck2012nonlinear,sterck2013steepest,brune2015composing}.
The nonlinear preconditioning iteration (\ref{eq:fixed-point}) (inner iteration) can be viewed as accelerating the convergence of NGMRES or AA (outer  iteration), or, alternatively, the outer iteration can be viewed as a nonlinear convergence accelerator for the inner iteration
\cite{washio1997krylov,oosterlee2000krylov,sterck2012nonlinear,sterck2013steepest,brune2015composing}.

\subsection{Convergence theory for nonlinear acceleration methods}
Until recently, little was known about convergence theory for AA and NGMRES.
There was no convergence proof for AA until the recent paper \cite{toth2015}, which shows that AA($m$)
is locally $r$-linearly convergent under the assumptions that $q(x)$ is contractive and the AA coefficients
remain bounded, but there is no proof that AA actually improves the convergence speed.
A convergence proof for NGMRES in the optimization context was given in \cite{sterck2013steepest}, but it relies on a line search globalization
step and only applies to the case where $q(x)$ is steepest descent with a line search that satisfies the Wolfe conditions.
Recently, \cite{evans2020proof} has made progress on the topic of understanding AA convergence acceleration by showing that,
to first order, the convergence gain provided by AA in step $k$ is
quantified by a factor \yh{$\theta_k \le 1$ that equals the ratio of the square root of the optimal value defined in (\ref{eq:Andersonbetas})  to $\|r(x_k)\|_2$}.
However, it is not clear how $\theta_k$ may
be evaluated or bounded in practice and how it may translate to improved asymptotic convergence behavior in general.
This is not surprising, though, since, as discussed above, for linear preconditioned GMRES the existence of linear asymptotic
convergence bounds depends on the properties of $PA=I-q'(x^*)$, see \cref{eq:qprec}, and is, thus, problem-dependent.

Just like in the linear case, it is natural to expect, however, that the asymptotic convergence
speeds of AA and NGMRES applied to nonlinear fixed-point iterations \cref{eq:fixed-point}
will also depend on matrix properties of $I-q'(x^*)$, including
the condition number, field of values and eigenvalue clustering of the matrix.
This paper will develop techniques and approaches
that will allow us to demonstrate that this is indeed the case
and quantify this.
This will shed light on how AA and NGMRES may be effective
in accelerating the asymptotic convergence of fixed-point method
(\ref{eq:fixed-point}) depending on matrix properties of $I-q'(x^*)$.

Since we are not aware of a tractable approach to investigate asymptotic convergence
for the nonstationary versions of AA and NGMRES with finite window size, we first resort
to stationary versions \cref{eq:sAA,eq:sNGMRES} of AA and NGMRES with small window size,
for which we determine the optimal coefficients $\beta_{i}$ that minimize the asymptotic convergence factor
$\rho$ of Eq.\ \cref{eq:rhodef}, given knowledge of $x^*$ and $q'(x^*)$.
The optimal stationary methods we consider are not intended to be practical
computational tools, since we need to know $x^*$ and $q'(x^*)$ to compute the optimal $\beta_{i}$,
but they do allow us to make substantial progress in understanding and quantifying
how sAA and sNGMRES can improve the asymptotic convergence speed of fixed-point iteration
(\ref{eq:fixed-point}).
We derive theoretical results on optimal weights for sAA(1) and sNGMRES(1)
for the case that all eigenvalues of $q'(x^*)$ are real, and for the complex eigenvalue
case. For sAA(1) applied to steepest descent, we obtain known optimal weights for
Nesterov acceleration of steepest descent for the case of sufficiently smooth $f(x)$,
see, for example, \cite{NesterovBook,odonoghue2015adaptive,scieur2016regularized,lessard2016analysis}.
In our numerical results section we also compare with the (nonstationary) Nesterov-type acceleration
methods of \cite{mitchell2020nesterov} with restart.
The nonstationary AA and NGMRES do not use these \emph{globally optimal} stationary coefficients,
but rather perform a \emph{local optimization} of the coefficients in every step $k$ based on Eqs.\ (\ref{eq:Andersonbetas}) and
(\ref{eq:NGMRESbetas}).
As $x$ approaches $x^*$ in the asymptotic regime and $q'(x)$ approaches
$q'(x^*)$, it is not unreasonable to expect the convergence behavior of AA and NGMRES with locally-optimal
$\beta_i^{(k)}$ weights to be similar to the behavior of sAA and sNGMRES with
weights that are, based on $q'(x^*)$, globally optimal in obtaining the best asymptotic convergence rate.
In the numerical results at the end of the paper
we investigate this.

In a second approach for quantifying the asymptotic convergence behavior of AA and NGMRES,
we investigate optimal convergence for infinite window size.
We apply GMRES with $m=\infty$ to fixed-point equation (\ref{eq:fixed-point}) linearized about $x^*$, and
use known techniques to obtain an asymptotic convergence factor bound.
We investigate numerically whether this convergence factor for the linear case may also be relevant
for the nonlinear AA($\infty$) and NGMRES($\infty$) iterations as $x_k \rightarrow x^*$.

The rest of this paper is organized as follows. Section \ref{sec:back} provides background on the tensor approximation
problem we use as a case study in our paper, and on convergence of stationary iterative methods.
Sections \ref{sec:convergence-Anderson} and \ref{sec:convergence-NGMRES} derive optimal asymptotic convergence factors for the stationary sAA($m$) and sNGMRES($m$) iterations with optimal coefficients, based on the spectrum of $q'(x^*)$.
Section \ref{sec:inf-bounds} discusses asymptotic convergence factor estimates that are derived
from applying GMRES to the fixed-point iteration linearized about $x^*$.
Section \ref{sec:num-sec} provides numerical tests to illustrate how the asymptotic convergence acceleration provided by
AA and NGMRES is determined by matrix properties of $q'(x^*)$. Section \ref{sec:concl} formulates conclusions.

\section{Background}\label{sec:back}

\subsection{Canonical tensor decomposition}
In this paper we consider the problem of canonical tensor decomposition: we solve the following nonconvex optimization problem to fit an $N$-mode data tensor $\mathcal{Z}$ with a rank-$r$ tensor in the Frobenius norm,
\begin{equation}\label{eq:tensorf}
  \min f(A^{(1)},A^{(2)},\cdots, A^{(N)}) := \frac{1}{2}\Big\|\mathcal{Z}-[[A^{(1)},A^{(2)},\cdots, A^{(N)}]]\Big\|,
\end{equation}
where
\begin{equation}\label{eq:can}
  [[A^{(1)},A^{(2)},\cdots, A^{(N)}]] = \sum_{j=1}^{r} a^{(1)}_j\circ a^{(2)}_j \circ \cdots a^{(N)}_j.
\end{equation}
Here, $\circ$ denotes the vector outer product, and $a^{(m)}_j$ are the columns of
factor matrices $A^{(m)} \in \mathbb{R}^{n_m \times r}$, for $j = 1,\ldots,r$, $m=1,\ldots,N$.
We consider two fixed-point methods of form (\ref{eq:fixed-point}) that will be accelerated by sAA and sNGMRES: SD and ALS.

For SD with constant step length $\alpha$, we have
\begin{equation}\label{eq:SD-form}
  x_{k+1} = q_{SD}(x_k) = x_k -\alpha \nabla f(x_k).
\end{equation}
Furthermore,
\begin{equation}\label{eq:SD-jac}
q'_{SD}(x) =I -\alpha H(x).
\end{equation}

In each iteration, ALS sequentially updates a block of variables at a time, by minimizing expression (\ref{eq:tensorf}) while keeping the other blocks fixed.
Updating a factor matrix $A^{(i)}$ is a linear least-squares problem, see
\cite{kolda2009tensor,acar2011scalable,mitchell2020nesterov}.
The ALS fixed-point function is denoted by $q_{ALS}(x)$. As discussed in \cite[{\rm Lemma }3.2]{Uschmajew2012} or \cite{Ortega2000},
\begin{equation}\label{eq:GS-iter}
  q'_{ALS}(x^*) = I -M^{-1}(x^*)H(x^*),
\end{equation}
where $M$ is the lower block triangular part of $H(x)$ (including the block diagonal).
The derivation of $H(x)$ can be found in \cite{acar2009optimizationHess}.
Convergence of ALS is proved in \cite{Uschmajew2012}.

Comparing with $q'(x)=I-P\,A$ for the fixed-point function of linear preconditioned GMRES in (\ref{eq:qprec}),
we see that accelerating $q_{ALS}(x)$ with AA or NGMRES is indeed the nonlinear equivalent of using a block Gauss-Seidel
type preconditioner for GMRES, where the Hessian $H(x)$ plays the role of the non-preconditioned matrix $A$,
and $M^{-1}(x)$ plays a role similar to the linear Gauss-Seidel preconditioning matrix $P$ with $P=L^{-1}$ and $L$ being the lower triangular part of $A$.
\vlong{Similarly, using $q_{SD}(x)$ as the nonlinear preconditioner for AA or NGMRES is the nonlinear equivalent of using non-preconditioned GMRES, with $P=\alpha I$, see also \cite{sterck2013steepest}.}

Due to the scaling indeterminacy in the rank-$r$ tensor $[[A^{(1)},A^{(2)},\cdots, A^{(N)}]]$,
$H(x^{*})$ has at least $(N-1)r$ zero eigenvalues \cite{Uschmajew2012}.
Thus, we need to modify the definition of condition number of a matrix in our discussion.
Assume that $H(x^*)$ is positive semi-definite. We define the (modified) condition number of $H(x^*)$ as
\begin{equation}\label{eq:modified-condition-number}
  \bar{\kappa}=\frac{\lambda_{\max}}{\lambda_{\min}}=\frac{L}{\ell},
\end{equation}
where $L= \lambda_{\max}$ is the largest eigenvalue of $H(x^{*})$ and $\ell=\lambda_{\min}$ is the smallest nonzero eigenvalue of $H(x^{*})$. For simplicity, we assume that $L>\ell$ in what follows.

\subsection{Asymptotic convergence of stationary acceleration methods}
In Sections \ref{sec:convergence-Anderson} and \ref{sec:convergence-NGMRES} we will rely on asymptotic
convergence results for stationary iterations \cref{eq:sAA,eq:sNGMRES}.
Consider sNGMRES($m$) of \cref{eq:sNGMRES} written in system form
 \begin{equation}\label{eq:vecit}
   \boldsymbol{y}_{k+1}= \Psi(\boldsymbol{y}_{k}),
 \end{equation}
where $\boldsymbol{y}_k =\begin{pmatrix} x_k & x_{k-1} &\cdots & x_{k-m} \end{pmatrix}^T$ and
\begin{equation} \label{equ:matrix-NGMRES}
  \Psi(\boldsymbol{y}_k) = \begin{pmatrix}   (1+\sum_{i=0}^{m}\beta_{i})q(x_k) - \sum_{i=0}^{m}\beta_{i} x_{k-i}\\
     x_k\\
     \vdots\\
     x_{k-m+1}
  \end{pmatrix}.
\end{equation}
We state a convergence result form \cite{Varga1986}:
\begin{theorem}\label{thm:conv}
If the ($m+1$)-step stationary iterative method \cref{equ:matrix-NGMRES} converges
locally near fixed point $x^{*}$, then the linear root-convergence factor is equal to the spectral radius $\rho(T(q'(x^*)))$
of the $(m+1)\times (m+1)$ block matrix
$T(q'(x^*)):=\Psi'(\boldsymbol{y}^*)$, with $\boldsymbol{y}^*=\begin{pmatrix} x^* & x^* &\cdots & x^* \end{pmatrix}^T$,
where
\begin{equation}\label{eq:TsNGMRES}
  T(q'(x^*)) = \begin{bmatrix}  (1+\sum_{i=0}^{m}\beta_{i})q'(x^*)-\beta_0 I &-\beta_{1}I  &\cdots   & -\beta_{m-1}I &-\beta_{m}I\\
  I&   0&   &    0&  0\\
  0&   I&   &    0&  0\\
  \vdots & \vdots &  &\vdots &\vdots\\
  0&  0&  \cdots&   I& 0
  \end{bmatrix}.
\end{equation}
\end{theorem}
\vlong{The convergence result from \cite{Varga1986} shows that the asymptotic convergence factor of the stationary
iteration \cref{eq:sNGMRES} for a nonlinear fixed-point function $q(x)$ equals the asymptotic convergence factor
of the stationary linear iteration obtained by applying \cref{eq:sNGMRES} to the linearized fixed-point function,
where the linearization is done about the fixed point.}
An equivalent result follows easily for sAA($m$), Eq.\ \cref{eq:sAA}.

One technical complication with the previous result is the following.
For $q_{SD}(x)$ and $q_{ALS}(x)$, \cref{eq:SD-form} and \cref{eq:GS-iter}
show that $T$ is a function of the Hessian. If the Hessian has zero eigenvalues at $x^*$, as in our canonical tensor decomposition
problem, then $q'(x^*)$ has eigenvalues 1, and, as a consequence $T$ in \cref{eq:TsNGMRES} also has
eigenvalues 1.
In that case, we denote $\rho(T)$ as the second largest modulus of eigenvalues of $T$ (but if the Hessian is positive definite, $\rho(T)$ stands for the standard spectral radius of $T$), and similar for $\rho(q')$. Also, for a given matrix $B$, $\sigma(B)$ denotes the spectrum of $B$.

Note also that a simple application of \cref{thm:conv}, with $m=0$, $q(x)=x-\nabla f(x)$, and $\beta_0=\alpha-1$,
can be used to show that, for the steepest descent method, the optimal asymptotic convergence factor is given by
\begin{equation}
  \rho_{SD} = \frac{\bar{\kappa}-1}{\bar{\kappa}+1},
\end{equation}
with optimal step length $\alpha=\frac{2}{L+\ell}$;
see also \cite{Luenberger}.

\section{Optimal asymptotic convergence factors for stationary Anderson and Nesterov acceleration}\label{sec:convergence-Anderson}
In this section, we consider the theoretical problem of finding coefficient $\beta_0$ in sAA($m$) iteration (\ref{eq:sAA}) for
$m=1$ that results in the optimal convergence factor, assuming $q'(x^*)$ is known. We simplify notation and
consider the iteration
\begin{equation}\label{eq:anderson-1-step}
  x_{k+1} = (1+\beta) q(x_k) -\beta q(x_{k-1}).
\end{equation}
We will consider \yh{two cases: fixed-point iterations where the Jacobian $q'(x^*)$ has real spectrum (e.g., $q_{SD}(x)$) or
complex spectrum (e.g., $q_{ALS}(x)$).}
Note that the sAA(1) iteration of \cref{eq:anderson-1-step} is also a stationary version of Nesterov acceleration
\cite{nesterov1983method,NesterovBook,mitchell2020nesterov}.

We can rewrite the above iteration as a system
\begin{equation}\label{system-form}
\boldsymbol{y}_{k+1}=\begin{pmatrix}
x_{k+1}\\x_k
\end{pmatrix} =
\begin{pmatrix}
(1+\beta)q(x_k) -\beta q(x_{k-1})\\
x_{k}
\end{pmatrix}
 =: \Psi(\boldsymbol{y}_k),
\end{equation}
and we find that
\begin{equation}\label{eq:matrix-step1}
\Psi'(x^*,x^*)=  T(q'(x^*))  = \begin{bmatrix}
(1+\beta)q'(x^*)& -\beta q'(x^*)\\
I  &  0
\end{bmatrix}.
\end{equation}
\yh{We will write $T(q'(x^*);\beta)$ when it is useful to emphasize the dependence of $T$ on $\beta$.}

Let $\mu \in \sigma(q')$, where we assume from now on that $q'$ is evaluated in $x^*$.
It can be shown easily that the eigenvalues $\lambda$ of $T$ in (\ref{eq:matrix-step1}) satisfy
\begin{equation}\label{eq:AA-eig-form}
  \lambda^2-(1+\beta)\mu \lambda+ \beta \mu=0.
\end{equation}
Then, the two roots of \cref{eq:AA-eig-form}  are given by
\begin{equation}
 \yh{\lambda_{1,2}} =\frac{(1+\beta)\mu\pm \sqrt{(1+\beta)^2\mu^2-4\beta\mu}}{2}.\label{eq:g1-postive-mu}
\end{equation}

For any given $\mu$, we define the set
\begin{equation*}
 \yh{ \mathcal{S}_{\mu}(\beta) =\Big\{ |\lambda_1|, |\lambda_2| \Big\}.  }
\end{equation*}

\subsection{Optimal asymptotic convergence factor \yh{of sAA(1) applied to fixed-point methods with real Jacobian spectrum}}
\yh{We first consider fixed-point methods with Jacobians $q'(x^*)$ that have real spectrum.
For example, if $q$ is the steepest-descent fixed-point function of \cref{eq:SD-form}, it is obvious that $\mu \in \mathbb{R}$.}

Since the eigenvalues of $q'$ will affect the eigenvalues $\lambda$ of $T$ , it is useful to know how $|\lambda|$  changes for  any $\mu \in \sigma(q')$, and what the optimal result is of  $\min_{\beta}\max\mathcal{S}_{\mu}(\beta)$ for a given $\mu$.
We first consider the nonnegative case.
\begin{lemma}\label{minmax-real-nonnegative-case}
1. Assume $0<\mu<1$. Then
\begin{equation}\label{eq:lemma-result}
  \min_{\beta\in\mathbb{R}}\max \mathcal{S}_{\mu}(\beta) =  1-\sqrt{1-\mu},
\end{equation}
\yh{where the unique optimal $\beta$ is given by
$\displaystyle \beta_{\rm opt}(\mu)=\frac{1-\sqrt{1-\mu}}{1+\sqrt{1-\mu}}.$
Moreover, given any $\mu_1$ and $\mu_2$ such that $0<\mu_1<\mu_2<1$, we have}
\begin{equation*}
\max \mathcal{S}_{\mu_1}\big(\beta_{\rm{opt}}(\mu_2)\big)< \min_{\beta \in\mathbb{R}}\max\mathcal{S}_{\mu_2}(\beta).
\end{equation*}
2. Assume $1\leq \mu$. Then
$\displaystyle    \min_{\beta\in\mathbb{R}}\max \mathcal{S}_{\mu}(\beta) =  \sqrt{\mu}$,  \yh{ where the unique optimal $\beta$ is $\beta_{\rm opt}=-1$.}
\end{lemma}
\begin{proof}
We first consider $0<\mu<1$. Denote $\Delta =(1+\beta)^2\mu^2-4\beta\mu=\mu^2\big((1+\beta)^2-4\beta/\mu\big)$ with $\mu\neq 0$. Note that $\lambda$ in \cref{eq:g1-postive-mu} might be a real or complex number. Thus, we consider the following two cases. \\
{\bf Complex eigenvalues}: If $(1+\beta)^2-4\beta/\mu\leq 0$, then $\Delta \leq 0$. Moreover, $\mathcal{S}_{\mu}=\{\sqrt{\beta \mu}\}$. It follows that
\begin{equation}\label{eq:complex-minimization-form}
  \min_{\beta \in \mathbb{R}} \max\mathcal{S}_{\mu}(\beta)=\min_{\beta \in \mathbb{R}}\sqrt{\beta \mu},
\end{equation}
when
\begin{equation}\label{eq:constrain-beta-mu}
 d(\beta):=\beta^2+ (2-\frac{4}{\mu})\beta +1\leq 0.
\end{equation}
The two roots of $d(\beta)=0$ are
\begin{equation}\label{eq:complex-domain-ends}
 \beta_1=\frac{1-\sqrt{1-\mu}}{1+\sqrt{1-\mu}},\,\, \beta_2=\frac{1+\sqrt{1-\mu}}{1-\sqrt{1-\mu}}.
\end{equation}
\yh{Note that $\beta_1$ and $\beta_2$ are functions of $\mu$. When we want to emphasize the dependence on the variable  $\mu$, we will write $\beta_1(\mu)$ and $\beta_2(\mu)$.} Moreover, we rewrite \cref{eq:complex-minimization-form} and \cref{eq:constrain-beta-mu} as
$
\displaystyle   \min_{\beta\in[\beta_1,\beta_2]} \sqrt{\beta \mu} =\sqrt{\beta_1\mu} = 1-\sqrt{1-\mu}.
$\\
{\bf Real eigenvalues}: When $\beta <\beta_1$ or $\beta>\beta_2$, the eigenvalues $\lambda$ are real. We claim that $\max \mathcal{S}_{\mu}(\beta)$ is decreasing over $(-\infty, \beta_1]$ and increasing over $[\beta_2,\infty)$.
Note that for  $-1<\beta <\beta_1$ or $\beta>\beta_2$,
\begin{equation*}
  \max \mathcal{S}_{\mu} =\frac{(1+\beta)\mu + \sqrt{(1+\beta)^2\mu^2-4\beta\mu}}{2}=:g_1(\beta),
\end{equation*}
and
$\displaystyle g_1(\beta) =\frac{(1+\beta)\mu + \mu \sqrt{\beta^2+(2-\frac{4}{\mu})\beta+1}}{2} =\frac{(1+\beta)\mu + \mu\sqrt{d(\beta)}}{2}$.

When $\beta> \beta_2$, $d(\beta)$ is increasing. It follows that $g_1(\beta)$ is increasing over $[\beta_2,\infty)$. We claim that $g_1(\beta)$ is decreasing when $\beta \in [-1,\beta_1)$.  In fact,
\begin{eqnarray*}
  g'_1(\beta) = \frac{\mu}{2}\Big(1+\frac{1+\beta-\frac{2}{\mu}}{\sqrt{(1+\beta)^2-\frac{4\beta}{\mu}}}\Big)
   &=&\frac{\mu}{2}\frac{\sqrt{(1+\beta-\frac{2}{\mu})^2+\frac{4}{\mu}(1-\frac{1}{\mu})}+1+\beta-\frac{2}{\mu}}{\sqrt{(1+\beta)^2-\frac{4\beta}{\mu}}}  \\
   &<&\frac{\mu}{2} \frac{\sqrt{(1+\beta-\frac{2}{\mu})^2}+1+\beta-\frac{2}{\mu}}{{\sqrt{(1+\beta)^2-\frac{4\beta}{\mu}}}} = 0,
\end{eqnarray*}
where the last equality is due to $-1<\beta<1$ and $0<\mu<1$.

For $\beta<-1$,
$\displaystyle   \max \mathcal{S}_{\mu}(\beta) =\frac{-(1+\beta)\mu + \sqrt{(1+\beta)^2\mu^2-4\beta\mu}}{2}$.
It is clear that $\max \mathcal{S}_{\mu}(\beta)$ is decreasing  over $(-\infty, -1)$.

Combining the above two cases, we know that $|g_1(\beta)|$ is decreasing over $(-\infty,\beta_1]$ and increasing over $[\beta_1,\infty)$. Thus, $\displaystyle \min_{\beta\in\mathbb{R}}\max \mathcal{S}_{\mu}(\beta)=1-\sqrt{1-\mu}$ if and only if $\beta=\beta_1$.

We now prove the second statement. From \cref{eq:complex-domain-ends} and \yh{the fact that $\beta_1(\mu)\beta_2(\mu)=1$, we know that for any given $\mu_1$ and $\mu_2$ such that $0<\mu_1<\mu_2<1$,}
\begin{equation*}
  \beta_1(\mu_1) < \beta_1(\mu_2) <\beta_2(\mu_2)<\beta_2(\mu_1).
\end{equation*}
It follows that for $\mu_1$, when $\beta\in [\beta_1(\mu_2),\beta_2(\mu_2)]$, the corresponding $\lambda$ in \cref{eq:AA-eig-form} is a complex number. Thus, \begin{equation*}
\max \mathcal{S}_{\mu_1}\big(\beta_1(\mu_2)\big) =\sqrt{\beta_1(\mu_2)\mu_1} <\sqrt{\beta_1(\mu_2)\mu_2}=\min_{\beta}\max \mathcal{S}_{\mu_2}(\beta),
\end{equation*}
which is the desired result.

Now we consider $\mu\geq 1$. Recall $\Delta =(1+\beta)^2\mu^2-4\beta\mu=\mu^2\big((1+\beta)^2-4\beta/\mu\big)$. We claim that $\Delta\geq 0$. For $\beta<0$, this is obvious. When $\beta\geq 0$, $\Delta = \mu^2\big((1-\beta)^2+4\beta(1-\frac{1}{\mu})\big)\geq0$. \yh{This means that the roots of \cref{eq:AA-eig-form} are real.}
When $\beta<-1$,
\begin{equation*}
  \max \mathcal{S}_{\mu}(\beta) =\frac{-(1+\beta)\mu + \sqrt{(1+\beta)^2\mu^2-4\beta\mu}}{2}.
\end{equation*}
It is clear that $\max \mathcal{S}_{\mu}(\beta)$ is decreasing  over $(-\infty, -1)$.

When $\beta\geq -1$,
$\displaystyle   \max \mathcal{S}_{\mu}(\beta) =g_1(\beta)$,
and we know
\begin{eqnarray*}
  g'_1(\beta) = \frac{\mu}{2}\Big(1+\frac{1+\beta-\frac{2}{\mu}}{\sqrt{(1+\beta)^2-\frac{4\beta}{\mu}}}\Big)
   &=&\frac{\mu}{2}\frac{\sqrt{(1+\beta-\frac{2}{\mu})^2+\frac{4}{\mu}(1-\frac{1}{\mu})}+1+\beta-\frac{2}{\mu}}{\sqrt{(1+\beta)^2-\frac{4\beta}{\mu}}}\\
   &\geq&\frac{\mu}{2}\frac{\sqrt{(1+\beta-\frac{2}{\mu})^2}+1+\beta-\frac{2}{\mu}}{\sqrt{(1+\beta)^2-\frac{4\beta}{\mu}}}   \geq 0.
\end{eqnarray*}
This means that $g_1(\beta)$ is increasing over $[-1,\infty)$. Thus,
\begin{equation*}
\max_{\beta\in \mathbb{R}} \mathcal{S}_{\mu}(\beta) =\mathcal{S}_{\mu}(\beta=-1)=\sqrt{\mu}.
\end{equation*}
\end{proof}

For negative eigenvalues $\mu$ of  $q'$ we have a similar result as presented in \cref{minmax-real-nonnegative-case}.
\begin{lemma}\label{minmax-real-negative-case}
Assume $\mu<0$. Then
\begin{equation}\label{eq:lemma-result-negative}
  \min_{\beta \in \mathbb{R}}\max \mathcal{S}_{\mu}(\beta) =  \sqrt{1-\mu}-1,
\end{equation}
\yh{where the unique optimal $\beta$ is given by
$\displaystyle \beta_{\rm opt}(\mu)= \frac{1-\sqrt{1-\mu}}{1+\sqrt{1-\mu}}<0.$
Moreover, given any $\mu_1$ and $\mu_2$ such that  $\mu_2<\mu_1<0$, we have}
\begin{equation}\label{eq:AA-SD-negative}
  \min_{\beta \in \mathbb{R}}\max \mathcal{S}_{\mu_1}(\beta)<\min_{\beta \in \mathbb{R}}\max \mathcal{S}_{\mu_2}(\beta),
  \quad \textrm{and} \quad
  \max \mathcal{S}_{\mu_1}\big(\beta_{\rm{opt}}(\mu_2)\big)< \min_{\beta\in \mathbb{R}}\max \mathcal{S}_{\mu_2}(\beta).
\end{equation}
\end{lemma}
\begin{proof}
The proof is similar to case 1 of \cref{minmax-real-nonnegative-case}. In fact, when $\mu<0$, $\lambda$ is complex for $\beta\in(\beta_2,\beta_1)$ and $\lambda$ is real for $\beta\in(-\infty,\beta_2)$ and $(\beta_1,\infty)$. Moreover, $\max\mathcal{S}_{\mu}(\beta)$ is decreasing over $(-\infty, \beta_1)$ and increasing over $(\beta_{1},\infty)$.  Thus $\min_{\beta \in \mathbb{R}}\max \mathcal{S}_{\mu}(\beta)$ is obtained at $\beta=\beta_1$.
\end{proof}

Let $\rho_{q'}=\rho(q')$ (where, as before, $q'$ is always evaluated at $x^*$, and the spectral radius excludes the eigenvalues 1 that result from the degeneracy of the Hessian at $x^*$ \cite{Uschmajew2012}).
In the following, we assume that $q$ is a convergent operator, that is, $\rho_{q'}<1$. It follows that $\alpha>0$.
If one wants to minimize the spectral radius of $T$ and $q'$ has both positive and negative eigenvalues, \cref{minmax-real-nonnegative-case,minmax-real-negative-case} can be combined, by considering the largest and smallest eigenvalues of $q'$. This will be done in \cref{Thm:positive-negative-AA-SD}.
In the case of SD, however, when the step length $\alpha \in(0,\frac{1}{L}]$, all the eigenvalues of $q'$ are nonnegative, and based on just
\cref{minmax-real-nonnegative-case} we can obtain a known result for sAA(1)-SD, i.e., the stationary version of Nesterov's method, as follows.
\begin{theorem}\label{thm:psd-SD-case}
\yh{Let $x^*$ be a fixed point of iteration (\ref{eq:fixed-point}).}
For any given step length $\alpha \in(0,\frac{1}{L}]$ in SD, we denote the spectral radius of $q'=I-\alpha H$ in \cref{eq:SD-form} as $\rho_{q'}=1-\alpha \ell$.
Then, for the sAA(1)-SD method \yh{(\ref{system-form}) with Jacobian matrix $T$ defined in (\ref{eq:matrix-step1}),} the optimal asymptotic convergence factor is given  by
\begin{equation}\label{eq:cs-SD-bound}
  \min_{\beta\in\mathbb{R}}\rho\big(T(x^{*};\beta)\big) =  1-\sqrt{1-\rho_{q'}} < (\rho_{q'})^{1.75},
\end{equation}
\yh{where the unique optimal $\beta$ is given by
$\displaystyle \beta_{\rm opt}= \frac{1-\sqrt{1-\rho_{q'}}}{1+\sqrt{1-\rho_{q'}}}.$}
Moreover, the best result for $\alpha\in(0,\frac{1}{L}]$ is achieved at $\alpha=\frac{1}{L}$, and
\begin{equation*}
 1-\sqrt{1-\rho_{q'}} = 1-\sqrt{\frac{\ell}{L}}=:\rho^{(1/L)}_{\rm sAA(1)-SD}, \quad \textrm{with} \quad \beta=\frac{1-\sqrt{\ell/L}}{1+\sqrt{\ell/L}}.
\end{equation*}
\end{theorem}
\begin{proof}
Note that when $\alpha \in (0,\frac{1}{L}]$, all the eigenvalues of $q'$ are nonnegative. According to  \cref{minmax-real-nonnegative-case}, the convergence factor of \cref{system-form} is determined by the largest eigenvalue of $q'$, which equals the spectral radius of $q'$.
\end{proof}

Note that \cite{NesterovBook,odonoghue2015adaptive,scieur2016regularized} have studied the choice of $\alpha=\frac{1}{L}$, with results consistent with the special case in \cref{thm:psd-SD-case}. The choice $\alpha=\frac{1}{L}$ is the best choice in $(0,\frac{1}{L}]$, but \cref{Thm:positive-negative-AA-SD} shows that a better convergence factor can be obtained when $\alpha$ is chosen optimally in $(\frac{1}{L},\infty]$.
\begin{theorem}\label{Thm:positive-negative-AA-SD}
\yh{Let $x^*$ be a fixed point of iteration (\ref{eq:fixed-point}).}
For the sAA(1) method \yh{(\ref{system-form}) with Jacobian matrix $T$ defined in (\ref{eq:matrix-step1}),} applied to
$q_{SD}$ with $\alpha>0$, $\beta\in\mathbb{R}$, the optimal asymptotic convergence factor is given by
\begin{equation}\label{eq:improved-SD-opt-cs}
  \rho^{*}_{\rm sAA(1)-SD}=\min_{\alpha>0,\beta\in\mathbb{R}}\rho(T(x^{*};\beta)) = \frac{\sqrt{3\bar{\kappa}+1}-2}{\sqrt{3\bar{\kappa}+1}}= 1-\sqrt{\frac{4\ell}{3L+\ell}}<1-\sqrt{\frac{\ell}{L}},
\end{equation}
\yh{where the unique optimal $\alpha$ and $\beta$ are given by}
\begin{equation}\label{eq:SD-opt-alpha-beta}
\alpha^* =\frac{4}{3L+\ell}, \quad  \beta^{*}=\frac{1-\sqrt{\alpha^* \ell}}{1+\sqrt{\alpha^* \ell}}=\frac{\sqrt{3\bar{\kappa}+1}-2}{\sqrt{3\bar{\kappa}+1}+2}.
\end{equation}
\end{theorem}
\begin{proof}
From  \cref{minmax-real-nonnegative-case,minmax-real-negative-case}, we only need to consider the extreme eigenvalues of $q'$ to minimize the spectral radius of $T$. Based on positive or negative eigenvalues of $q'$ that determine the spectral radius of $q'$, we divide the discussion into four cases.

{\bf Case 1}: $\alpha \in(0, \frac{1}{L}]$. According to  \cref{thm:psd-SD-case}, the optimal convergence factor is $1-\sqrt{\frac{\ell}{L}}$, for the step length choice $\alpha=\frac{1}{L}$.

{\bf Case 2}: $\alpha \in[\frac{1}{\ell},\infty)$. Note that all eigenvalues of $q'$ are nonpositive. Based on \cref{eq:AA-SD-negative}, we need to choose $\alpha$ to maximize $1-\alpha L$. It follows  $\alpha=\frac{1}{\ell}$. Then, according to \cref{eq:lemma-result-negative} in  \cref{minmax-real-negative-case}, the optimal convergence factor is achieved at $\beta=\frac{1-\sqrt{1-\mu}}{1+\sqrt{1-\mu}}$, where $\mu=1-\frac{L}{\ell}$, given by $\sqrt{1-(1-\frac{L}{\ell})}-1=\sqrt{\frac{L}{\ell}}-1$, which is larger than $1-\sqrt{\frac{\ell}{L}}$ given in case 1.

{\bf Case 3}: $\alpha\in [\frac{1}{L},\frac{2}{L+\ell}]$. Let $\mu_+=1-\alpha \ell>0$ and $\mu_-=1-\alpha L\leq0$. Then, $\sigma(q')\subseteq[\mu_-,\mu_+]$. Moreover, $|1-\alpha L|=\alpha L-1 <1-\alpha \ell$.

From \cref{minmax-real-nonnegative-case,minmax-real-negative-case}, we know
\begin{equation*}
\min_{\beta\in\mathbb{R}} \max\mathcal{S}_{\mu_+}(\beta)=1-\sqrt{1-(1-\alpha \ell)}=1-\sqrt{\alpha \ell}>\min_{\beta\in\mathbb{R}} \max \mathcal{S}_{\mu_-}(\beta)=\sqrt{\alpha L}-1.
\end{equation*}
Here, the optimal bound for $\mathcal{S}_{\mu_+}$, $1-\sqrt{\alpha \ell}$, is obtained for
\begin{equation}\label{eq:beta-opt}
\beta(\alpha)=\frac{1-\sqrt{\alpha \ell}}{1+\sqrt{\alpha \ell}}.
\end{equation}
Minimizing $\displaystyle\max(\mathcal{S}_{\mu_-}\cup\mathcal{S}_{\mu_+})$ over $\beta$
then requires us to choose $\alpha\in [\frac{1}{L},\frac{2}{L+\ell}]$ that minimizes
$1-\sqrt{\alpha \ell}$ using $\beta(\alpha)$ as in (\ref{eq:beta-opt}), but making sure
not to exceed $\max \mathcal{S}_{\mu_-}(\beta(\alpha))$.
Therefore, we seek $\alpha$ such that
\begin{equation}\label{eq:-negative-positive-same-rate}
  1-\sqrt{\alpha \ell} =\max \mathcal{S}_{\mu_-}(\beta(\alpha)).
\end{equation}
By an easy calculation, the right-hand side of \cref{eq:-negative-positive-same-rate} can be written as
\begin{equation*}
  \max \mathcal{S}_{\mu_-}(\beta(\alpha)) =\frac{-(1+\beta(\alpha))\mu_{-} + \sqrt{(1+\beta(\alpha))^2\mu_{-}^2-4\beta\mu_{-}}}{2}.
\end{equation*}

Simplifying \cref{eq:-negative-positive-same-rate} leads to
$ (3\ell L+\ell^2)\alpha^2-(3L+5\ell)\alpha+4=0$,
whose two roots are
$\displaystyle
  \alpha_1=\frac{4}{3L+\ell},  \,\, \alpha_2=\frac{1}{\ell}$.
It is obvious that $\displaystyle \frac{1}{L}<\alpha_1<\frac{2}{L+\ell}$. $\alpha_2>\frac{2}{L+\ell}$. Thus the optimal parameters are
$\displaystyle
\alpha^*=\frac{4}{3L+\ell}, \,\,\beta^{*}=\frac{1-\sqrt{\alpha^{*}\ell}}{1+\sqrt{\alpha^{*}\ell}},
$
and
\begin{equation*}
  \min_{\alpha \in [\frac{1}{L},\frac{2}{L+\ell}],\beta\in\mathbb{R}}\max \mathcal{S}_{\mu_-}(\alpha,\beta)=1-\sqrt{\alpha^*\ell}=1-\sqrt{\frac{4\ell}{3L+\ell}}<1-\sqrt{\frac{\ell}{L}}.
\end{equation*}
Note that
\begin{eqnarray*}
  1-\sqrt{\frac{4\ell}{3L+\ell}} &=& \frac{1-4\ell/(3L+\ell)}{1+\sqrt{4\ell/(3L+\ell)}}
   = \frac{3(\bar{\kappa}-1)}{3\bar{\kappa}+1+2\sqrt{3\bar{\kappa}+1}}
   =\frac{\sqrt{3\bar{\kappa}+1}-2}{\sqrt{3\bar{\kappa}+1}},
\end{eqnarray*}
and
\begin{eqnarray*}
  \beta^{*}=\frac{1-\sqrt{\alpha^{*}\ell}}{1+\sqrt{\alpha^{*}\ell}} &=& \frac{1-\sqrt{4\ell/(3L+\ell)}}{1+\sqrt{4\ell/(3L+\ell)}}
   = \frac{\sqrt{3\bar{\kappa}+1}-2}{\sqrt{3\bar{\kappa}+1}+2}.
\end{eqnarray*}
{\bf Case 4}: $\alpha \in \big[\frac{2}{L+\ell},\frac{1}{\ell}\big)$.  Let $\eta_+ =1-\alpha \ell>0$ and $\eta_{-}=1-\alpha L<0$. Note that $\sigma(q')\subset [\eta_-,\eta_+]$ and $\rho_{q'}=-\eta_-=\alpha L-1$.  Recall that $\max \mathcal{S}_{\eta_-}\big(\beta_{\rm opt}(\eta_-)\big)$ is increasing over $\big[\beta_{\rm opt}(\eta_-),0\big)$ and $\mathcal{S}_{\eta_+}(\beta)$ is decreasing over $[\beta_{\rm opt}(\eta_-),0)$ with
\begin{equation}
\max \mathcal{S}_{\eta_+}(\beta) =\frac{(1+\beta)\eta_{+} +\sqrt{(1+\beta)^2\eta^2_{+}-4\beta\eta_{+}}}{2},
\end{equation}
and $\max \mathcal{S}_{\eta_-}(\beta)>\max \mathcal{S}_{\eta_+}(\beta).$
Note that when $\alpha \in[\frac{4}{L+\ell+2\sqrt{\ell L}},\frac{1}{\ell}]=:I_1$,
\begin{equation*}
 \min_{\beta} \max \mathcal{S}_{\eta_-}(\beta)> \min_{\beta} \max \mathcal{S}_{\eta_+}(\beta).
 \end{equation*}
Minimizing  $\displaystyle\max(\mathcal{S}_{\eta_-}\cup\mathcal{S}_{\eta_+})$ is equivalent to finding $\alpha\in I_1$ such that
\begin{equation*}
  \min_{\beta} \max\mathcal{S}_{\eta_-}(\beta)=\sqrt{1-\eta_-}-1 =\max\mathcal{S}_{\eta_+}(\beta^*),
\end{equation*}
where $\beta^*=\frac{1-\sqrt{1-\eta_-}}{1+\sqrt{1-\eta_-}}=\frac{1-\sqrt{\alpha \ell}}{1+\sqrt{\alpha \ell}}$. An
easy calculation then shows that $\alpha$ satisfies
$(3\ell L+L^2)\alpha^2-(5L+3\ell)\alpha+4=0$,
with roots
$\displaystyle
  \alpha_1=\frac{4}{L+3\ell},  \,\, \alpha_2=\frac{1}{L}$.
It is obvious that $\frac{2}{L+\ell}<\alpha_1<\frac{1}{\ell}$ and $\alpha_1\in I_1$, and $\alpha_2<\frac{2}{L+\ell}$. Thus the optimal parameters are
\begin{equation*}
\alpha^*=\frac{4}{L+3\ell}, \,\,\beta^{*}=\frac{1-\sqrt{\alpha^{*}\ell}}{1+\sqrt{\alpha^{*}\ell}}.
\end{equation*}
Furthermore,
\begin{equation*}
\min_{\alpha \in I_1,\beta\in\mathbb{R}} \max\mathcal{S}_{\eta_-}(\beta)=\sqrt{1-\eta_-}-1=\sqrt{\alpha^* L}-1=\sqrt{\frac{4L}{L+3\ell}}-1>1-\sqrt{\frac{4\ell}{3L+\ell}}.
\end{equation*}
When $\alpha \in\big[\frac{2}{L+\ell},\frac{4}{L+\ell+2\sqrt{L\ell}}\big]$, although $\min_{\beta}\max \mathcal{S}_{\eta_-}(\beta)$ is less than $\sqrt{\alpha^* L}-1$, $\max \mathcal{S}_{\eta_+}(\beta)$ is decreasing with respect to both $\alpha$ and $\beta$. Thus, when $\alpha \in\big[\frac{2}{L+\ell},\frac{4}{L+\ell+2\sqrt{L\ell}}\big]$, $\displaystyle\min \max(\mathcal{S}_{\eta_-}\cup\mathcal{S}_{\eta_+})$ is larger than $\sqrt{\alpha^* L}-1$.

Combining the above four cases, we can conclude that the optimal convergence factor is achieved in case 3.
\end{proof}
\begin{remark}
Note that the optimal parameters of \cref{Thm:positive-negative-AA-SD} confirm the optimal parameters given
in \cite{lessard2016analysis} (without proof) for Nesterov's accelerated gradient descent. 
\end{remark}

\subsection{Lower bound on optimal asymptotic convergence factor of sAA(1) \yh{applied to fixed-point methods with complex Jacobian spectrum}}
\yh{We now consider fixed-point methods with complex Jacobian spectrum.
We assume $\rho_{q'}<1$ at $x^*$, which is true for  ALS , see \cite{Uschmajew2012}.
The complex eigenvalues make it more difficult to analyze the convergence factor of sAA(1).}
Thus, we only give a lower bound in \cref{thm:lower-bound-ALS} on the optimal asymptotic convergence factor under the
condition that $\rho_{q'}=\mu$ for some eigenvalue of $q'$.
Interestingly, for the case of ALS applied to canonical tensor decomposition, all our numerical tests (with randomized and real-world data) show that this condition is always satisfied and the lower bound is always achieved, so we formulate conjectures on this that may be provable based on the special structure of the canonical tensor decomposition Hessian (see \cite{acar2009optimizationHess}), but remain a topic of further research.
\begin{theorem}[lower bound for fixed-point methods with complex Jacobian spectrum]\label{thm:lower-bound-ALS}
 \yh{
 Let $x^*$ be a fixed point of iteration (\ref{eq:fixed-point}).
 Let the spectral radius of $q'(x^*)$ be $\rho_{q'}$. Assume that there exists a real eigenvalue $\mu$ of $q'(x^*)$ such that $\rho_{q'}=\mu$. Then the optimal asymptotic convergence factor of the sAA(1) method (\ref{system-form}) with Jacobian matrix $T$ defined in (\ref{eq:matrix-step1})
is bounded below by}
\begin{equation}\label{eq:lower-bound-ALS}
\yh{ \min_{\beta \in\mathbb{R}} \rho(T(x^{*};\beta))} \geq  1-\sqrt{1-\rho_{q'}}=:\rho_p,
\end{equation}
and if the equality holds, then \yh{the unique optimal $\beta$ is given by}
\begin{equation}\label{opt-beta-real}
  \beta_{\rm opt}=\frac{1-\sqrt{1-\rho_{q'}}}{1+\sqrt{1-\rho_{q'}}}.
\end{equation}

\begin{proof}
Since $\rho_{q'} \in \sigma(q')$,
\begin{equation*}
 \min_{\beta \in\mathbb{R}} \rho(T(x^{*};\beta))= \min_{\beta \in\mathbb{R}}\max\big(\cup_{\mu\in\sigma(q')} \mathcal{S}_{\mu}(\beta)\big)\geq \min_{\beta \in\mathbb{R}}\max\mathcal{S}_{\rho_{q'}}(\beta).
\end{equation*}
Based on $0<\rho_{q'}<1$ and  \cref{eq:lemma-result} in  \cref{minmax-real-nonnegative-case}, we have
\begin{equation*}
   \min_{\beta \in\mathbb{R}}\max\mathcal{S}_{\rho_{q'}}(\beta)=1-\sqrt{1-\rho_{q'}},
\end{equation*}
with $\beta$ given in \cref{minmax-real-nonnegative-case}.
This is the desired result.
\end{proof}

\end{theorem}

The numerical results in Section \ref{sec:num-sec} suggest the following conjectures \yh{for ALS applied to canonical tensor decomposition}:
\begin{conj}\label{conjec1-sAA-ALS}
For ALS applied to canonical tensor decomposition, there exists a real eigenvalue $\mu$ of $q'$ such that $\rho_{q'}=\mu$ (where $q'$ is
evaluated in a fixed point $x^*$, and the spectral radius excludes the eigenvalues 1 that result from the degeneracy of the Hessian
at $x^*$ \cite{Uschmajew2012}).
\end{conj}

\begin{conj}\label{conjec2-sAA-ALS}
For sAA(1)-ALS for canonical tensor decomposition,
\begin{equation*}
 \min_{\beta \in\mathbb{R}} \rho(T(x^{*};\beta))=1-\sqrt{1-\rho_{q'}},
\end{equation*}
\yh{where $x^*$ is the fixed point and the Jacobian matrix $T$ of the sAA(1) method (\ref{system-form}) is defined in (\ref{eq:matrix-step1}).}
\yh{The unique optimal $\beta$ is given by \cref{opt-beta-real}.}
\end{conj}

If the assumption that $\mu=\rho_{q'}$ in \cref{thm:lower-bound-ALS} does not hold, a weaker form of \cref{thm:lower-bound-ALS} may still hold, see Supplementary Materials (SM), Section \ref{subsec:weaker}.

\section{Optimal asymptotic convergence factors for stationary NGMRES}\label{sec:convergence-NGMRES}
We now consider the theoretical problem of finding the coefficients $\beta_0$ and
$\beta_1$ in sNGMRES($m$) iteration (\ref{eq:sNGMRES}) for
$m=1$ that result in the optimal asymptotic convergence factor, assuming $q'(x^*)$ is known.
As it turns out, the analysis for sNGMRES($m$) is simplified if one first considers a reduced version of the method,
where the first term in the sum, with coefficient $\beta_0$, is left out:
\begin{align}
  x_{k+1}&= q(x_k) + \sum_{i=1}^{m}\beta_{i}(q(x_k)-x_{k-i})  \qquad k=m, m+1,\ldots.
\label{eq:sNGMRES-R}
\end{align}
We call this reduced version sNGMRES-R($m$).
As explained in SM Section \ref{sub:NGMRES-NGMRES-R},
the optimal convergence factors for sNGMRES-R(1)-SD and sNGMRES(1)-SD are the
same, and the performance of optimally tuned sNGMRES($m$)-ALS cannot be worse than sNGMRES-R($m$)-ALS.

\subsection{Optimal asymptotic convergence factor of sNGMRES-R(1)
}
\label{subsec:optimal}
We first consider sNGMRES-R(1)-SD, that is
\begin{equation}\label{eq:M-NGMRES(1)}
  x_{k+1} = (1+\beta) q(x_k) -\beta x_{k-1},
\end{equation}
applied to $q_{SD}(x)$.
We rewrite the above iteration as a system
\begin{equation}\label{system-form-NGM}
\boldsymbol{y}_{k+1}=\begin{pmatrix}
x_{k+1}\\x_k
\end{pmatrix} =
\begin{pmatrix}
(1+\beta)q(x_k)  -\beta x_{k-1} \\
x_k
\end{pmatrix}
 =: \Psi_N(\boldsymbol{y}_k).
\end{equation}
Note that
\begin{equation}\label{eq:matrix-step1-NGM}
  \Psi_N'(x^*,x^*) =: T_N(x^*) =\begin{bmatrix}
(1+\beta)q'(x^*)& -\beta I\\
I  &  0
\end{bmatrix}.
\end{equation}

Optimal parameters for sNGMRES-R(1)-SD are determined as follows:
\begin{theorem}\label{thm:sNGMRE-R-SD}
For sNGMRES-R(1) acceleration of the steepest descent method as defined in \cref{system-form-NGM}, the optimal asymptotic convergence factor is given by
\begin{equation}\label{eq:improved-SD-opt-cs-NGM}
  \rho^{*}_{\rm sNGMRES-R(1)-SD}=\min_{\alpha>0,\beta\in\mathbb{R}}\rho(T_N(x^{*};\beta)) =\frac{\sqrt{\bar{\kappa}}-1}{\sqrt{\bar{\kappa}}+1} <1-\sqrt{\frac{4\ell}{3L+\ell}},
\end{equation}
\yh{where $x^*$ is the fixed point and the Jacobian matrix $T_N$ of the sNGMRES-R(1) method (\ref{system-form-NGM}) is defined in (\ref{eq:matrix-step1-NGM}).}
\yh{The unique optimal $\alpha$ and $\beta$ are given by }
\begin{equation}\label{eq:SD-opt-alpha-beta-NGM}
\alpha_N^* =\frac{2}{L+\ell}, \, \beta_N^{*}=\frac{1-\sqrt{1-\rho_{q'}^2}}{1+\sqrt{1-\rho_{q'}^2}}=\Big(\frac{\sqrt{\bar{\kappa}}-1}{\sqrt{\bar{\kappa}}+1}\Big)^2, \textrm{ where } \rho_{q'}=\frac{L-\ell}{L+\ell}.
\end{equation}
\end{theorem}

The proof is similar to the proof of \cref{Thm:positive-negative-AA-SD} and
can be found in SM Section \ref{subsec:proof4.1}, building on a Lemma
similar to \cref{minmax-real-nonnegative-case}, \yh{for the case of fixed-point methods with real Jacobian spectrum}.
Note that the result in \cref{thm:sNGMRE-R-SD} can be derived from
\cite{MR703121} or \cite{hong2018accelerating}, see SM Section \ref{subsec:proof4.1},
but our proof is different, and our result in \cref{minmax-real-nonnegative-case-NGM}
can be used to derive optimal sNGMRES-R(1) parameters for other fixed-point methods
than SD in the case the spectrum of $q'(x^*)$ is real.

\yh{Optimal bounds for sNGMRES-R(1) applied to fixed-point methods with complex Jacobian spectrum are discussed in SM Section \ref{subsec:sNGMRES-R(1)-ALS-bounds}.
We obtain a result similar to \cref{thm:lower-bound-ALS} for sAA(1), but our numerical results show that, in the case of ALS for canonical tensor decomposition,} there is no equivalent to \cref{conjec2-sAA-ALS}. However, further lower and upper bounds are stated in \cref{thm:lowe-upper-bound}.

\subsection{Summary of optimal asymptotic convergence factors for sAA(1) and sNGMRES-R(1) acceleration of SD}
\label{subsec:optfac}

\begin{table}
 \caption{Optimal asymptotic convergence factors $\rho^*$ for SD, sAA(1)-SD with step length $\alpha=\frac{1}{L}$, sAA(1)-SD with optimal $\alpha$ as in \cref{eq:SD-opt-alpha-beta}, and sNGMRES-R(1)-SD with optimal $\alpha$ as in
 \cref{eq:SD-opt-alpha-beta-NGM}.}
\centering
\begin{tabular}{|l|c|c|c|c|}
\hline
 method                 &$\alpha$      &$\beta$   &$\rho_{q'}$   & $\rho^*$    \\
\hline
SD                       &$\frac{2}{L+\ell}$    & -   &$\frac{\bar{\kappa}-1}{\bar{\kappa}+1}$        &$\frac{\bar{\kappa}-1}{\bar{\kappa}+1}\approx 1-\frac{2}{\bar{\kappa}}$ \\
\hline
 sAA(1)-SD with $\alpha=\frac{1}{L}$         &$\frac{1}{L}$          &$\frac{\sqrt{\bar{\kappa}}-1}{\sqrt{\bar{\kappa}+1}}$   &$\frac{\bar{\kappa} -1}{\bar{\kappa}}$        & $\frac{\sqrt{\bar{\kappa}}-1}{\sqrt{\bar{\kappa}}}= 1-\frac{1}{\sqrt{\bar{\kappa}}}$\\
\hline
sAA(1)-SD with optimal $\alpha$              &$\frac{4}{3L+\ell}$          &$\frac{\sqrt{3\bar{\kappa}+1}-2}{\sqrt{3\bar{\kappa}+1}+2}$       &$\frac{3(\bar{\kappa}-1)}{3\bar{\kappa}+1}$             &$\frac{\sqrt{3\bar{\kappa}+1}-2}{\sqrt{3\bar{\kappa}+1}}\approx 1-\frac{2}{\sqrt{3} \sqrt{\bar{\kappa}}}$  \\
\hline
 sNGMRES-R(1)-SD              &$\frac{2}{L+\ell}$     &$\Big(\frac{\sqrt{\bar{\kappa}}-1}{\sqrt{\bar{\kappa}}+1}\Big)^2$       &$\frac{\bar{\kappa}-1}{\bar{\kappa}+1}$            &$\frac{\sqrt{\bar{\kappa}}-1}{\sqrt{\bar{\kappa}}+1} \approx 1-\frac{2}{\sqrt{\bar{\kappa}}}$ \\
\hline
\end{tabular}\label{comparsion-ASD-vs-NGMRESSD}
\end{table}

We summarize the optimal asymptotic convergence for sAA(1)-SD and sNGMRES-R(1)-SD in  \cref{comparsion-ASD-vs-NGMRESSD}.
The approximations of $\rho^*$ for large $\bar{\kappa}$ in the last column of \cref{comparsion-ASD-vs-NGMRESSD} show that,
as is well-known, 1-step acceleration methods are most useful for ill-conditioned
problems (large $\bar{\kappa}$): for the error reduction to reach a relative tolerance $\tau$, SD requires $O(\bar{\kappa})$ iterations, and the three 1-step acceleration methods each require $O(\sqrt{\bar{\kappa}})$ iterations.
Furthermore, the optimal sAA(1)-SD with the optimal step length
as in \cref{eq:SD-opt-alpha-beta} requires approximately $\sqrt{3}/2\approx 86.7\%$ of the iterations of the optimal sAA(1)-SD with the standard step length $1/L$, and the optimal sNGMRES-R(1)-SD needs about half the number of iterations of the optimal sAA(1)-SD with step length $1/L$,
see also \cref{acce-ratio-plot}.
\vlong{Note that the Heavy Ball Method is another 1-step acceleration method for SD that has the same $\rho^*$ as
sNGMRES-R(1)-SD \cite{recht2010cs726}.}

\section{The $m=\infty$ case: linear asymptotic GMRES convergence bounds for estimating
AA($\infty$) and NGMRES($\infty$) convergence factors}\label{sec:inf-bounds}

In this section we discuss how applying GMRES convergence bounds to fixed-point problem (\ref{eq:fixed-point}) linearized about
$x^*$ may relate to linear asymptotic convergence factors for the nonlinear AA and NGMRES iterations with window size $m=\infty$.

Linearizing fixed-point function $q(x)$ about $x^*$ gives $q(x)\approx q(x^*)+q'(x^*) \, (x-x^*)$, and
using $x^*=q(x^*)$ one obtains from $x=q(x)$ the linearized fixed-point problem
\begin{equation}
  \left(I -q'(x^*) \right)x = \left(I -q'(x^*)\right)x^*.
\label{eq:fixed-lin}
\end{equation}

Consider applying GMRES to $Ax=b$ with $A \in \mathbb{R}^{n \times n}$.
When $0 \notin \textrm{FOV}(A)$, where FOV$(A)$ is the field of values or numerical range of $A$ \cite{trefethen2005spectra},
the following property holds:
\begin{theorem}{\cite{beckermann2005some}
}\label{thm:GMRES-conv}
Define $\nu(A)=min_{z \in \textrm{FOV}(A)} |z|$ (the distance of FOV($A$) to the origin), $r(A)=max_{z \in \textrm{FOV}(A)} |z|$ (the numerical radius of $A$), and
$$\cos{\beta}=\nu(A)/r(A).$$
If $0 \notin \textrm{FOV}(A)$, then
\begin{equation}
\frac{\|r_k\|}{\|r_0\|} \le c_\beta \, \rho_\beta^k<10 \, \rho_\beta^k \quad \textrm{for any } r_0,
\end{equation}
where $\rho_\beta=2\sin{(\beta/(4-2\beta/\pi))} < \sin{\beta}$ and $c_\beta=(2+2/\sqrt{3}) \, (2+\rho_\beta)$.
\end{theorem}

Given the linearized problem matrix $I -q'(x^*)$ from \cref{eq:fixed-lin}, its field of values
can be computed numerically, or a box that is a superset of the
field of values can easily be computed \cite{mees1979domains}.
Both approaches allow to compute a linear convergence factor $\rho_\beta$ in the bound
of \cref{thm:GMRES-conv}.
When a linear asymptotic convergence factor $\rho_\beta$ is in hand for linearized problem
(\ref{eq:fixed-lin}), it is reasonable to expect that this convergence factor may also
be relevant for the asymptotic convergence behavior of AA and NGMRES as
$x_k \rightarrow x^*$. We are not aware of a proof that this is indeed the case,
so we formulate the conjecture below on a local linear convergence bound for
AA($\infty$) and NGMRES($\infty$). Our numerical tests in Section \ref{sec:num-sec}
are consistent with this conjecture.

\begin{conj}\label{conjec:inf}
Consider GMRES($\infty$) applied to linearized fixed-point problem (\ref{eq:fixed-lin}) with fixed point $x^*$.
If the GMRES residuals satisfy
$$
\frac{\|r_k\|}{\|r_0\|}  \le c_1 \rho^k \quad \textrm{for any } r_0,
$$
then the nonlinear residuals of applying NGMRES($\infty$) and AA($\infty$) to the nonlinear fixed-point iteration (\ref{eq:fixed-point})
associated with (\ref{eq:fixed-lin}) satisfy
$$
\frac{\|r_k\|}{\|r_0\|}  \le c_2 \rho^k,
$$
provided $x_0$ is chosen such that the nonlinear methods converge to $x^*$, and $x_0$ is chosen sufficiently close to $x^*$.
\end{conj}

\section{Numerical experiments}\label{sec:num-sec}
In all our numerical experiments, we consider the problem of canonical tensor decomposition, see
(\ref{eq:tensorf}).
In our tests, we consider three-way cubic tensors of size 50, that is, $N=3$, $n_1=n_2=n_3 =50$.
We randomly generate data tensors $\mathcal{Z}$ following the procedures and parameters in \cite{acar2011scalable,sterck2012nonlinear,mitchell2020nesterov}.
The data tensors are composed by generating underlying rank-$r$ tensors in the format of
\cref{eq:can}, with $r=3$, to which noise is added.
We randomly generate the factor matrices $A^{(1)}, A^{(2)}$, and $A^{(3)}$ of the underlying rank-$r$ tensor
so that the collinearity of the factors in each mode is set to a particular value, $c$, see \cite{acar2011scalable,sterck2012nonlinear,mitchell2020nesterov}, \yh{given by
\begin{equation*}
c = \frac{a_i^{(m)T} a_j^{(m)}}{ \|a_i^{(m)}\|  \|a_j^{(m)}\|},
\end{equation*}
where $a^{(m)}_i$ and $a^{(m)}_j$ are the columns of factor matrices $A^{(m)} \in \mathbb{R}^{n_m \times r}$, for $i,j = 1,\ldots,r$, $m=1,\ldots,N$, see \cref{eq:tensorf,eq:can}.}
The goal is to recover these underlying factor matrices once assembled into the tensor and noise
has been added.  We set $l_1 = 1$, and $l_2 = 1$ to be the desired noise ratios
of homoscedastic and heteroscedastic noise, respectively \cite{acar2011scalable,sterck2012nonlinear}. \yh{The tensor $\mathcal{Z}$ is generated as follows. First generate an $r\times r$ matrix $K$
that has diagonal elements 1 and off-diagonal elements $c$, and compute the Cholesky
factor $C$ of $K$. Then generate 3 uniformly random $50\times r$ matrices, orthonormalize
their columns using the QR decomposition, and multiply on the right with $C$. Then
let $T_r$ be the canonical rank-$r$ tensor generated by these matrices as factor matrices. 
Two types of noise are added to $T_r$. Generate tensors $\mathcal{N}_1$ and $\mathcal{N}_2 \in \mathbb{R}^{50\times  50\times 50}$ with
elements drawn from the standard normal distribution. An intermediate tensor $\hat{\mathcal{Z}}$ is
generated as $\hat{\mathcal{Z}}= T_r + (100/l_1 -1)^{-1/2}\|T_r\|_F \mathcal{N}_1/\|\mathcal{N}_1\|_F$, and finally $\mathcal{Z}$ is obtained
as $\mathcal{Z}=\hat{\mathcal{Z}} + (100/l_2 - 1)^{-1/2} \|\hat{\mathcal{Z}}\|_F (\mathcal{N}_2\ast\hat{\mathcal{Z}})/\|\mathcal{N}_2\ast \hat{\mathcal{Z}}\|_F$, where $\ast$ denotes element-wise
multiplication. }  All numerical tests were performed in Matlab, using the Tensor Toolbox \cite{bader2015matlab} and
the Poblano Toolbox for optimization \cite{dunlavy2010poblano}.

\begin{figure}
\centering
\includegraphics[width=0.45\linewidth]{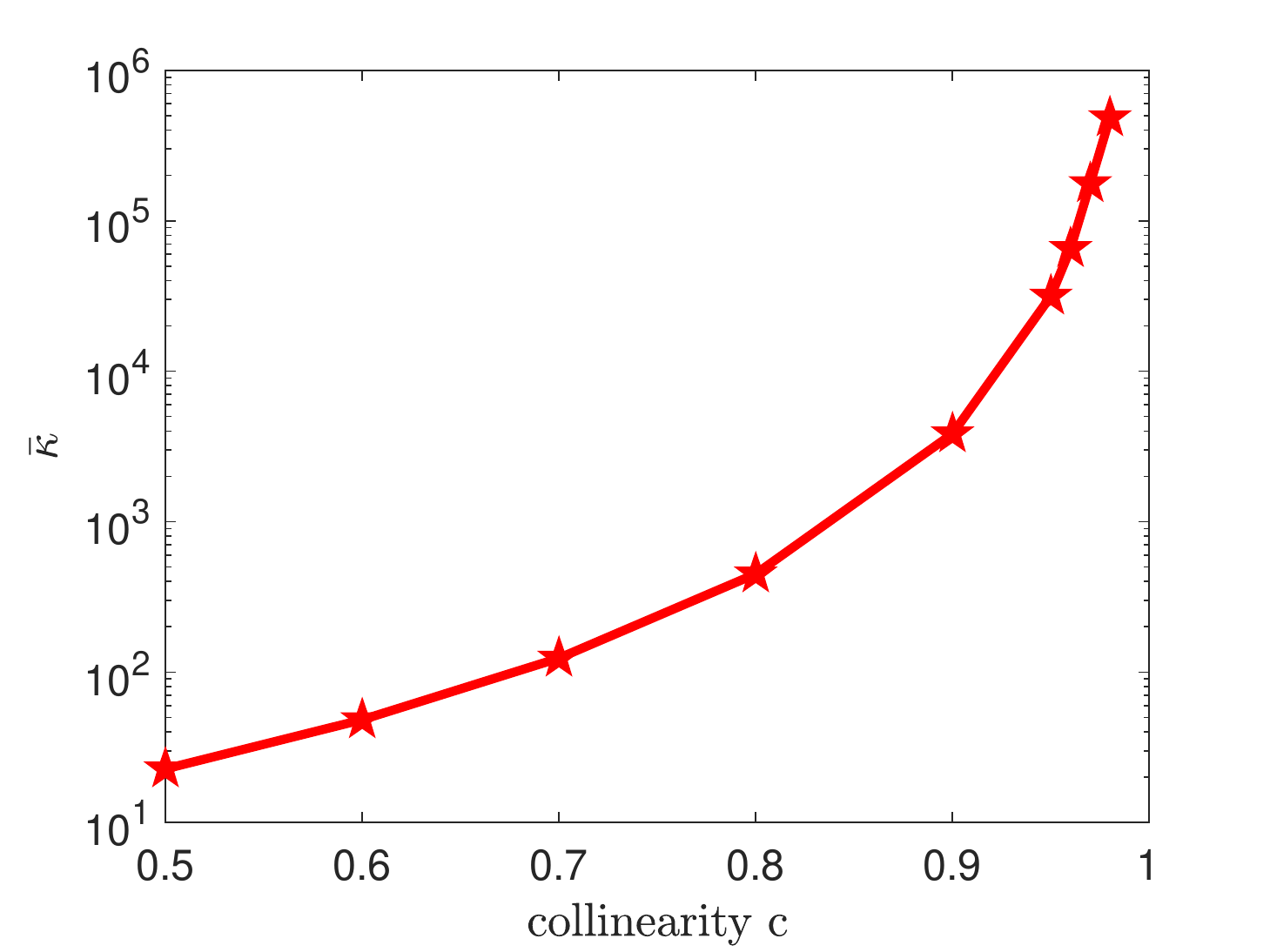}
\caption{Condition number of the Hessian at $x^*$ as a function of collinearity $c$.} \label{condi-vs-collinearity-plot}
\end{figure}
\yh{Since in our numerical tests we need to evaluate the
Hessian $H(x)$ of $f(x)$ at a local minimum, $x^{*}$, we first run our optimization methods until we obtain an
approximation of a fixed point $x^*$ where $\nabla f(x)$ vanishes up to machine accuracy.}
To get an idea of the difficulty of our test problems as a function of the collinearity parameter $c$,
we plot the (modified) condition number $\bar{\kappa}$ of $H(x^{*})$ (see \cref{eq:modified-condition-number}) as a function of $c$ in  \cref{condi-vs-collinearity-plot}. It can be observed that with increasing $c$, the condition number increases substantially.
It was known before that higher collinearity $c$ requires more iterations for ALS and other methods to converge,
but we now quantify this ill-conditioning using the modified condition number.
In the following tests we use $c = 0.5, 0.7, 0.9$ to validate our theoretical results
and to gain insight into how and by how much an effective nonlinear preconditioner like $q_{ALS}(x)$ may lead to improved
asymptotic convergence for the AA and NGMRES iterations, or, equivalently, by how much AA and NGMRES can accelerate $q_{ALS}(x)$
asymptotically.

\begin{figure}
\centering
\includegraphics[width=0.45\linewidth]{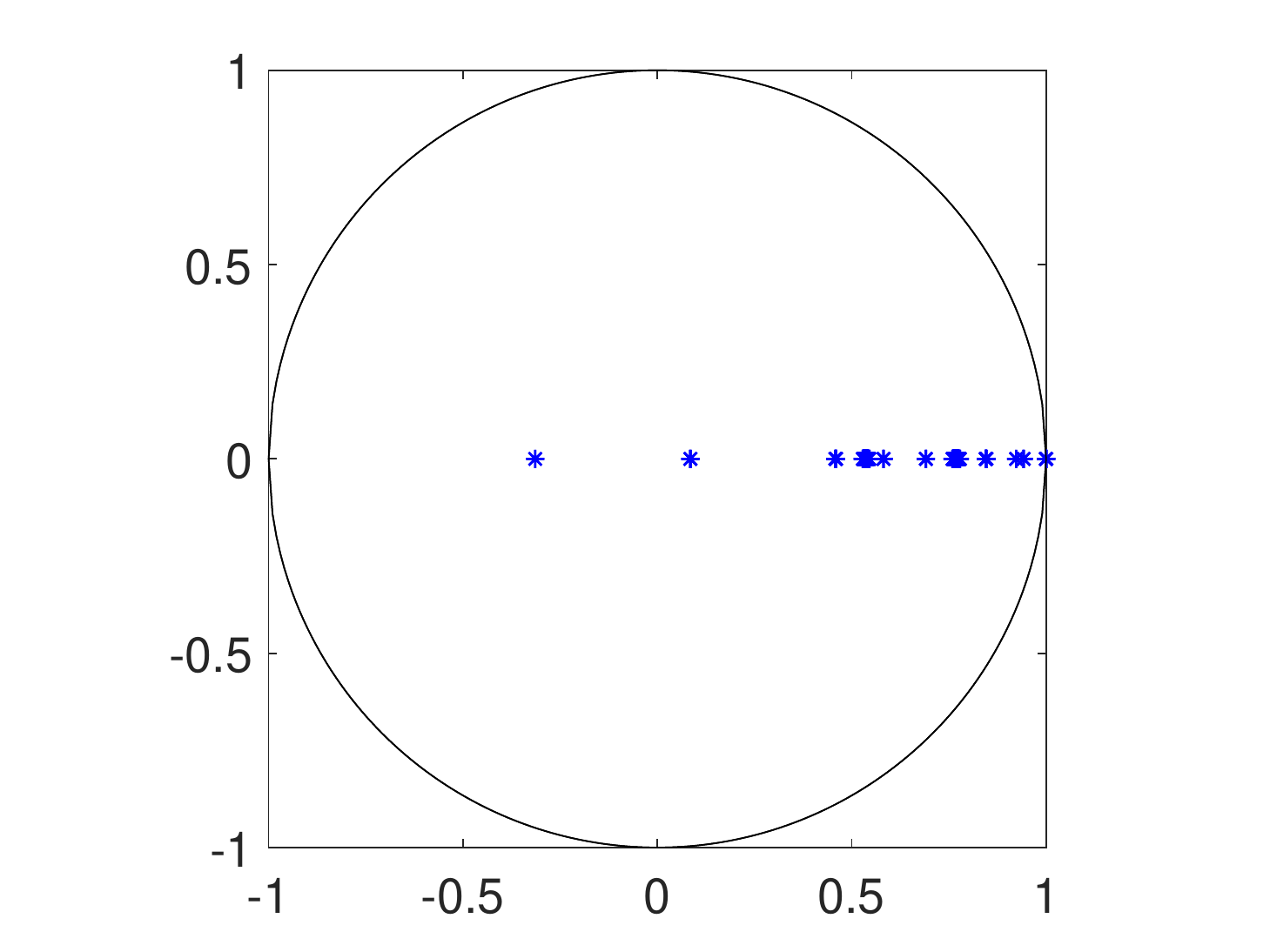}
\includegraphics[width=0.45\linewidth]{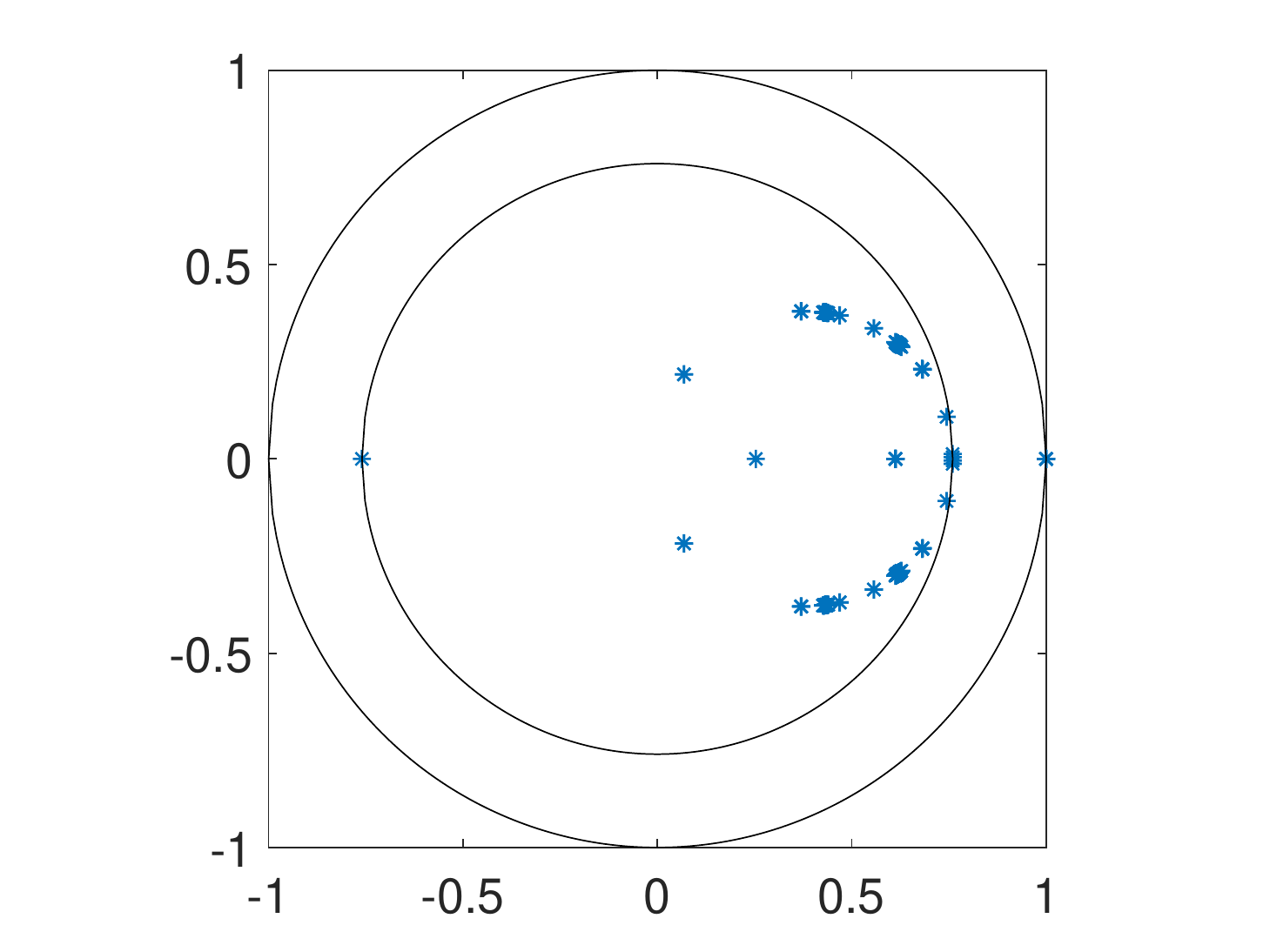}\\
\includegraphics[width=0.45\linewidth]{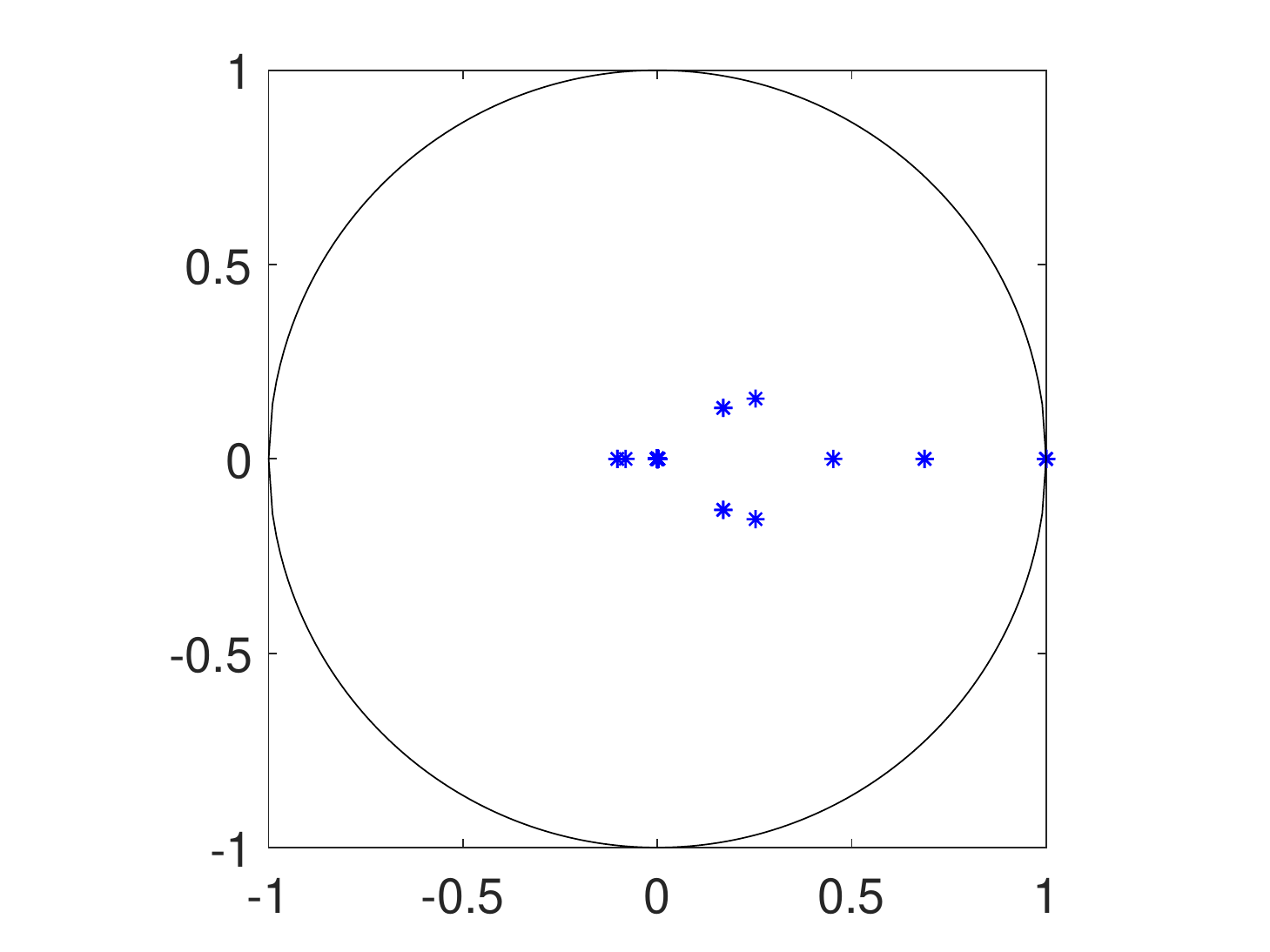}
\includegraphics[width=0.45\linewidth]{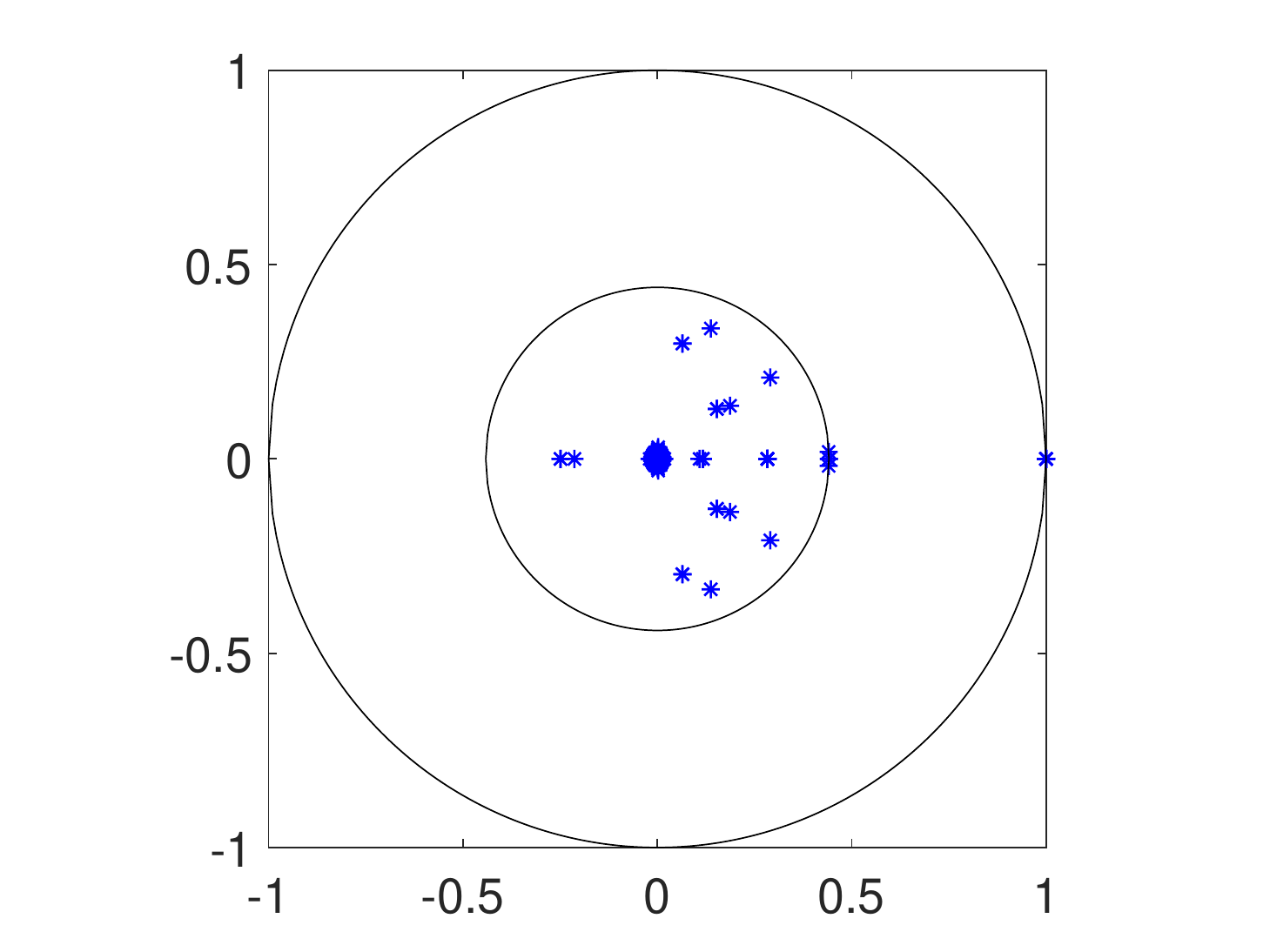}
\caption{
Eigenvalue distributions at $x^*$ for sAA(1) acceleration of steepest descent (top row) and ALS (bottom row) for a tensor problem with $c=0.5$.
(top left) Eigenvalues of $q'_{SD}$ with $\alpha$ from \cref{eq:SD-opt-alpha-beta}; $\rho(q'_{SD})=0.942$.
(top right) Eigenvalues of $T$ for sAA(1)-SD with the optimal parameters from \cref{Thm:positive-negative-AA-SD}; $\rho(T)=0.760$. The radius of the inner circle is $\rho^{*}_{\rm sAA(1)-SD}$ from \cref{eq:improved-SD-opt-cs}.
(bottom left) Eigenvalues of $q'_{ALS}$; $\rho(q')=0.688$.
(bottom right) Eigenvalues of $T$ for sAA(1)-ALS using the predicted $\beta$ in \cref{opt-beta-real}; $\rho(T)=0.441$.
The radius of the inner circle is $\rho_p$ in \cref{eq:lower-bound-ALS}.
For the ALS results, making abstraction of the eigenvalues one that correspond to the Hessian degeneracy,
the eigenvalue of $q'_{ALS}$ with the largest modulus is real, and the eigenvalue of $T$ with the largest modulus lies on the inner circle, in accordance with Conjectures \ref{conjec1-sAA-ALS} and \ref{conjec2-sAA-ALS}.
}
\label{eigs-sAA-c05-plot}
\end{figure}

\subsection{Quantifying asymptotic convergence acceleration by sAA and sNGMRES using spectral properties of the nonlinear preconditioner}
\label{subsec:spectral}

We first illustrate numerically how our theoretical results from Sections
\ref{sec:convergence-Anderson} and \ref{sec:convergence-NGMRES}
can be used to quantify the asymptotic convergence acceleration that can
be provided by the stationary sAA and sNGMRES methods with optimal parameters.
\vlong{In particular, this provides insight, for the canonical tensor decomposition application,
into how and by how much the asymptotic convergence of SD and
ALS can be accelerated by nonlinear
convergence acceleration methods of Anderson and NGMRES type, aiming to
explain asymptotic convergence behavior as seen in \cref{compare-SD-ALS-AA-plot}.}

\cref{eigs-sAA-c05-plot} considers acceleration by sAA(1)
for a mildly ill-conditioned tensor decomposition problem with $c=0.5$ and condition number
$\bar{\kappa}=22.76$.
The top row shows eigenvalue distributions for acceleration of the SD method.
The eigenvalues of $q'_{SD}(x^*)=I - \alpha H(x^*)$ (left panel) are real. SD converges slowly,
with asymptotic convergence factor $\rho(q'_{SD})=0.942$ (where $\alpha$ from \cref{eq:SD-opt-alpha-beta} is used).
The top right panel of \cref{eigs-sAA-c05-plot} shows how sAA(1) with the optimal parameters
from \cref{Thm:positive-negative-AA-SD} modifies the real $q'_{SD}(x^*)$ spectrum into a complex
spectrum for $T(q'(x^*))$ from \cref{eq:matrix-step1} with substantially reduced spectral radius:
$\rho(T)=0.760$ and asymptotic convergence is faster.
This optimal asymptotic convergence factor $\rho(T)$ for sAA(1)-SD
can be computed as a function of the condition number of $H$
using our theoretical result from \cref{Thm:positive-negative-AA-SD}.
Note that due to the scaling indeterminacy, $H$ has $2r$ eigenvalues 0, so $q'$ and $T$ have $2r$ eigenvalues 1.
These eigenvalues of value 1 do not influence the convergence speed.
\vlong{The convergence factor is the largest modulus
of eigenvalues of $q'$ and $T$ smaller than 1.}

The bottom row of \cref{eigs-sAA-c05-plot} shows how sAA(1) accelerates ALS.
The spectrum of $q'_{ALS}(x^*)=I - M(x^*)^{-1} H(x^*)$ contains complex eigenvalues and has a much smaller
spectral radius than SD, $\rho(q'_{ALS})=0.688$.
The bottom right panel shows how sAA(1) contracts the spectrum of $q'_{ALS}$, resulting in
a substantially reduced spectral radius for $T(q'(x^*))$ from \cref{system-form-NGM}: $\rho(T)=0.441$,
with the fastest asymptotic convergence by far.
Making abstraction of the eigenvalues one that correspond to the Hessian degeneracy,
the eigenvalue of $q'_{ALS}$ with the largest modulus is real, and the eigenvalue of $T$
with the largest modulus lies on the inner circle
with radius $\rho_p$ from \cref{eq:lower-bound-ALS}, in accordance with Conjectures
\ref{conjec1-sAA-ALS} and \ref{conjec2-sAA-ALS}.
This means that the asymptotic convergence factor of sAA(1)-ALS is given by $1-\sqrt{1-\rho_{q'}}$,
in accordance with \cref{thm:lower-bound-ALS} and \cref{conjec2-sAA-ALS}.
We have also verified that \cref{conjec2-sAA-ALS} holds for tensor problems with $c=0.7$ and $c=0.9$,
(see SM \cref{eigs-sAA-c07-c09-plot})
and for two additional real-data canonical tensor problems
from \cite{mitchell2020nesterov} (see \cref{eigs-sAA-real-plot}).
These results provide explanations and quantification of asymptotic convergence acceleration
by AA and NGMRES as seen in \cref{compare-SD-ALS-AA-plot} and in convergence plots in Section \ref{subsec:nonstationary}.

\begin{figure}
\centering
\includegraphics[width=0.45\linewidth]{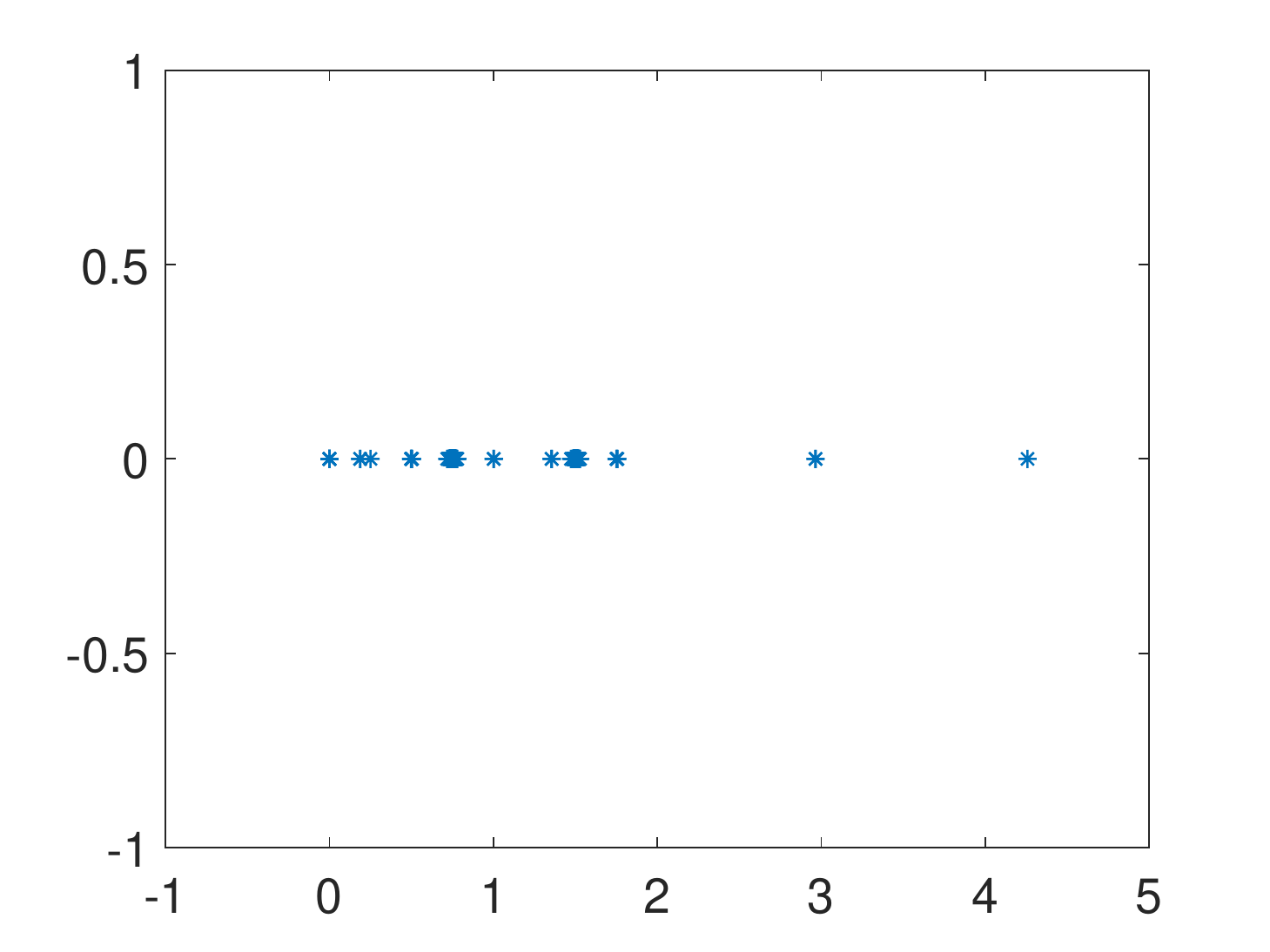}
\includegraphics[width=0.45\linewidth]{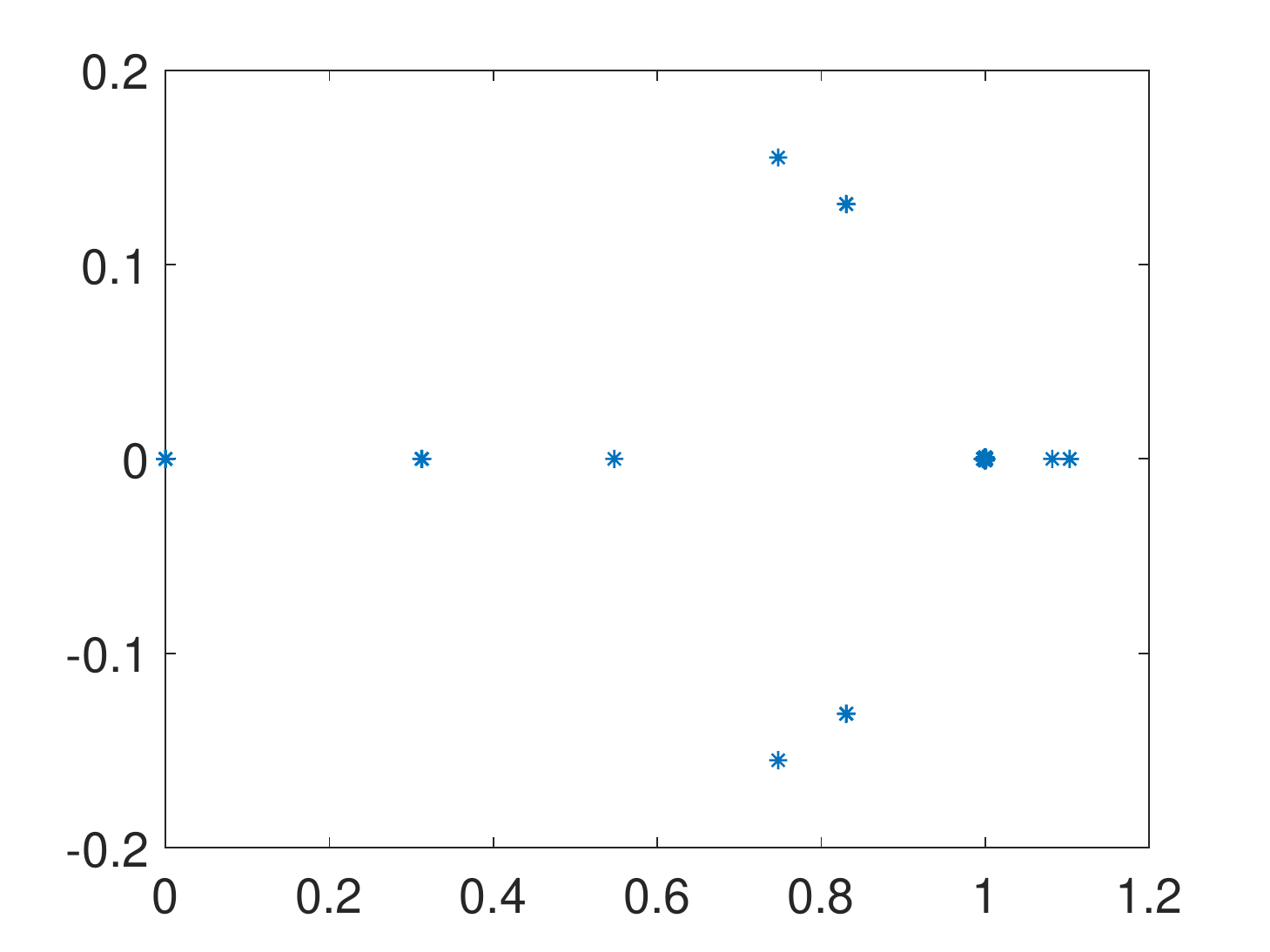}
\caption{Tensor problem with $c=0.5$. (left) Eigenvalue distribution of $H(x^*)$; the (modified) 2-norm condition number $\bar{\kappa}_2(H(x^*))=22.76$.
(right) Eigenvalue distribution of $M^{-1}(x^*)H(x^*)$; $\bar{\kappa}_2(M^{-1}(x^*)H(x^*))=7.39$.
}
\label{eigs-Hessian-c05-plot}
\end{figure}

\begin{figure}
\centering
\includegraphics[width=.45\linewidth]{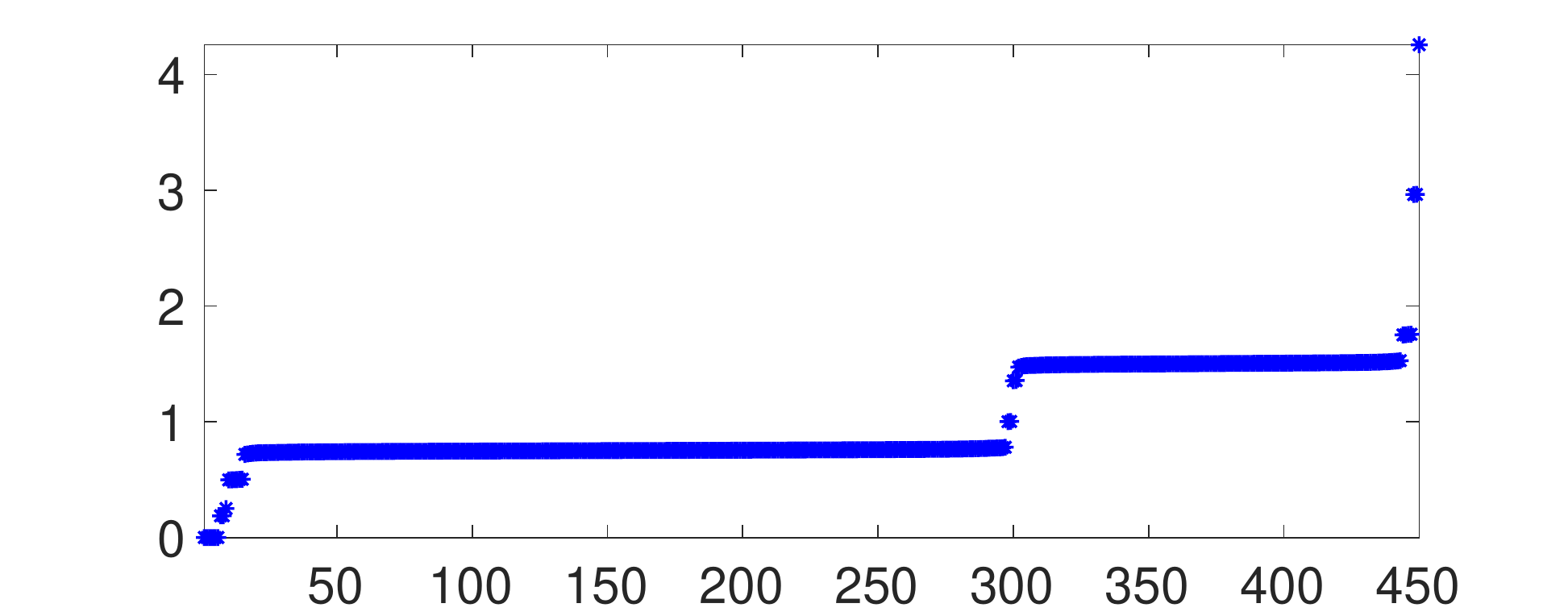}
\includegraphics[width=.45\linewidth]{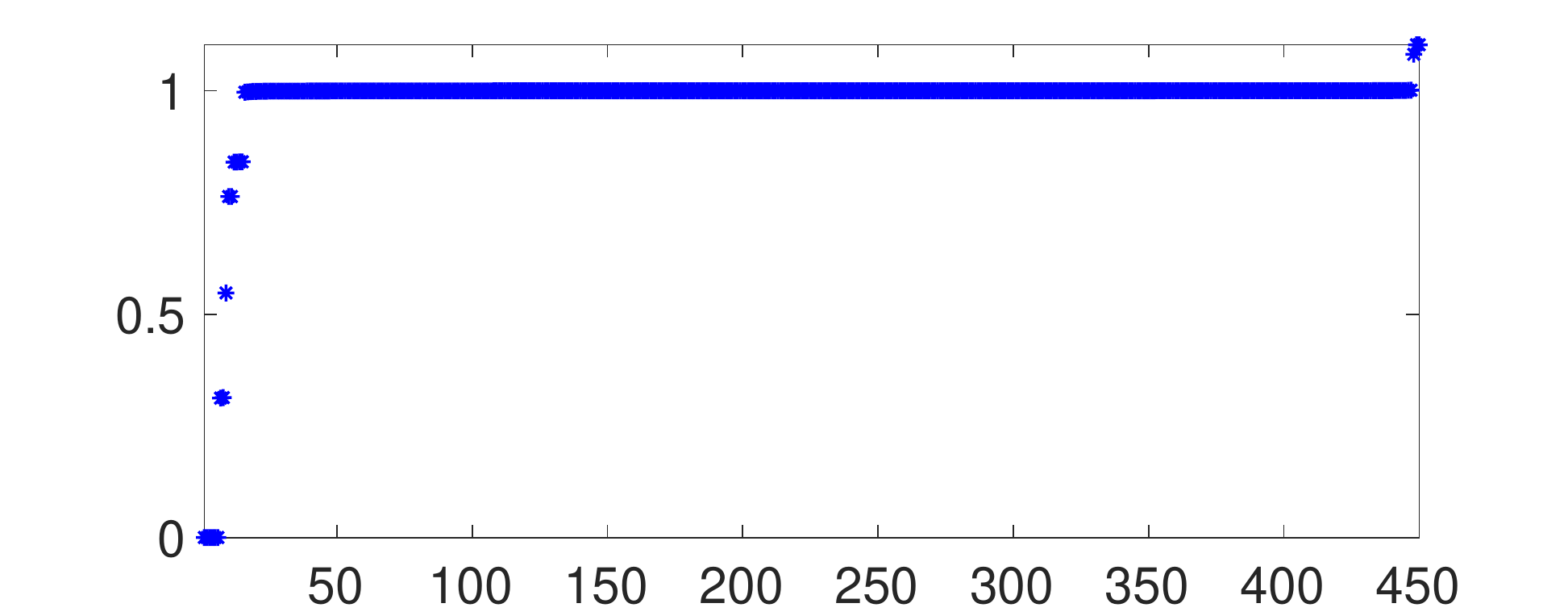}
\caption{Tensor problem with $c=0.5$. (left) Modulus of the eigenvalues of $H(x^*)$. (right) Modulus of the eigenvalues of $M^{-1}(x^*)H(x^*)$.} \label{eigs-Hessian-full-c05-plot}
\end{figure}

It is also interesting to interpret the difference in asymptotic convergence speed of
sAA(1)-SD and sAA(1)-ALS, as indicated by the convergence factors of the right-hand
panels of \cref{eigs-sAA-c05-plot},
in terms of the efficiency of the nonlinear preconditioners $q_{SD}$ and $q_{ALS}$ for sAA(1).
Indeed, comparing $q'_{SD}(x^*)=I - \alpha H(x^*)$ and $q'_{ALS}(x^*)=I - M(x^*)^{-1} H(x^*)$
with $q'=I - P\,A$ for the linear preconditioned case of \cref{eq:qprec} with preconditioning matrix $P$,
we consider the spectrum of the nonlinearly preconditioned Hessian $M(x^*)^{-1} H(x^*)$ (as used in ALS,
with $M(x^*)^{-1}$ being the nonlinear equivalent of $P$)
and the spectrum of the un-preconditioned Hessian $H(x^*)$ (as used in SD)
in \cref{eigs-Hessian-c05-plot,eigs-Hessian-full-c05-plot}.
In the linear case, it is known from preconditioning for GMRES that preconditioning
can substantially improve asymptotic convergence in several ways, including by
reducing the condition number of $A$, and by clustering eigenvalues such that the
value of the GMRES polynomial can more effectively be minimized over the spectrum
(see, e.g., \cite{trefethen1997numerical}).
In \cref{eigs-Hessian-c05-plot,eigs-Hessian-full-c05-plot} we see that ALS' preconditioning
by $M(x)^{-1}$ does indeed reduce the condition number of $H(x^*)$ by contracting the spectrum (\cref{eigs-Hessian-c05-plot}),
and it also clusters many eigenvalues at 1 (\cref{eigs-Hessian-full-c05-plot}), resulting in very efficient
nonlinear preconditioning for AA and NGMRES, compared to the identity-preconditioning provided by SD.

Results for sNGMRES(1) acceleration of SD and ALS that are similar to the sAA(1) results of \cref{eigs-sAA-c05-plot} are given in
SM Section \ref{subsubsec:sNGMRES(1)}, as well as further results for more ill-conditioned
tensor problems with $c=0.7$ and $c=0.9$, confirming the general findings of \cref{eigs-sAA-c05-plot}.

\subsection{Convergence acceleration by nonstationary AA and NGMRES}
\label{subsec:nonstationary}
In this section we shift the focus from the asymptotic numerical results at $x^*$ of Section \ref{subsec:spectral} in terms
of eigenvalue spectra and spectral radii, and
consider complete nonlinear convergence histories $f(x_k)-f(x^*)$ starting from the initial guess $x_0$, with special
attention for the convergence behavior as $x_k \rightarrow x^*$.
We also investigate the linear upper bounds described in Section \ref{sec:inf-bounds} for GMRES($\infty$) applied to the linearized problem, and compare with the asymptotic convergence behavior of AA and NGMRES for $m=\infty$ and for finite $m$. We focus on ALS.
Note that the convergence plots in this section show $f(x_k)-f(x^*)$, which converges asymptotically with
factor $\rho^2$, see \cref{rem:f2}.
In the nonlinear test runs we use the following parameters and notation:
\begin{enumerate}
\item We use a globalization method based on the Mor\'{e}-Thuente cubic line search of \cite{MR1367800} for all AA($m$), sAA($m$), sNGMRES($m$), and NGMRES($m$) runs, with line search parameters chosen as in \cite{sterck2012nonlinear,sterck2013steepest}.
\item For the Nesterov method with restart \cite{mitchell2020nesterov}: we use a gradient ratio formula for $\beta^{(k)}$, and we use the function restart mechanism. For details, see \cite{mitchell2020nesterov}.
\item Theoretical convergence factors:  The optimal convergence factors $\rho_{\rm SD}$, $\rho_{\rm ALS}$,
$\rho_{\rm sAA(1)-SD}$, and $\rho_{\rm sNGMRES(1)-SD}$ are as in \cref{comparsion-ASD-vs-NGMRESSD}. The optimal
convergence factor $\rho_{\rm sAA(1)-ALS}$ is from \cref{thm:lower-bound-ALS,conjec2-sAA-ALS}, and the
optimal convergence factor $\rho_{\rm sNGMRES(1)-ALS}$ is from \cref{convergence-table-AA-ALS}.
\item All initial guesses $x_0$ are chosen with uniformly random components in [0,1].
\end{enumerate}
Matlab code with the acceleration methods used for our tests can be found at
\url{https://github.com/hansdesterck/nonlinear-preconditioning-for-optimization}.

\subsubsection{Comparing asymptotic convergence of nonstationary AA and NGMRES with optimal stationary convergence
factors}
\label{subsubsec:nonstationary-stationary}

\begin{figure}
\centering
\includegraphics[width=0.45\linewidth]{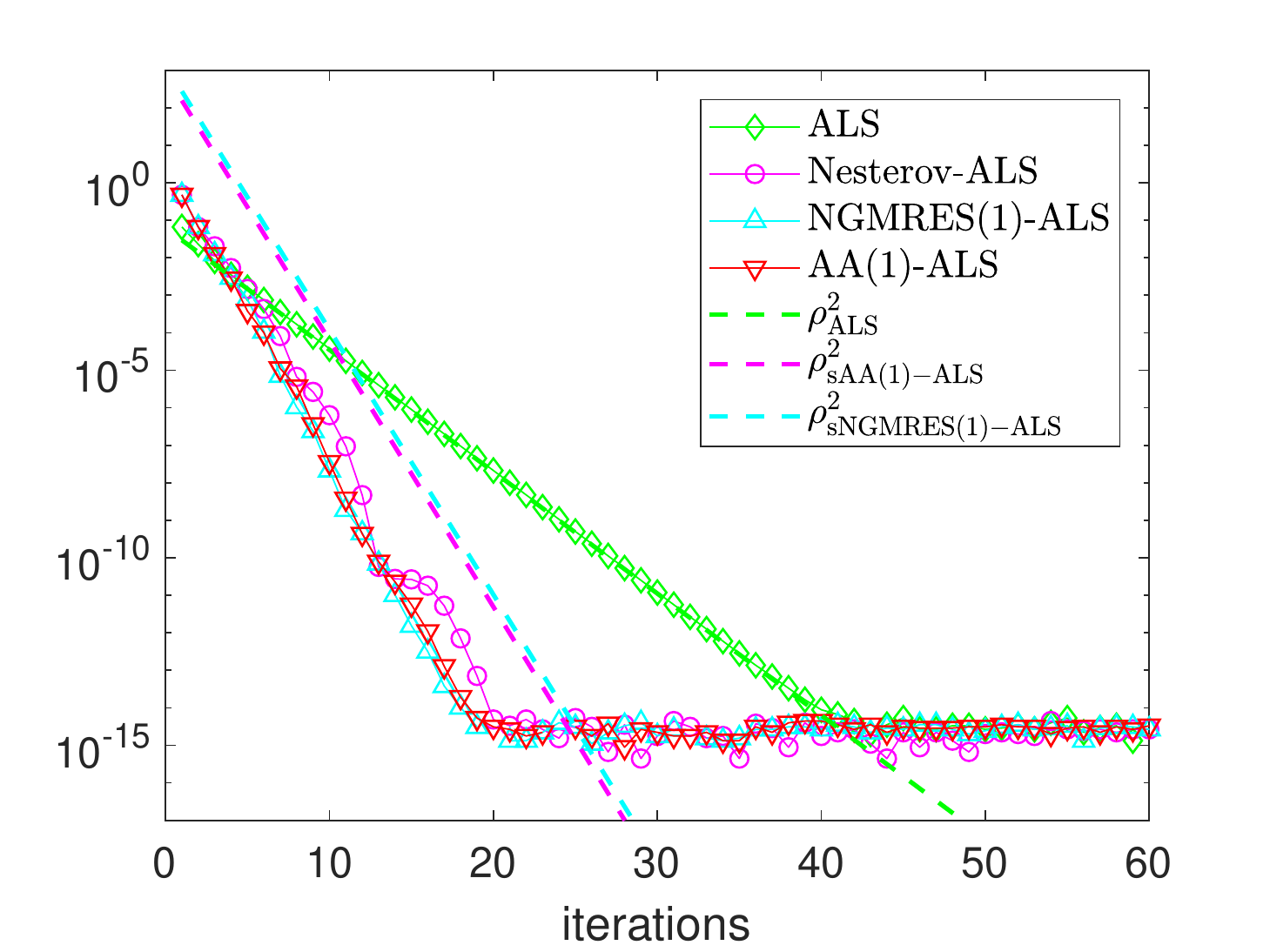}
\includegraphics[width=0.45\linewidth]{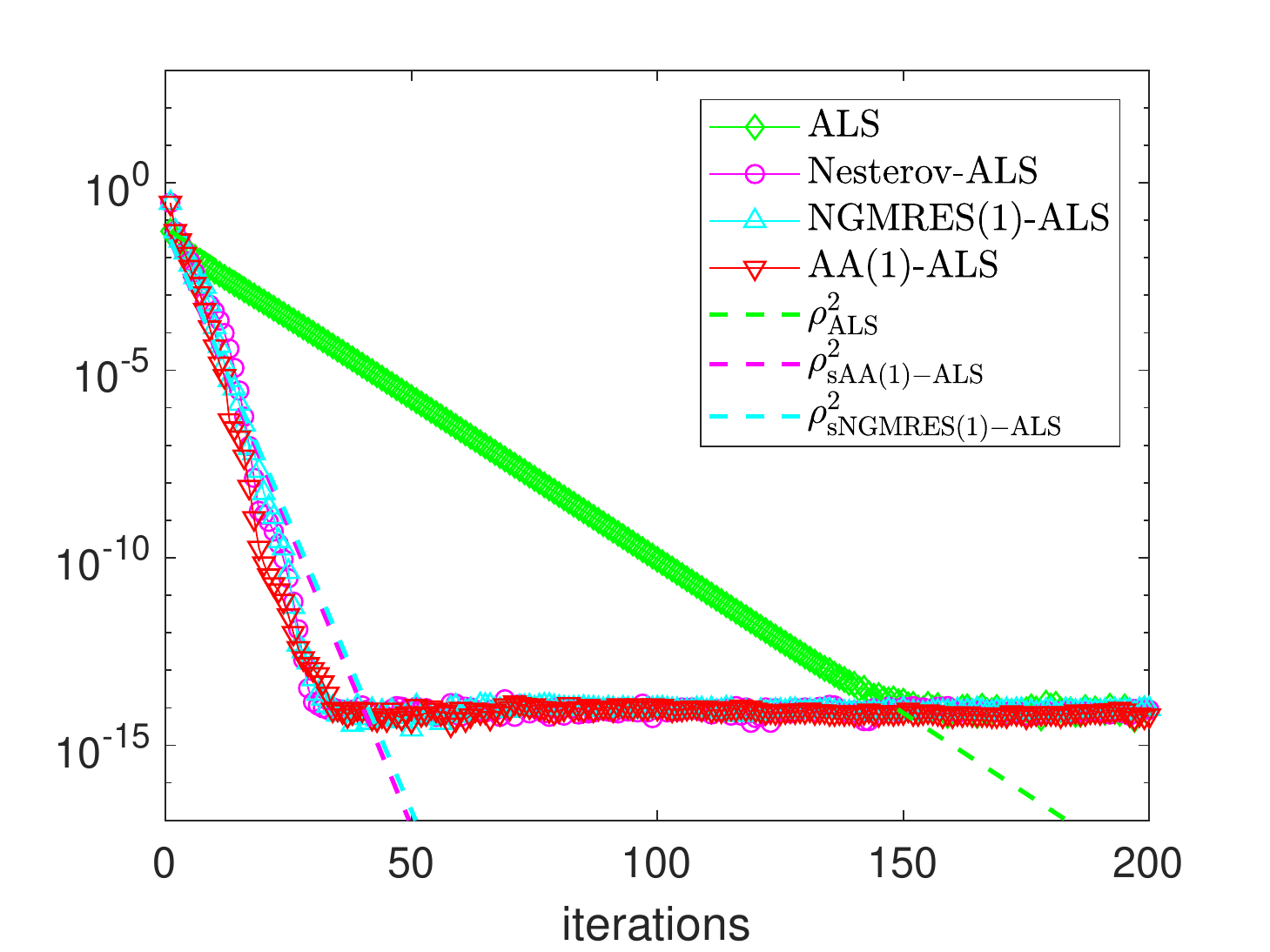}
\caption{
Comparison of the nonstationary AA(1)-ALS, NGMRES(1)-ALS, and Nesterov-ALS methods with theoretical asymptotic convergence factors for optimal stationary methods, for tensor problems with $c=0.5$ (left panel) and $c=0.7$ (right panel). The vertical axis represents $f(x_k)-f(x^*)$, the convergence towards the minimum value of $f(x)$.
}
\label{c05-AA-NGMRES-plot}
\end{figure}

\cref{c05-AA-NGMRES-plot} shows what we believe is an interesting result.
For the same tensor problem with $c=0.5$ as in \cref{eigs-sAA-c05-plot}, and another problem with
$c=0.7$, the figure shows that the nonstationary
iterations AA(1)-ALS, NGMRES(1)-ALS, and Nesterov-ALS converge \emph{with nearly the same asymptotic convergence factor}
as the optimal stationary methods sAA(1)-ALS and sNGMRES(1)-ALS.
\cref{c05-AA-NGMRES-plotb} confirms this overall picture for a tensor problem with $c=0.9$.
Intuitively this is not unexpected, but perhaps still surprising:
a plausible explanation is that the \emph{locally optimal} least-squares coefficients in each iteration of the nonstationary methods lead to asymptotic convergence behavior that has nearly the same linear convergence factor as the stationary methods with fixed coefficients that are \emph{globally optimal} in terms of asymptotic convergence factor.
These numerical results indicate that the effectiveness of ALS as a nonlinear preconditioner, as was
demonstrated and quantified for the stationary sAA(1) and SNGMRES(1) methods theoretically in Sections \ref{sec:convergence-Anderson} and \ref{sec:convergence-NGMRES} and numerically and in terms of spectral properties in Section \ref{subsec:spectral},
appears to translate to the nonstationary methods of \cref{c05-AA-NGMRES-plot}. As such, we can extrapolate
that our ways to understand and quantify the effectiveness of nonlinear preconditioners for the stationary methods
also offer good predictions for the nonstationary, practical methods.

Note that our numerical results as in \cref{c05-AA-NGMRES-plot} report iteration counts, where it has to be taken into account that the cost of an accelerated iteration is about two to four times the cost of an SD or ALS iteration, see \cite{mitchell2020nesterov}.
\subsubsection{Comparing asymptotic convergence of nonstationary AA and NGMRES with GMRES($\infty$) convergence factors}
\label{subsubsec:nonstationary-GMRES}

\begin{figure}
\centering
\includegraphics[width=0.45\linewidth]{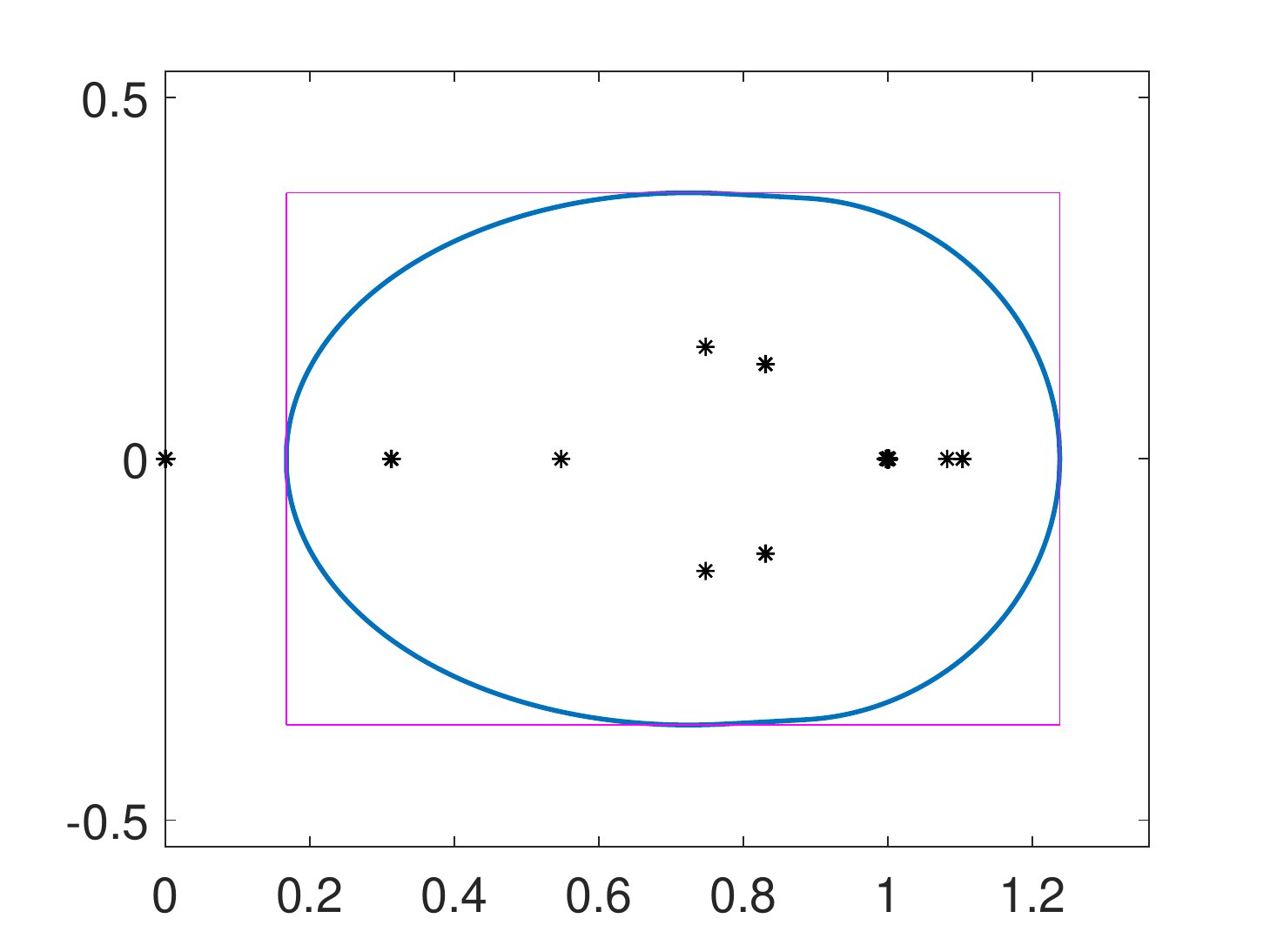}
\caption{
Field of values (blue) and bounding rectangle \cref{eq:bounding} for the matrix $B$ of \cref{eq:B}, where $B$ is the projection
of the preconditioned Hessian $M^{-1}(x^*) H(x^*)$ from the right panel of \cref{eigs-Hessian-c05-plot} onto the subspace
spanned by the eigenvectors of $M^{-1}(x^*) H(x^*)$ that have nonzero eigenvalues. The eigenvalues of $B$ and $M^{-1}(x^*) H(x^*)$
are also shown.
The FOV provides a linear convergence bound for GMRES applied to the (equivalent) linearized
fixed-point equation $B \, z = Q^T \, b$, see \cref{thm:GMRES-conv}.
}
\label{c05-FOV-plot}
\end{figure}

\begin{figure}
\centering
\includegraphics[width=0.45\linewidth]{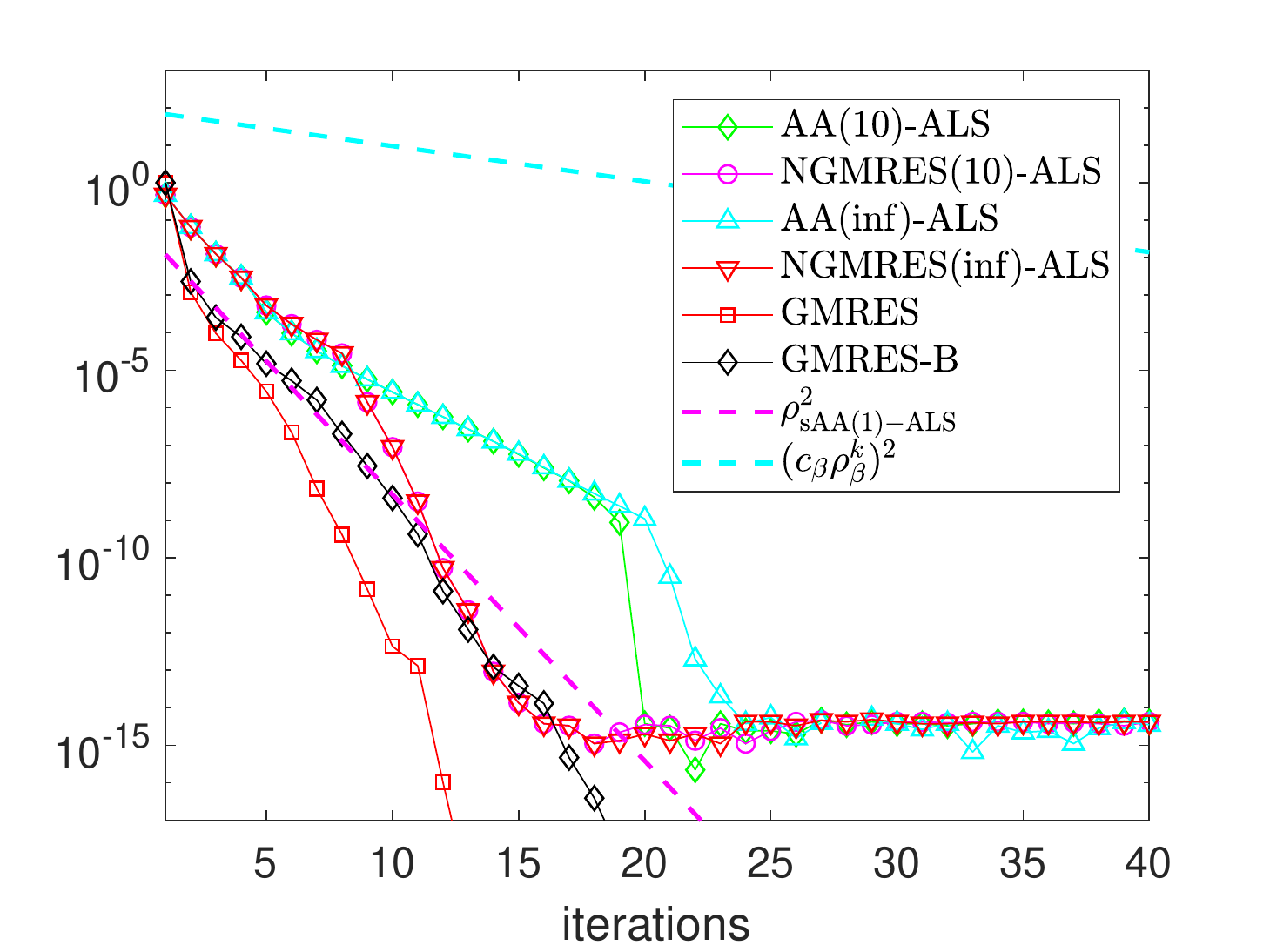}
\includegraphics[width=0.45\linewidth]{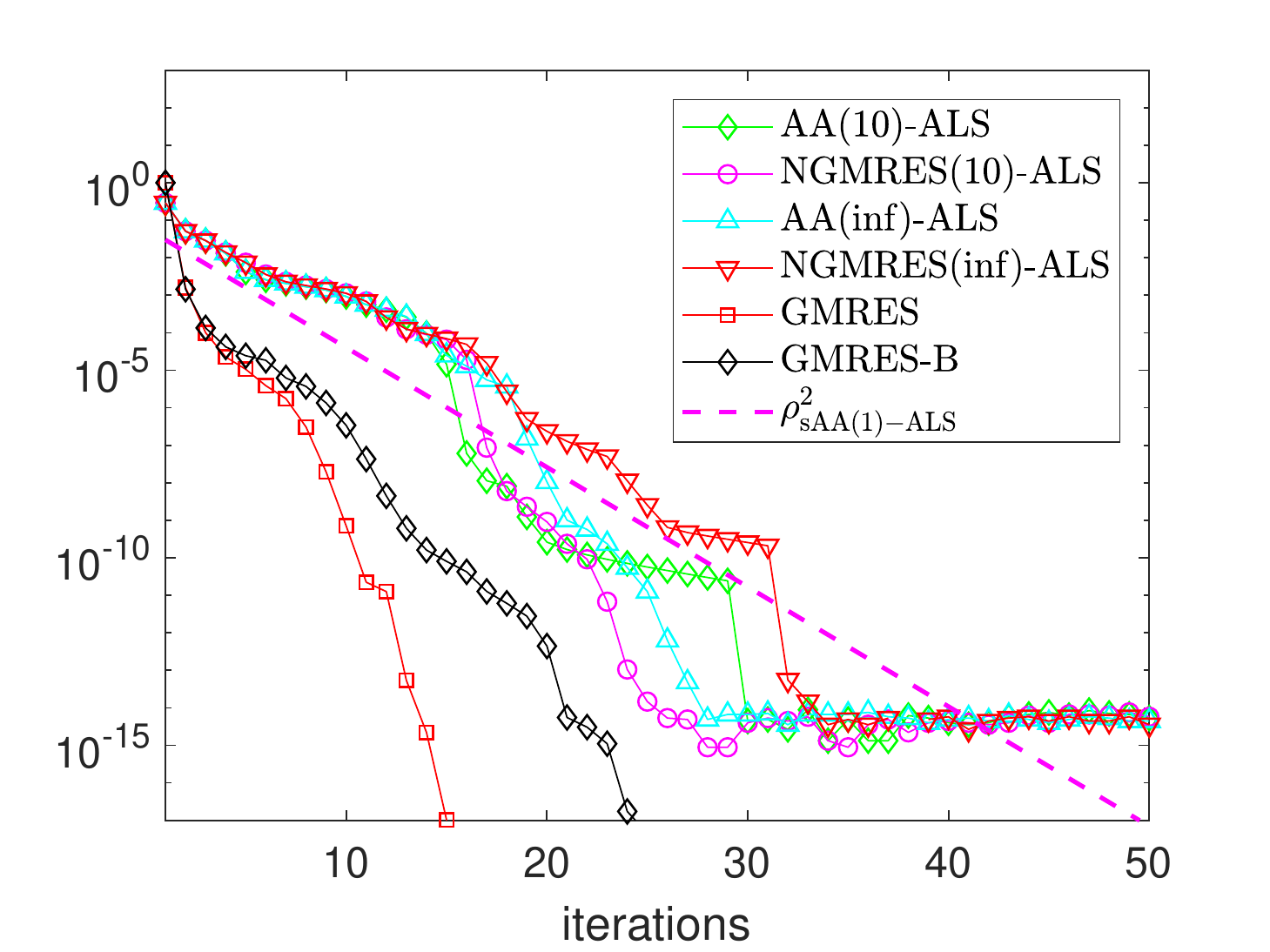}
\caption{Comparison of AA and NGMRES convergence curves for two tensor problems with $c=0.5$ (left panel)
and $c=0.7$ (right panel).
The four nonlinear AA and NGMRES curves are compared with GMRES applied to
linearized equation \cref{eq:fixed-lin} and GMRES-B applied to projected
nonsingular linearized system \cref{eq:projected}. For $c=0.5$,
$c_\beta \rho_{\beta}^k$ (computed based on the numerical FOV of \cref{c05-FOV-plot})
provides a pessimistic upper bound.
Our new $\rho_{sAA(1)-ALS}$ from \cref{thm:lower-bound-ALS,conjec2-sAA-ALS} appears to provide
a useful indication of the convergence speed of the linear and nonlinear methods.
For the four nonlinear methods, the vertical axis represents $f(x_k)-f(x^*)$.
For the GMRES runs, the vertical axis represents $\|r_k\|^2/\|r_0\|^2$.}
\label{c05-inf-plot}
\end{figure}

Finally, we discuss results for the other, more direct, way of predicting asymptotic convergence factors for ALS accelerated by nonstationary
AA and NGMRES, based on GMRES bounds for the linearized problem about $x^*$, with window size $m=\infty$
and using the FOV bounds from \cref{thm:GMRES-conv}.
We consider the tensor problem with $c=0.5$ and associated preconditioned Hessian $M^{-1}(x^*) H(x^*)$
from the right panel of \cref{eigs-Hessian-c05-plot}.
We consider linearized fixed-point equation (\ref{eq:fixed-lin}) with $I-q'(x^*)=M^{-1}(x^*) H(x^*)$, see \cref{eq:GS-iter}.

Since $M^{-1}(x^*) H(x^*)$ is singular, we cannot directly use \cref{thm:GMRES-conv} to determine linear convergence
bounds for solving this system using GMRES.
We proceed as follows to transform the singular system to an equivalent nonsingular system that can be used to quantify
asymptotic GMRES convergence. Let $V$ be the matrix with the eigenvectors of $M^{-1}(x^*) \, H(x^*)$ as its columns, but
with the 6 eigenvectors that correspond to eigenvalues 0 removed. Let $V=QR$ be the thin $QR$ decomposition of $V$,
and consider
\begin{equation}
B=Q^T \, M^{-1}(x^*) \, H(x^*) \,Q.
\label{eq:B}
\end{equation}
Matrix $B$ has the same eigenvalues as $M^{-1}(x^*) \, H(x^*)$, except for
the 6 zero eigenvalues, and the eigenvectors $y$ of $B$ are related to the eigenvectors $x$ of $M^{-1}(x^*) \, H(x^*)$ by
$Q \, y=x$. We can transform the singular linearized system $M^{-1}(x^*) \, H(x^*) \, x = M^{-1}(x^*) \, H(x^*) \, x^* =: b$ into
the equivalent nonsingular system
\begin{equation}
B \, z = Q^T \, b,
\label{eq:projected}
\end{equation}
with $x=Q\,z$.
We can now obtain linear convergence bounds for solving this transformed system using GMRES with the help of
\cref{thm:GMRES-conv}.

We compute $\rho_\beta$ in the linear asymptotic convergence bound of \cref{thm:GMRES-conv}
in two ways. First, we approximate the field of values of $B$ by a bounding rectangle \cite{mees1979domains}
\begin{equation}
   [\lambda_{\min}(B_s), \lambda_{\max}(B_s)] \times [ -\rho(B_a)i,  \rho(B_a)i ],
\label{eq:bounding}
\end{equation}
where $B_s=(B+B^T)/2$ and $B_a=(B-B^T)/2$ are the symmetric and anti-symmetric parts of $B$.
Based on this bounding rectangle for the FOV, we obtain the asymptotic convergence factor
$\rho_{\beta,bb}=0.9014$ for the bound of \cref{thm:GMRES-conv}.
A slightly better convergence factor for the bound can be obtained by computing the FOV numerically
\cite{driscoll2014chebfun,trefethen2005spectra}, giving $\rho_{\beta,num}=0.8971$.
\cref{c05-FOV-plot} shows the FOV of $B$ and the bounding rectangle estimate obtained from \cref{eq:bounding}.

The results in \cref{c05-inf-plot} show several interesting findings.
We first focus on the left panel with $c=0.5$.
First, each of the AA($\infty$)-ALS, AA(10)-ALS,
NGMRES($\infty$)-ALS and NGMRES(10)-ALS methods show comparable convergence behavior.
Both GMRES applied to linearized equation \cref{eq:fixed-lin} and GMRES applied to the projected
nonsingular system \cref{eq:projected} converge with similar asymptotic speed as the four nonlinear methods.
All these methods appear to satisfy the asymptotic linear convergence bound of \cref{thm:GMRES-conv},
with convergence factor $\rho_{\beta}=\rho_{\beta,num}$ computed based on the FOV of \cref{c05-FOV-plot}.
This indicates that, if a linear asymptotic convergence factor bound can be found for the linearized problem
about $x^*$, e.g., as in \cref{thm:GMRES-conv}, then the convergence of the nonlinear AA and NGMRES iterations may locally
have the same  linear asymptotic convergence factor bound, i.e., if $x_0$ is chosen close enough to $x^*$, in accordance with
\cref{conjec:inf}.
However, this bound, while rigorous for GMRES applied to \cref{eq:projected}, appears quite pessimistic.
On the other hand, and remarkably, our new theoretical convergence factor $\rho_{sAA(1)-ALS}$ from
\cref{thm:lower-bound-ALS,conjec2-sAA-ALS} appears to be an accurate indicator of the asymptotic
convergence speed of the four nonlinear methods and GMRES.
The $c=0.7$ result confirms this overall picture. We note, however, that we were not able to obtain an FOV
convergence factor for $c=0.7$ using \cref{thm:GMRES-conv}, because the left intersection of the numerical
FOV with the $x$-axis occurred at a slightly negative $x$-value. It is possible the accurate FOV contains 0, but
this negative number may also be a result of the ill-conditioning of $M^{-1}(x^*)$ and resulting inaccuracies
in the eigenvector and $QR$ computations. This is another potential drawback of estimating a linear
convergence factor through \cref{thm:GMRES-conv} for ill-conditioned matrices, while our
$\rho_{sAA(1)-ALS}$ prediction may be more robust.
Further results for the problems from \cref{c05-inf-plot}
are shown in \cref{c05-inf-plotb} with different random seeds,
confirming the general trends from \cref{c05-inf-plot}.

\section{Conclusion} \label{sec:concl}
In this work, we provide two methods for estimating the asymptotic
convergence improvement resulting from AA and NGMRES
acceleration of fixed-point methods. While such improvement has been
observed numerically in many applications, there is a lack in understanding
and quantifying this improvement theoretically.
Asymptotic convergence results for AA and NGMRES
with finite window size appear difficult, but we made progress
in the simplified setting of stationary versions of AA and NGMRES.
We derived theoretical results, for small window sizes, on finding
coefficients for the stationary methods that result in optimal asymptotic
convergence factors, assuming knowledge of $q'(x^*)$.
This allowed us to understand the effectiveness
of a fixed-point iteration viewed as a nonlinear preconditioner for AA or NGMRES
in terms of the spectral properties of $q'(x^*)$.
We showed numerically that the convergence factors of the stationary
methods with globally optimal, fixed coefficients indeed provide a
good estimate of the asymptotic convergence of nonstationary AA
and NGMRES, which determine optimal coefficients locally in each iteration.

Our second way of estimating AA and NGMRES asymptotic convergence factors
applies GMRES to the fixed-point method
linearized about the fixed point, and derives linear convergence bounds
for GMRES using the field of values of $I-q'(x^*)$. While these bounds are
rigorous for GMRES
\vlong{and may provide the best hope of deriving local
linear convergence bounds for the nonlinear methods through our
\cref{conjec:inf},}
we found the associated linear convergence
factors pessimistic in our numerical tests, and less predictive than the
convergence factor estimates we obtained from our analysis
of stationary AA and NGMRES methods.
\cref{conjec:inf} may provide a direction for
proving local linear convergence bounds for AA and NGMRES with
infinite window size, but, similar to what we
explained for the linear case of GMRES, it is likely that the mere existence
of such bounds will depend substantially on the matrix properties
of $I-q'(x^*)$, including the location of its eigenvalues in the complex
plane and the geometry of its eigenvector basis.
\vlong{Specifically, if
$0 \notin \textrm{FOV}(I-q'(x^*))$, \cref{thm:GMRES-conv} and
\cref{conjec:inf} would imply existence of a local linear
asymptotic convergence bound for AA and NGMRES.}

In terms of the canonical tensor decomposition application,
this paper provides the insight
and methodology to understand and quantify why and by
how much the acceleration by AA and NGMRES improves the
asymptotic convergence of ALS, or, equivalently, why ALS
is an effective nonlinear preconditioner for AA and NGMRES,
as had been observed numerically before
\cite{sterck2012nonlinear,de2016nonlinearly,mitchell2020nesterov}.
Next steps include proving the conjectures we made on
optimal sAA(1)-ALS convergence bounds for canonical
tensor decomposition, and attempting to bound ALS convergence
factors in terms of $\bar{\kappa}(H(x^*))$, using the
structure of the canonical tensor decomposition Hessian.
Similarly, it may be possible to bound $\rho_\beta$ in
\cref{thm:GMRES-conv} applied to \cref{eq:fixed-lin}
in terms of $\bar{\kappa}(H(x^*))$
or other properties of the Hessian.
\yh{More generally, the approaches and results of this paper
can be applied to quantify convergence acceleration by AA or NGMRES
applied to other fixed-point methods such as the Alternating
Direction Method of Multipliers (ADMM) \cite{zhang2019accelerating},
and can be extended to problems with less smoothness.
For example, see \cite{wang2020} for an application of the findings
of this paper to AA acceleration of ADMM.}


\vspace{-.4cm}

\bibliographystyle{siamplain}
\bibliography{convergence_reference}

\newpage

\renewcommand\thefigure{\thesection.\arabic{figure}}
\setcounter{figure}{0}
\renewcommand\thetable{\thesection.\arabic{table}}
\setcounter{table}{0}
\appendix
\addtocounter{section}{19}
\section*{Supplementary materials}

\subsection{A weaker form of \cref{thm:lower-bound-ALS}}
\label{subsec:weaker}
\ \\
If the assumption that $\mu =\rho_{q'}$ in \cref{thm:lower-bound-ALS} does not hold, we still can give an estimate of the lower bound on the asymptotic convergence factor by considering the nonnegative (real) eigenvalues of $q'$, which is stated in the following.
\begin{corollary}[lower bound for fixed-point methods with complex Jacobian spectrum]\label{cor:weaker}
\yh{Consider the sAA(1) acceleration method \cref{system-form} with Jacobian matrix $T$ defined in (\ref{eq:matrix-step1}) applied to fixed-point method (\ref{eq:fixed-point}) with fixed point $x^*$.}
Denote by $\rho_{+}$ the largest nonnegative (real) eigenvalue of $q'(x^*)$. 
Then the optimal asymptotic convergence factor of sAA(1) is bounded by
\begin{equation*}
  \min_{\beta\in\mathbb{R}}\rho(T(x^{*};\beta)) \geq  1-\sqrt{1-\rho_{+}},
\end{equation*}
and if the equality holds, then \yh{the unique optimal $\beta$ is given by}
\begin{equation*}
\beta_{\rm opt}=\frac{1-\sqrt{1-\rho_+}}{1+\sqrt{1-\rho_{+}}}.
\end{equation*}
\end{corollary}
\begin{proof}
The proof is similar to that in \cref{thm:lower-bound-ALS}. We use the fact that $\rho_{+} \in \sigma(q')$, and
\begin{equation*}
 \min_{\beta \in\mathbb{R}} \rho(T(x^{*};\beta))= \min_{\beta \in\mathbb{R}}\max\big(\cup_{\mu\in\sigma(q')} \mathcal{S}_{\mu}(\beta)\big)\geq \min_{\beta \in\mathbb{R}}\max\mathcal{S}_{\rho_{+}}(\beta).
\end{equation*}
It is easy to see that $\min_{\beta \in\mathbb{R}}\max\mathcal{S}_{\rho_{+}}(\beta)=1-\sqrt{1-\rho_{q_+}}$, using
\cref{minmax-real-nonnegative-case}.
\end{proof}

\subsection{Proof of \cref{thm:sNGMRE-R-SD} -- Optimal asymptotic convergence factor of sNGMRES-R(1) \yh{applied to fixed-point methods with real Jacobian spectrum}}
\label{subsec:proof4.1}
\ \\
We prepare for the proof of \cref{thm:sNGMRE-R-SD} \yh{for sNGMRES-R(1) applied to SD by first proving \cref{minmax-real-nonnegative-case-NGM} for the general case of fixed-point methods with real Jacobian spectrum}.

\yh{Denote the eigenvalues of $q'$ as $\mu$ (real or complex).} Then it can be shown that the eigenvalues $\lambda$ of $T_N$ in \cref{eq:matrix-step1-NGM} satisfy
\begin{equation}\label{eq:AA-eig-form-NGM}
  \lambda^2-(1+\beta)\mu \lambda+ \beta=0.
\end{equation}
The two roots of the above equation are
\begin{equation}
 \yh{ \lambda_{1,2}}=\frac{(1+\beta)\mu\pm \sqrt{(1+\beta)^2\mu^2-4\beta}}{2}=:y(\beta).
  \label{eq:g1-postive-mu-NGM}
\end{equation}

For any given $\mu$, we define the set
\begin{equation}\label{eq:eig-set-q-prime-NGM}
  \yh{\mathcal{T}_{\mu}(\beta) =\Big\{|\lambda_1|, |\lambda_2| \Big\}.}
\end{equation}
Since the eigenvalues of $q'$ will affect the eigenvalues $\lambda$ of $T_N$ in \cref{eq:AA-eig-form-NGM}, it is useful to know how $\lambda$ changes for a given $\mu \in \sigma(q')$, and what the optimal value of  $\min_{\beta}\max\mathcal{T}_{\mu}(\beta)$  is for a given $\mu$.

\yh{Let us now consider the case of real $\mu$.}
In order to guarantee sNGMRES-R(1) converges, $\lambda<1$ is required. It follows that $\beta\in (-1,1)$, since the product of the two roots of \cref{eq:AA-eig-form-NGM} is $\beta$.
\begin{lemma}\label{minmax-real-nonnegative-case-NGM}
\begin{enumerate}
\item Assume $\mu \in \mathbb{R}$, and $0<|\mu|<1$. Then
\begin{equation*}
  \min_{\beta}\max\mathcal{T}_{\mu}(\beta) =  \frac{|\mu|}{1+\sqrt{1-\mu^2}},
\end{equation*}
if and only if \yh{$\beta$ is taken to be}
\begin{equation*}
  \beta_{\rm opt}(\mu)=\frac{1-\sqrt{1-\mu^2}}{1+\sqrt{1-\mu^2}}.
\end{equation*}
Moreover, for any given $\mu_1, \mu_2 \in \mathbb{R}$,  if $0<|\mu_1|<|\mu_2|<1$, then
\begin{equation}\label{eq:optimal-increase-mu-NGM}
  \max\mathcal{T}_{\mu_1}\big(\beta_{\rm{opt}}(\mu_2)\big)= \min_{\beta}\max \mathcal{T}_{\mu_2}(\beta).
\end{equation}
\item  Assume $\mu \in \mathbb{R}$ and  $|\mu|\geq 1$. Then,
\begin{equation}\label{eq:lemma-NGM-div}
  \min_{\beta}\max \mathcal{T}_{\mu}(\beta) = 1.
\end{equation}
\end{enumerate}
\end{lemma}
\begin{proof}
We first consider $0< |\mu|\leq 1$. Denote $\Delta =(1+\beta)^2\mu^2-4\beta$. Note that $\lambda$ in \cref{eq:eig-set-q-prime-NGM} might be real or complex. Thus, we consider two cases as follows.\\
{\bf Complex eigenvalues:} If $\mu^2\leq \frac{4\beta}{(1+\beta)^2}$, then $\Delta \leq 0$.  Moreover, $|y(\beta)|^2=\beta $.   In order to minimize $ \displaystyle\max\mathcal{T}_{\mu}(\beta)$, we only need to solve
\begin{equation}\label{eq:complex-minimization-form-NGM}
  \min_{\beta} |y(\beta)|=\min_{\beta}\sqrt{\beta},
\end{equation}
under the condition that
\begin{equation*}
 s(\beta)=\beta^2+ (2-\frac{4}{\mu^2})\beta +1\leq 0.
\end{equation*}
Note that $(2-\frac{4}{\mu^2})^2-4>0$ since $\mu^2<1$. The two roots of $s(\beta)=0$ are
\begin{equation}\label{eq:complex-domain-ends-NGM}
 \beta_{N,1}(\mu)=\frac{1-\sqrt{1-\mu^2}}{1+\sqrt{1-\mu^2}},\,\, \beta_{N,2}(\mu)=\frac{1+\sqrt{1-\mu^2}}{1-\sqrt{1-\mu^2}}>1.
\end{equation}
So the solution of \cref{eq:complex-minimization-form-NGM} is
\begin{eqnarray*}
  \min_{\beta\in[\beta_{N,1},1]} \sqrt{\beta} &=&\sqrt{\beta_{N,1}(\mu)} \\
  &=&\sqrt{\frac{1-\sqrt{1-\mu^2}}{1+\sqrt{1-\mu^2}}}\\
  &=&\frac{1-\sqrt{1-\mu^2}}{|\mu|}\\
  &=&\frac{|\mu|}{1+\sqrt{1-\mu^2}}.
\end{eqnarray*}
{\bf Real eigenvalues:} When $\beta <\beta_{N,1}(\mu)$ or $\beta>\beta_{N,2}(\mu)$, the $y(\beta)$ are real. However, since $\beta_{N,2}(\mu)>1$, we have $\lambda>1$. Thus, we only consider $-1<\beta<\beta_{N,1}(\mu)$.  Note that
\begin{equation*}
  \max |y(\beta)| =\frac{(1+\beta)|\mu| + \sqrt{(1+\beta)^2\mu^2-4\beta}}{2}=:y_1(\beta).
\end{equation*}
We rewrite $y_1(\beta)$ as
\begin{equation*}
  y_1(\beta) =\frac{(1+\beta)|\mu| + |\mu| \sqrt{\beta^2+(2-\frac{4}{\mu^2})\beta+1}}{2} =\frac{(1+\beta)|\mu| + |\mu|\sqrt{s(\beta)}}{2}.
\end{equation*}
Note that $ y_1(\beta_{N,1}(\mu))=\frac{|\mu|}{1+\sqrt{1-\mu^2}}$.  We claim that $y_1(\beta)$ is decreasing over $(-1,\beta_{N,1}(\mu))$.  In fact, when $\beta \in (-1,\beta_{N,1}(\mu))\subseteq [-1,1)$, $1+\beta<2$. It follows that $1+\beta-\frac{2}{\mu^2}<0$. Note that
\begin{eqnarray*}
  y'_1(\beta) &=& \frac{|\mu|}{2}\Big(1+\frac{1+\beta-\frac{2}{\mu^2}}{\sqrt{(1+\beta)^2-\frac{4\beta}{\mu^2}}}\Big) \\
   &=&\frac{|\mu|}{2}\frac{\sqrt{(1+\beta-\frac{2}{\mu^2})^2+\frac{4}{\mu^2}(1-\frac{1}{\mu^2})}+1+\beta-\frac{2}{\mu^2}}{\sqrt{(1+\beta)^2-\frac{4\beta}{\mu^2}}}  \\
   &<&\frac{|\mu|}{2} \frac{\sqrt{(1+\beta-\frac{2}{\mu^2})^2}+1+\beta-\frac{2}{\mu^2}}{{\sqrt{(1+\beta)^2-\frac{4\beta}{\mu^2}}}}\\
   &=&0,
\end{eqnarray*}
where the last equality is due to $1+\beta-\frac{2}{\mu^2}<0$.

Combining the above two cases, we know that $|y_1(\beta)|$ is decreasing on $[-1,\beta_{N,1}(\mu)]$ and increasing on $[\beta_{N,1}(\mu),\beta_{N,2}(\mu)]$. Thus,
\begin{equation*}
\min_{\beta}\max\mathcal{T}_{\mu}(\beta)=y_1(\beta_{N,1}(\mu))=\frac{|\mu|}{1+\sqrt{1-\mu^2}}.
\end{equation*}
Next we prove the second statement. From \cref{eq:complex-domain-ends-NGM} and \yh{the fact that $\beta_{N,1}(\mu)\beta_{N,2}(\mu)=1$, we know that for any given $\mu_1$ and $\mu_2$ such that $|\mu_1|<|\mu_2|$, }
\begin{equation*}
  \beta_{N,1}(\mu_1) < \beta_{N,1}(\mu_2) <\beta_{N,2}(\mu_2)<\beta_{N,2}(\mu_1).
\end{equation*}
It follows that for $\beta\in [\beta_{N,1}(\mu_2),\beta_{N,2}(\mu_2)]$, the  $\lambda$ corresponding to $\mu_1$ in \cref{eq:AA-eig-form-NGM} are complex. Thus, \begin{equation*}
\max \mathcal{T}_{\mu_1}\big(\beta_{N,1}(\mu_2)\big) =\sqrt{\beta_{N,1}(\mu_2)} =\min_{\beta}\max \mathcal{T}_{\mu_2}(\beta),
\end{equation*}
which is the desired result.

Finally we consider $|\mu|> 1$. Recall $\Delta =(1+\beta)^2\mu^2-4\beta=\mu^2\big((1+\beta)^2-4\beta/\mu^2\big)$. We claim that $\Delta\geq 0$. For $\beta<0$, this is obvious. When $\beta\geq 0$, $\Delta = \mu^2\big((1-\beta)^2+4\beta(1-\frac{1}{\mu^2})\big)\geq0$. This means $\lambda$ is real.

When $-1<\beta<1$,
\begin{equation*}
  \max \mathcal{T}_{\mu}(\beta) =y_1(\beta).
\end{equation*}
Since $1-\frac{1}{\mu^2}>0$, we have
\begin{eqnarray*}
   y'_1(\beta) &=& \frac{|\mu|}{2}\Big(1+\frac{1+\beta-\frac{2}{\mu^2}}{\sqrt{(1+\beta)^2-\frac{4\beta}{\mu^2}}}\Big) \\
   &=&\frac{|\mu|}{2}\frac{\sqrt{(1+\beta-\frac{2}{\mu^2})^2+\frac{4}{\mu^2}(1-\frac{1}{\mu^2})}+1+\beta-\frac{2}{\mu^2}}{\sqrt{(1+\beta)^2-\frac{4\beta}{\mu^2}}}  \\
   &>&\frac{|\mu|}{2} \frac{\sqrt{(1+\beta-\frac{2}{\mu^2})^2}+1+\beta-\frac{2}{\mu^2}}{{\sqrt{(1+\beta)^2-\frac{4\beta}{\mu^2}}}}\\
   &\geq&0.
\end{eqnarray*}
This means that $g_1(\beta)$ is increasing over $[-1,1)$. Thus,
\begin{equation*}
\max_{\beta\in \mathbb{R}} \mathcal{T}_{\mu}(\beta) =\max_{\beta=-1} \mathcal{T}_{\mu}(\beta)=1.
\end{equation*}
Clearly, when $|\mu|=1, \max_{\beta\in \mathbb{R}} \mathcal{T}_{\mu}(\beta)=1$.
\end{proof}

From the proof of  \cref{minmax-real-nonnegative-case-NGM}, we know that when $\beta\in(-1,0]$, $\lambda$ is real and $\max\mathcal{T}_{\mu}(\beta)=\frac{(1+\beta)|\mu|+\sqrt{(1+\beta)^2\mu^2-4\beta}}{2}$ is a decreasing function of $\beta$ in $[-1,0]$. Thus, we will only consider $\beta\in[0,1]$ in sNGMRES-R(1)-SD.\\

\noindent
\textbf{Proof of \cref{thm:sNGMRE-R-SD}:}
\begin{proof}
From \cref{eq:optimal-increase-mu-NGM} in \cref{minmax-real-nonnegative-case-NGM}, we only need to minimize $\rho_{q'}$, the spectral radius of $q'$, since $\frac{|\mu|}{1+\sqrt{1-\mu^2}}$ is an increasing function of $|\mu|$.  Recall that $q'= I-\alpha H$, so $\displaystyle\min_{\alpha} \rho_{q'}(\alpha)=\frac{L-\ell}{L+\ell}$ if and only if $\alpha=\frac{2}{L+\ell}$. Thus, from   \cref{minmax-real-nonnegative-case-NGM}, we have
\begin{equation*}
  \rho^{*}_{\rm sNGMRES-R(1)-SD} =\min_{\alpha,\beta}\rho(T_N) =\frac{\rho_{q'}}{1+\sqrt{1-\rho^2_{q'}}},
\end{equation*}
if and only if
\begin{equation*}
   \alpha= \alpha_N^* =\frac{2}{L+\ell}, \, \beta=\beta_N^{*}=\frac{1-\sqrt{1-\rho_{q'}^2}}{1+\sqrt{1-\rho_{q'}^2}}.
\end{equation*}
Since $\rho_{q'} =\frac{L-\ell}{L+\ell}$,
\begin{eqnarray*}
   \frac{\rho_{q'}}{1+\sqrt{1-\rho^2_{q'}}}&<& \frac{\rho_{q'}}{1+\sqrt{1-\rho_{q'}}} \\
    &=&1-\sqrt{1-\rho_{q'}} \\
    &=& 1-\sqrt{\frac{2\ell}{L+\ell}}\\
    &<& 1-\sqrt{\frac{4\ell}{3L+\ell}}.
\end{eqnarray*}
Note also that
\begin{eqnarray*}
 \frac{\rho_{q'}}{1+\sqrt{1-\rho^2_{q'}}}&=&\frac{(L-\ell)/(L+\ell)}{1+\sqrt{1-\big((L-\ell)/(L+\ell)\big)^2}}\\
 &=&\frac{L-\ell}{L+\ell+2\sqrt{\ell L}}\\
 &=&\frac{\bar{\kappa}-1}{\bar{\kappa}+1+2\sqrt{\bar{\kappa}}}\\
 &=&\frac{\sqrt{\bar{\kappa}}-1}{\sqrt{\bar{\kappa}} +1}.
\end{eqnarray*}
Since $\rho^{*} = \sqrt{\beta^{*}_{N}}$, $\beta^{*}_{N} =\Big(\frac{\sqrt{\bar{\kappa}}-1}{\sqrt{\bar{\kappa}}+1}\Big)^2$.
\end{proof}
\begin{remark}
In \cite{hong2018accelerating} the asymptotic convergence factor is computed for a two-grid version of the Sequential Subspace Optimization method to accelerate multigrid optimization (SESOP-MG) with window size 1 for quadratic objectives. In that work, the convergence factor is determined by analyzing a $2 \times 2$ coarse-grid correction block matrix similar to \cref{eq:matrix-step1-NGM}.
The same optimal convergence parameters are obtained using a different proof technique. Indeed, the 1-step SESOP acceleration method is related to sNGMRES-R(1).
\end{remark}

\begin{remark}
In \cite{MR703121}, upper and lower bounds on the convergence factor of stationary $k$-step iterative methods for linear systems are discussed. In the case of $k=2$, the optimal result using the theory of Euler methods is the same as what we presented here after a transformation of the parameters. Our optimal results can be retrieved with some effort from these results in \cite{MR703121}, but
our proof is substantially shorter and more elementary, and provides more specific insight in the variation of the parameters as $\beta$ varies.
\end{remark}
\subsection{Lower and upper bounds on the optimal asymptotic convergence factor for sNGMRES-R(1) \yh{applied to fixed-point methods with complex spectrum} -- extension of Section \ref{subsec:optimal}}
\label{subsec:sNGMRES-R(1)-ALS-bounds}
\ \\
\yh{Here we consider sNGMRES-R(1) acceleration in the case that $q'$ has complex specturm, e.g., ALS applied to canonical tensor decomposition.}

First, using \cref{minmax-real-nonnegative-case-NGM}, we can obtain a lower bound on the optimal convergence
factor of \yh{sNGMRES-R(1)}.

\begin{theorem}[lower bound for fixed-point methods with complex Jacobian spectrum]\label{thm:similar-lower-bound-NGMRES-ALS}
 \yh{
Let $x^*$ be a fixed point of iteration (\ref{eq:fixed-point}).
Let the spectral radius of $q'(x^*)$ be $\rho_{q'}$. Assume that there exists a real eigenvalue $\mu$ of $q'(x^*)$ such that $\rho_{q'}=|\mu|$. Then the optimal asymptotic convergence factor of the sNGMRES-R(1) iteration \cref{system-form-NGM} with Jacobian matrix $T_N$ defined in (\ref{eq:matrix-step1-NGM}) is bounded below by}
\begin{equation}\label{eq:similar-lower-bound-NGMRES-ALS}
 \min_{\beta \in\mathbb{R}} \rho(T_N(x^{*};\beta)) \geq  \frac{\rho_{q'}}{1+\sqrt{1-\rho^2_{q'}}}=:\rho_{p,N},
\end{equation}
and if the equality holds, then \yh{the unique optimal $\beta$ is given by}
\begin{equation}\label{opt-beta-real-NGMRES-ALS}
  \beta_{\rm opt}=\frac{1-\sqrt{1-\rho^2_{q'}}}{1+\sqrt{1-\rho^2_{q'}}}.
\end{equation}
\end{theorem}

\begin{proof}
Since $\rho_{q'} \in \sigma(q')$,
\begin{equation*}
 \min_{\beta \in\mathbb{R}} \rho(T_N(x^{*};\beta))= \min_{\beta \in\mathbb{R}}\max\big(\cup_{\mu\in\sigma(q')} \mathcal{T}_{\mu}(\beta)\big)\geq \min_{\beta \in\mathbb{R}}\max\mathcal{T}_{\rho_{q'}}(\beta).
\end{equation*}
Based on $0<\rho_{q'}<1$ and \cref{minmax-real-nonnegative-case-NGM}, we have
\begin{equation*}
   \min_{\beta \in\mathbb{R}}\max\mathcal{T}_{\rho_{q'}}(\beta)=\frac{\rho_{q'}}{1+\sqrt{1-\rho^2_{q'}}},
\end{equation*}
which  is the desired result.
\end{proof}

In \yh{the case of ALS for canonical tensor decomposition,} our numerical tests in Section \ref{sec:num-sec} indicate
that $\rho_{q'}=\mu$ for sAA(1)-ALS, supporting \cref{conjec1-sAA-ALS}, but we do
find numerically that the inequality \cref{eq:similar-lower-bound-NGMRES-ALS} is non-strict for ALS, so there is no conjecture equivalent to
\cref{conjec2-sAA-ALS} for sNGMRES-R(1)-ALS.

So we pursue some further lower and upper bounds on the optimal convergence factor for sNGMRES-R(1) that may be more useful
for the case of fixed-point methods with complex Jacobian spectrum. Assume that the spectrum of $q'$ is bounded by a rectangle
\begin{equation}\label{eq:rectange}
R_{r_1,r_2}=\Big\{z \in\mathbb{C}| -r_1\leq {\rm Re}(z)\leq r_1, -r_2\leq {\rm Im}(z)\leq r_2\Big\},
\end{equation}
where $0\leq r_1,r_2<1$. In the case of ALS, the requirement that $r_1,r_2<1$ is reasonable, since ALS is convergent, that is, $\rho(q')<1$, see  \cite{Uschmajew2012}. In the literature, there is some research on the spectrum $\sigma(q')$ for $k$-step stationary  iterative methods, see \cite{MR703121,Varga1986}. We take advantage of the existing results there and  apply them to our problems. \yh{The following result is true for the
general case where the eigenvalues of $q'(x^*)$ are bounded by \cref{eq:rectange}.}
\begin{theorem}\label{thm:lowe-upper-bound}
The optimal asymptotic convergence factor $\rho$ obtained for sNGMRES-R(1) \yh{applied to a fixed-point method with Jacobian spectrum
satisfying (\ref{eq:rectange})} is bounded by
\begin{itemize}
\item If $r_1>r_2$, then
\begin{equation*}\label{eq:NR-lower-upper-bound-A}
\frac{2\eta_1}{r_2\eta_0-\sqrt{(r_2 \eta_0)^2-4\eta_1}}=\delta_1  <\rho< \delta_2(a)=\frac{2\tau_1}{-a\tau_0+\sqrt{(a \tau_0)^2+4\tau_1}},
\end{equation*}
where
\begin{eqnarray*}
\eta_0&=&\frac{2}{1+\sqrt{1-r_1^2+r_2^2}},\quad \eta_1=1-\eta_0,\\
\tau_0&=&\frac{2}{1+\sqrt{1-a^2+b^2}},  \quad \tau_1=1-\tau_0,
\end{eqnarray*}
with  $b=\frac{ar_2}{\sqrt{a^2-r_1^2}},$ and $r_1< a<1$.
\item If $r_1<r_2$, then
\begin{equation*}\label{eq:NR-lower-upper-bound-B}
\frac{2\eta_1}{r_2\eta_0-\sqrt{(r_2 \eta_0)^2-4\eta_1}}=\delta_1 <\rho< \delta_2(a)=\frac{2\tau_1}{b\tau_0-\sqrt{(b \tau_0)^2-4\tau_1}},
\end{equation*}
where
\begin{eqnarray*}
\eta_0&=&\frac{2}{1+\sqrt{1-r_1^2+r_2^2}},\quad \eta_1=1-\eta_0,\\
\tau_0&=&\frac{2}{1+\sqrt{1-a^2+b^2}},  \quad \tau_1=1-\tau_0,
\end{eqnarray*}
with  $b=\frac{ar_2}{\sqrt{a^2-r_1^2}},$ and $r_2< a<1$.
\item If $r_1=r_2$, then
\begin{equation*}
 \frac{r_1}{1+\sqrt{1-r^2_1}} =\delta_1< \rho.
\end{equation*}

\end{itemize}
\end{theorem}
\begin{proof}
The above results are based on \cite{MR703121}, especially the discussion in \cite[{\rm Section} 6]{MR703121}, in which the authors discuss $k$=2, and examples 1,\,2,\,7  in \cite[{\rm Section} 9]{MR703121} and setting $\mu_0=1+\beta$, $\mu_1=0$, and $\mu_2=-\beta$. We omit the details here.
\end{proof}
Our numerical results in Section \ref{subsubsec:sNGMRES(1)} confirm that the lower bounds of \cref{thm:lowe-upper-bound} are substantially tighter than the lower bound of \cref{thm:similar-lower-bound-NGMRES-ALS}.

\begin{remark}
Note that $\delta_1$ is a function of $r_1$ and $r_2$. It can be shown that $\delta_1$ is an increasing function of both $r_1$ and $r_2$.
Note that the upper bound $\delta_2$ is a function of $a$. We can numerically optimize $a$ to obtain a sharp bound, see
\cref{table:Lower-upper-NGMRES} in Section \ref{subsubsec:sNGMRES(1)}.
Although  \cref{thm:lowe-upper-bound} \yh{offers convergence factor bounds that can be applied to sNGMRES-R(1)-ALS,}  it does not tell us how to choose the parameter $\beta$ to achieve these bounds. This remains an open question.
\end{remark}

\subsection{Relation between sNGMRES and sNGMRES-R -- extension of Section \ref{sec:convergence-NGMRES}}\label{sub:NGMRES-NGMRES-R}
\ \\
Here we explain how the results described in Section \ref{sec:convergence-NGMRES} for the reduced sNGMRES-R iteration
of Eq.\ \cref{eq:sNGMRES-R} translate to sNGMRES of Eq.\ \cref{eq:sNGMRES}.

First, recall sNGMRES(0)-SD, which reads
\begin{eqnarray*}
  x_{k+1} &=&q(x_k) + \beta(q(x_k)-x_k)\\
    &=&  x_k-\alpha_0 \nabla f(x_k) +\beta(x_k-\alpha_0 \nabla f(x_k)-x_k)\\
    &=&x_k-\alpha_0(1+\beta)\nabla f(x_k)\\
    &=&x_k -\alpha\nabla f(x_k),
\end{eqnarray*}
where $\alpha =\alpha_0(1+\beta)$. It can be seen that NGMRES(0)-SD is a version of SD, with a modified step length.

Next, when $m=1$, sNGMRES(1)-SD reads
\begin{align}
  x_{k+1} &= q(x_k) +\beta_1(q(x_k) -x_k) +\beta_2(q(x_k)-x_{k-1})\nonumber\\
  &=(1+\beta_1+\beta_2)q(x_k) -\beta_1 x_k -\beta_2 x_{k-1}.\label{eq:NGMRES(1)-special}
\end{align}
If we let $\beta_1=0$, then \cref{eq:NGMRES(1)-special} is the reduced sNGMRES-R(1)-SD  method of \cref{eq:M-NGMRES(1)}. This means that the optimal convergence factor for sNGMRES(1)-SD cannot be worse than that of sNGMRES-R(1)-SD. However, we may wonder whether we can optimize
parameters in \cref{eq:NGMRES(1)-special} to obtain a better convergence factor than for sNGMRES-R(1)-SD.  We consider $T_N(q'(x^*))$ for \cref{eq:NGMRES(1)-special}, where
\begin{equation*}
  T_N =\begin{bmatrix}
  (1+\beta_1+\beta_2)q'-\beta_1I & -\beta_2 I\\
   I  & 0
  \end{bmatrix}.
\end{equation*}
  The eigenvalues of $T_N$, denoted by $\lambda$, satisfy
  \begin{equation*}
  \lambda^2-\big((1+\beta_1+\beta_2)\mu-\beta_1\big)\lambda +\beta_2 =0.
  \end{equation*}
We claim that $\min\rho(T_N)$ is the same as the minimum of  sNGMRES-R(1)-SD. Recall  that $\mu  =1-\alpha_0\xi \in \sigma(q'_{SD})$, where $\xi$ is an eigenvalue of $H$. Then,
\begin{eqnarray*}
&&\lambda^2-\big((1+\beta_1+\beta_2)\mu-\beta_1\big)\lambda +\beta_2\\
&=&\lambda^2-\big((1+\beta_1+\beta_2)-\alpha_0(1+\beta_1+\beta_2)\xi-\beta_1\big)\lambda +\beta_2\\
&=&\lambda^2-\big(1+\beta_2-\alpha_0(1+\beta_1+\beta_2)\xi\big)\lambda +\beta_2\\
&=&\lambda^2-(1+\beta_2)\Big(1-\alpha_0\frac{1+\beta_1+\beta_2}{1+\beta_2}\xi\Big)\lambda+\beta_2.
\end{eqnarray*}
If we let $\beta_2=\beta$ and $\alpha_0\frac{1+\beta_1+\beta_2}{1+\beta_2}=\alpha$, then
\begin{eqnarray*}
 && \lambda^2-(1+\beta_2)\Big(1-\alpha_0\frac{1+\beta_1+\beta_2}{1+\beta_2}\xi\Big)\lambda+\beta_2 \\
  &=& \lambda^2-(1+\beta)\Big(1-\alpha\xi\Big)\lambda+\beta_2\\
  &=& \lambda^2-(1+\beta)\mu\lambda+\beta,
\end{eqnarray*}
which is the same as \cref{eq:AA-eig-form-NGM}. This indicates that the optimized result of $\min(\rho(T_N))$ is the same as for sNGMRES-R(1)-SD. In conclusion, for SD, the optimal convergence factors for sNGMRES-R(1)-SD and sNGMRES(1)-SD are the
same.

Next, we move on to the discussion of sNGMRES-ALS. Let $1-\varsigma \in \sigma(q'_{ALS})$.

When $m=0$, sNGMRES(0)-ALS is
\begin{equation*}
  x_{k+1} = q(x_k)+\beta(q(x_k)-x_k).
\end{equation*}
The eigenvalues of $T_N$ satisfy
\begin{equation*}
  \lambda =1 - (1+\beta)\varsigma,
\end{equation*}
so we can interpret sNGMRES(0)-ALS as a damped version of ALS with weight $1+\beta$.

When $m=1$, similar as for SD, the eigenvalues of $T_N$ satisfy
\begin{equation}\label{eq:ALS-CNGMRES}
\lambda^2-(1+\beta_2)\Big(1-\frac{1+\beta_1+\beta_2}{1+\beta_2}\varsigma\Big)\lambda+\beta_2=0.
\end{equation}
Compared with $\lambda^2-(1+\beta)(1-\varsigma)\lambda+\beta$, the roots of \cref{eq:ALS-CNGMRES} can be treated as the eigenvalues of  damped sNMGRES-R(1)-ALS with damping parameter, $\frac{1+\beta_1+\beta_2}{1+\beta_2}$. Thus, the performance of sNGMRES(1)-ALS will not be worse than that of sNGMRES-R(1)-ALS, if the damping parameter is chosen optimally.

For general $m$, it is easy to see that the performance of optimally tuned sNGMRES($m$)-ALS cannot be worse than sNGMRES-R($m$)-ALS, since sNGMRES($m$)-ALS has one more free parameter than sNGMRES-R($m$)-ALS, and if we set this extra parameter to zero, then sNGMRES($m$)-ALS reduces to sNGMRES-R($m$)-ALS.

\subsection{Comparison of optimal asymptotic convergence factors for accelerated SD in \cref{comparsion-ASD-vs-NGMRESSD} of Section \ref{subsec:optfac}}

\begin{figure}[H]
\centering
\includegraphics[width=0.6\linewidth]{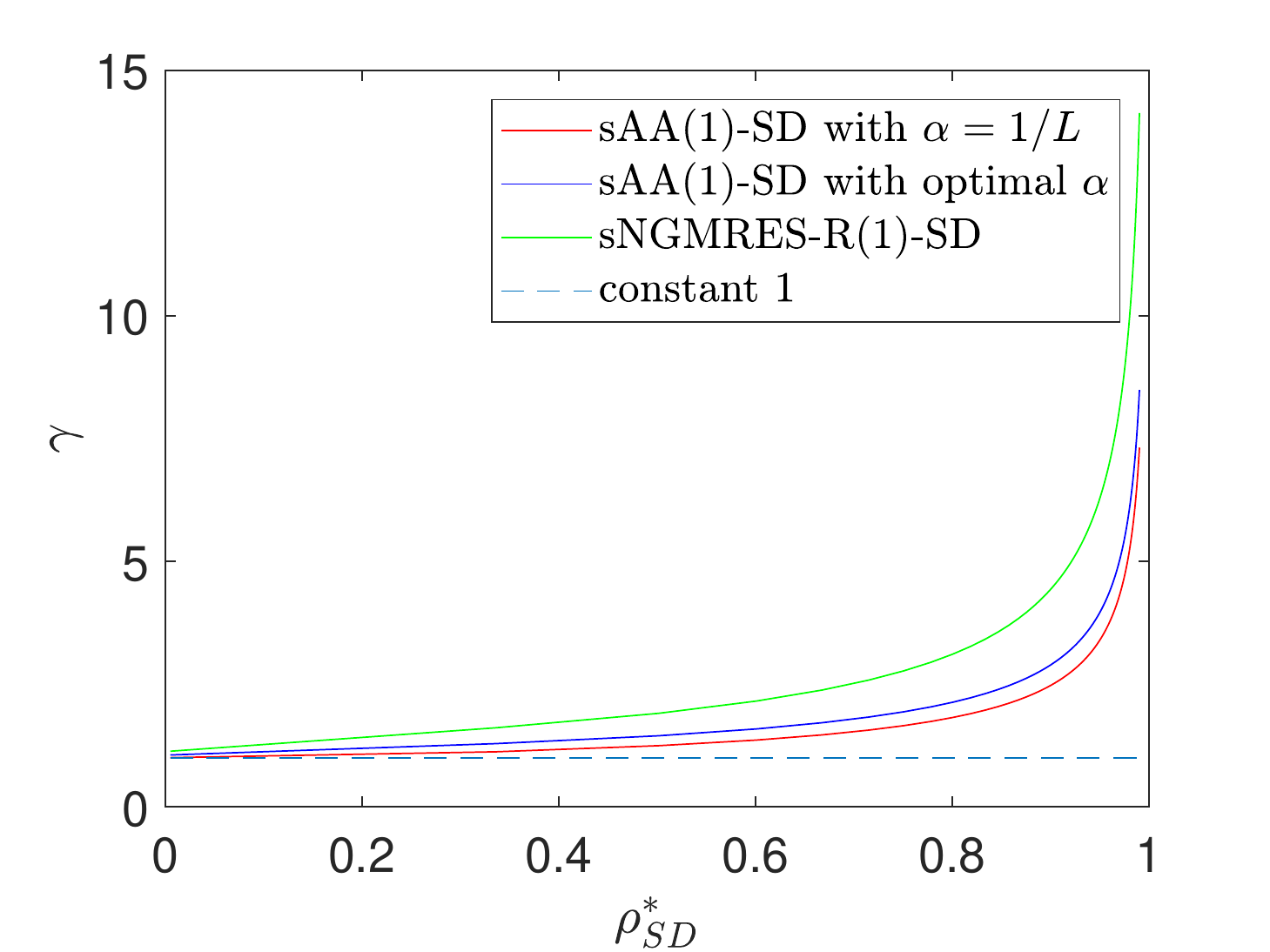}
\caption{Acceleration ratio compared to steepest descent method} \label{acce-ratio-plot}
\end{figure}

\ \\
For the results in \cref{comparsion-ASD-vs-NGMRESSD} on accelerating SD, we define the acceleration ratio compared with SD with optimal step length as
\begin{equation*}
  \gamma =\frac{{\rm log}(\rho^*)}{{\rm log}(\rho^*_{SD})},
\end{equation*}
where $\rho^*$ is the optimal convergence factor for the other methods as in  \cref{comparsion-ASD-vs-NGMRESSD}.
\cref{acce-ratio-plot} shows the acceleration ratios for sAA(1)-SD with $\alpha=\frac{1}{L}$ and optimal $\alpha$, and sNGMRES-R(1)-SD. Note that, as $\rho^{*}_{SD}$ approaches 1 and the problem becomes more ill-conditioned and harder to solve, the acceleration methods greatly improve the performance.

\subsection{sAA(2) and sNGMRES-R(2)  -- extension of Section \ref{sec:convergence-NGMRES}}
\ \\
Here we extend the theoretical results on stationary AA and NGMRES methods with optimal asymptotic convergence factors from
window size $m=1$ to $m=2$.

For sAA(2) and sNGMRES-R(2), it is more complicated to analyze the spectral radius of $T$ and $T_N$ in
\cref{eq:matrix-step1,eq:matrix-step1-NGM}
than for $m=1$, since the eigenvalues of $T$ and $T_N$ are the roots of polynomials of degree 3. Thus, we use brute-force search for $\beta_1$ and $\beta_2$ to find a good approximation to the optimal spectral radius of $T$ and $T_N$ as shown in \cref{convergence-table-AA-ALS}.
The table computes optimal parameters for AA($m$)-ALS and sNGMRES-R($m$)-ALS
for $m=1$ and $m=2$, using brute-force minimization of the spectral radius.
We also perform this brute-force optimization for sNGMRES($m$)-ALS with $m=0$ and $m=1$ (corresponding to
1 and 2 coefficients, as for the other two methods).
The optimal brute-force parameters $\beta_{bf}$ are listed in the table.
The search space for the $\beta_{bf}$ parameters was $-1:0.05:1.$.
The results show that sAA(1)-ALS outperforms sNGMRES-R(1)-ALS and sAA(2)-ALS outperforms sNGMRES-R(2)-ALS.
Also, sNGMRES(1)-ALS performs better than sNGMRES-R(1)-ALS, which is as expected, since it has one more free parameter.
Comparing with the analytically optimal parameters in \cref{convergence-table-analytical-para}, we can see that the brute-force result for sAA(1)-ALS in \cref{convergence-table-AA-ALS} is very close to the analytical result of sAA(1)-ALS using $\beta$ of \cref{opt-beta-real}, in accordance with \cref{conjec2-sAA-ALS}.

\begin{remark}
Since there are $2r=6$ eigenvalues 0  in the Hessian, $T$ has 6 eigenvalues of value 1. Thus, when optimizing the spectral radius of $T$, we minimize the  modulus of the first $n-2r$ eigenvalues of $T$ (excluding the $2r$ eigenvalues of value 1), where $T\in \mathbb{R}^{n\times n}$.
\end{remark}

\begin{table}
 \caption{Asymptotic convergence results for  sAA(m)-ALS, sNGMRES-R(m)-ALS and sNGMRES(m)-ALS for different $c$ using brute-force search for optimal parameters $\beta_{bf}$.}
\centering
\begin{tabular}{|l|l|l|l|l|}
\hline
& $c$  \quad       \                      &0.5        &0.7      &0.9   \\
\hline
\hline
\multicolumn{5} {|c|} {\footnotesize{sAA(m)-ALS}} \\
\hline
$m=1$ & $\rho$                   &0.4543               &0.7057         &0.9180   \\
& $\beta_{bf}$                        &0.30                 &0.55           &0.85   \\
\hline
$m=2$ & $\rho$                   &0.4257                &0.6784                  &0.9129   \\
& $\beta_{bf}$                       &(0.45, -0.05)     &(0.80, -0.10)       &(0.95, -0.05)   \\
\hline
\hline
\multicolumn{5} {|c|} {\footnotesize{sNGMRES-R(m)-ALS}} \\
\hline
$m=1$ & $\rho$       &0.4947          &0.7646    &0.9593  \\
& $\beta_{bf}$            &0.15     &0.35       &0.65   \\
\hline
$m=2$ & $\rho$     &0.4947         & 0.7198    &0.9208  \\
& $\beta_{bf}$         &(0.15, 0)      &(0.20, 0.10)      &(0.35, 0.20)  \\
\hline
\hline
\multicolumn{5} {|c|} {\footnotesize{sNGMRES(m)-ALS}} \\
\hline
$m=0$ & $\rho$      &0.5631       &0.8460         &0.9851   \\
& $\beta_{bf}$            &0.40           &0.65            &0.75   \\
\hline
$m=1$ & $\rho$       &0.4434            &0.6994           &0.9573   \\
& $\beta_{bf}$            &(0.30, 0.10) &(0.80, 0.25) &(0.75, 0.55)   \\
\hline
\end{tabular}\label{convergence-table-AA-ALS}
\end{table}

\begin{table}
 \caption{Asymptotic convergence results for different $c$ using the optimal parameters ($\alpha,\beta$) from Table \cref{comparsion-ASD-vs-NGMRESSD}.}
\centering
\begin{tabular}{|l|c|c|c|c|c|}
\hline
$c$                              &0.5         &0.7       &0.9   \\
\hline
$\bar{\kappa}$                   &22.76            &123.90             &3837.90  \\
\hline
\hline
$\rho_{\rm SD}$                  &0.9158          &0.9838             &0.9995        \\

\hline
$\rho_{\rm sAA(1)-SD}$ with $\alpha=1/L$        &0.7904         &0.9102            &0.9839   \\
\hline
$\rho_{\rm sAA(1)-SD}$ with optimal $\alpha$        &0.7597         &0.8964             & 0.9814   \\
\hline
$\rho_{\rm sNGMRES-R(1)-SD}$           &0.6543          &0.8351           & 0.9682     \\
\hline
\hline
$\rho_{\rm ALS}$                 &0.6879          &0.9055             &0.9915  \\
\hline
$\rho_{\rm sAA(1)-ALS}$  with  $\beta$ from \cref{opt-beta-real}      &0.4413          &0.6926            &0.9078     \\
\hline
\end{tabular}\label{convergence-table-analytical-para}
\end{table}

\subsection{Extending the numerical results of Section \ref{subsec:spectral}: sNGMRES-R(1) acceleration and ill-conditioned problems}
\label{subsubsec:sNGMRES(1)}
\ \\
Here we extend the numerical results of Section \ref{subsec:spectral} to sNGMRES-R(1) acceleration and ill-conditioned problems,
comparing with the theoretical results from Sections \ref{sec:convergence-Anderson} and \ref{sec:convergence-NGMRES}.

First, \cref{convergence-table-analytical-para} summarizes,
for increasingly ill-conditioned problems with $c$ ranging from 0.5 to 0.9,
the computed asymptotic convergence factors
of SD and ALS, and the optimal theoretical convergence factors for sAA(1) and sNGMRES-R(1)
acceleration according to \cref{thm:psd-SD-case}, \cref{Thm:positive-negative-AA-SD},
\cref{thm:sNGMRE-R-SD}, and \cref{thm:lower-bound-ALS} and \cref{conjec2-sAA-ALS},
using the optimal parameters ($\alpha,\beta$) from \cref{comparsion-ASD-vs-NGMRESSD}.
As $c$ and $\bar{\kappa}$ increase, the convergence of SD and ALS deteriorate, but
sAA(1) and sNGMRES(1) accelerate them effectively in accordance with the theoretical results.
It is clear that ALS is a much better nonlinear preconditioner than SD, consistent
with the observations in \cref{eigs-sAA-c05-plot,eigs-Hessian-c05-plot,eigs-Hessian-full-c05-plot}.

\begin{figure}[tbhp]
\centering
\includegraphics[width=6.cm]{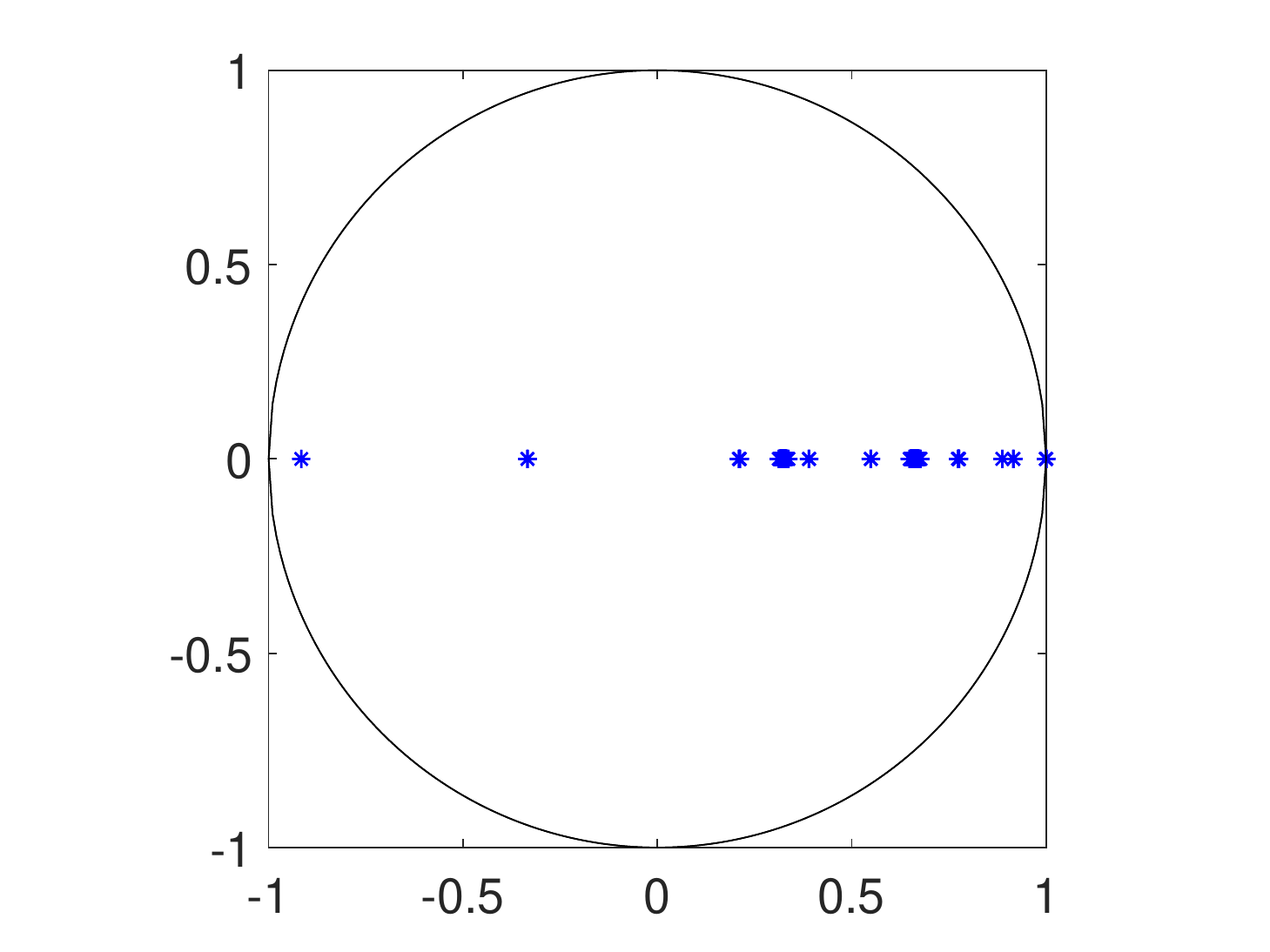}
\includegraphics[width=6.cm]{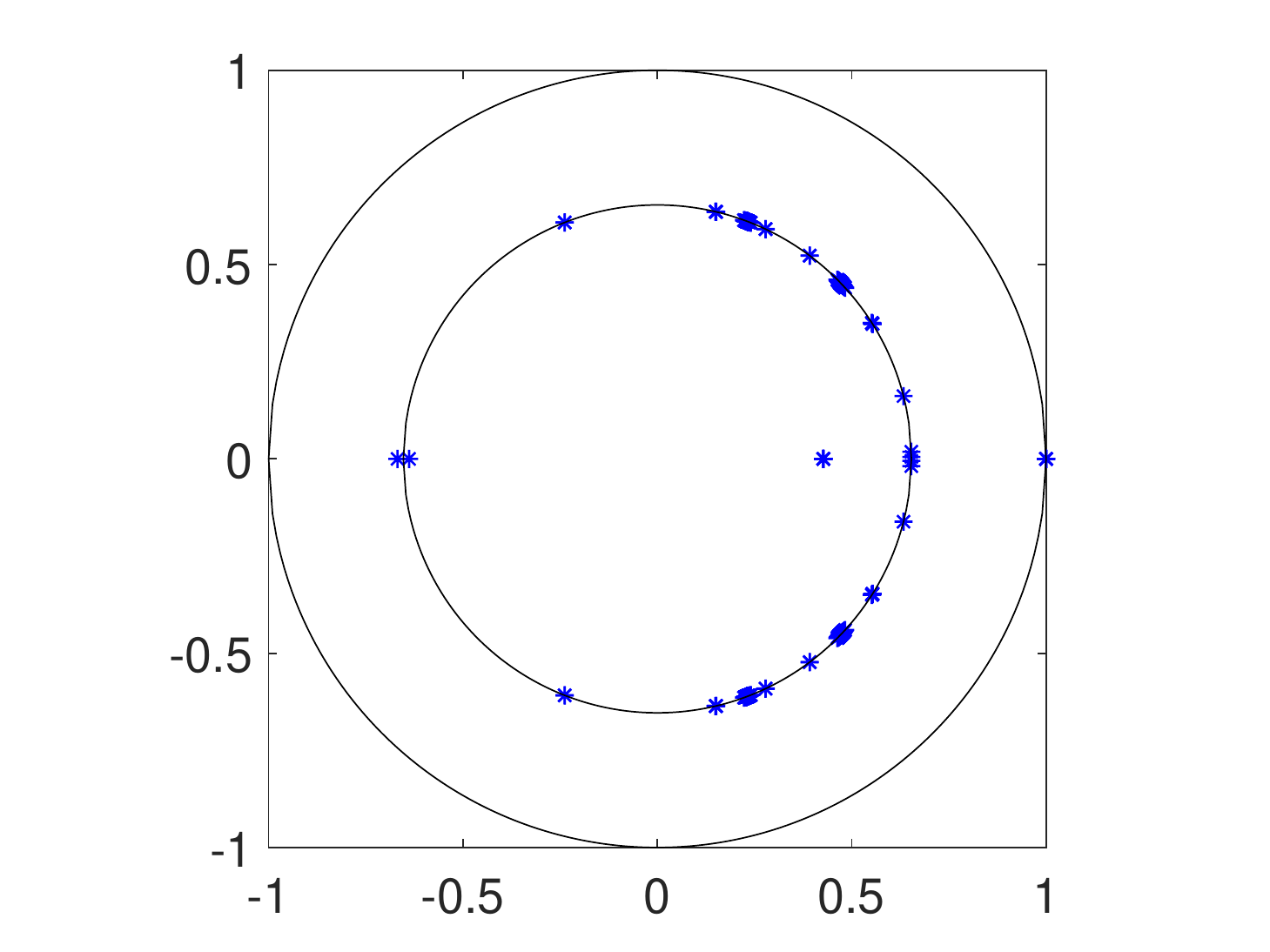}
\includegraphics[width=6cm]{eig-dis-q-prime-ALS-c05}
\includegraphics[width=6cm]{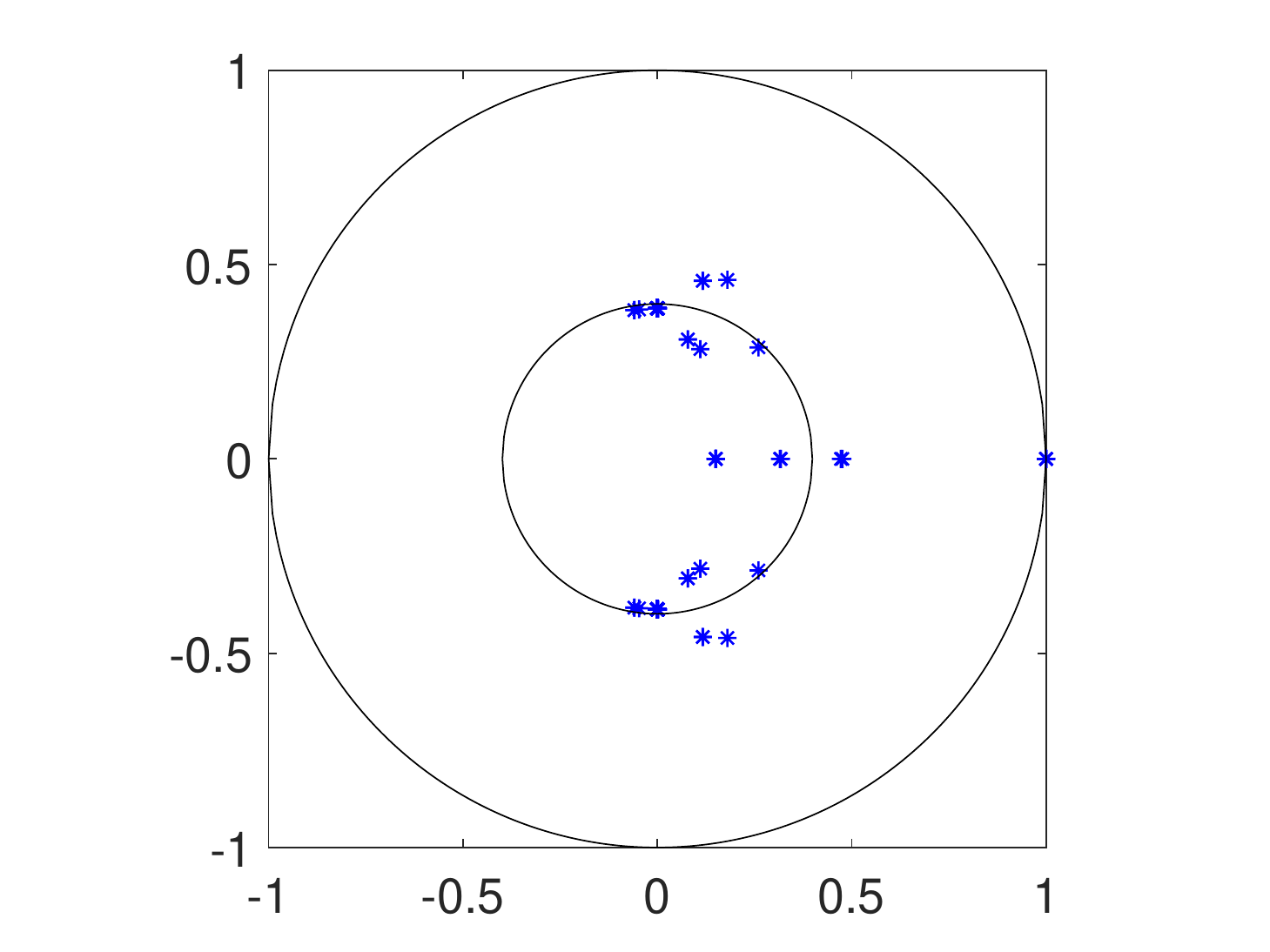}
\caption{
Eigenvalue distributions at $x^*$ for sNGMRES-R(1) acceleration of steepest descent (top row) and ALS (bottom row) for a tensor problem with $c=0.5$.
(top left) Eigenvalues of $q'_{SD}$ with $\alpha$ from \cref{eq:SD-opt-alpha-beta-NGM}; $\rho(q'_{SD})=0.916$.
(top right) Eigenvalues of $T$ for sNGMRES-R(1)-SD with the optimal parameters from \cref{thm:sNGMRE-R-SD}; $\rho(T)=0.654$. The radius of the inner circle is $\rho^{*}_{\rm sNGMRES-R(1)-SD}$ from \cref{eq:improved-SD-opt-cs-NGM}.
(bottom left) Eigenvalues of $q'_{ALS}$; $\rho(q')=0.688$.
(bottom right) Eigenvalues of $T$ for sNGMRES-R(1)-ALS using the $\beta$ obtained by using brute-force optimization in \cref{convergence-table-AA-ALS}; $\rho(T)=0.4947$, and the radius of the inner circle is $\rho_{p,N}=0.3986$ from \cref{eq:similar-lower-bound-NGMRES-ALS}.
} \label{eigs-sNGMRES-c05-plot}
\end{figure}

Next, we consider numerical results for sNGMRES-R(1) acceleration.
\cref{eigs-sNGMRES-c05-plot} shows and quantifies how the same kind of asymptotic convergence
acceleration as for sAA(1) in \cref{eigs-sAA-c05-plot} happens for sNGMRES-R(1) acceleration of SD and ALS,
for a tensor problem with $c=0.5$.
The top row shows how optimal sNGMRES-R(1) acceleration of SD reduces the asymptotic convergence
factor from $\rho(q'_{SD})=0.916$ to $\rho(T)=0.654$.
For convenience, we drop the subscript  $N$ in $T_N$ for the NGMRES iterations in the rest of this paper.
This optimal $\rho(T)$ can be computed as a function of the condition number of $H$
using the theoretical result from \cref{thm:sNGMRE-R-SD}.
The bottom row applies optimal sNGMRES-R(1) acceleration to ALS, reducing the convergence
factor from $\rho(q'_{ALS})=0.688$ to $\rho(T)=0.4947$. Lower and upper bounds for the
optimal $\rho(T)=0.4947$ can be computed from the theoretical results in \cref{thm:lowe-upper-bound},
see \cref{table:Lower-upper-NGMRES}.
We note that in our test problems, $r_1>r_2$ in  \cref{thm:lowe-upper-bound}.
We use a brute-force approach to optimize $a$ in $\delta_2$.
From  \cref{table:Lower-upper-NGMRES}, we see that $\delta_1$ gives a sharp bound, comparing with the optimal results $\rho$ (see  \cref{convergence-table-AA-ALS}) obtained by minimizing the spectral radius of $T$ using the brute-force approach; $\delta_2$ also
gives a useful upper bound.

\begin{table}
 \caption{Lower and upper bounds from \cref{thm:lowe-upper-bound} on the asymptotic convergence factors for sNGMRES-R(1)-ALS.}
\centering
\begin{tabular}{|l|c|c|c|}
\hline
 $c$                 &$\delta_1$      &$\rho$   &$\delta_2$   \\
\hline
 0.5                 &0.4839                  &0.4947    &0.6533 \\
\hline
0.7                  &0.7355                  &0.7647    & 0.9120\\
\hline
0.9                 &0.9548                  &0.9593    &0.9973 \\
\hline
\end{tabular}\label{table:Lower-upper-NGMRES}
\end{table}

\begin{figure}
\centering
\includegraphics[width=4.2cm]{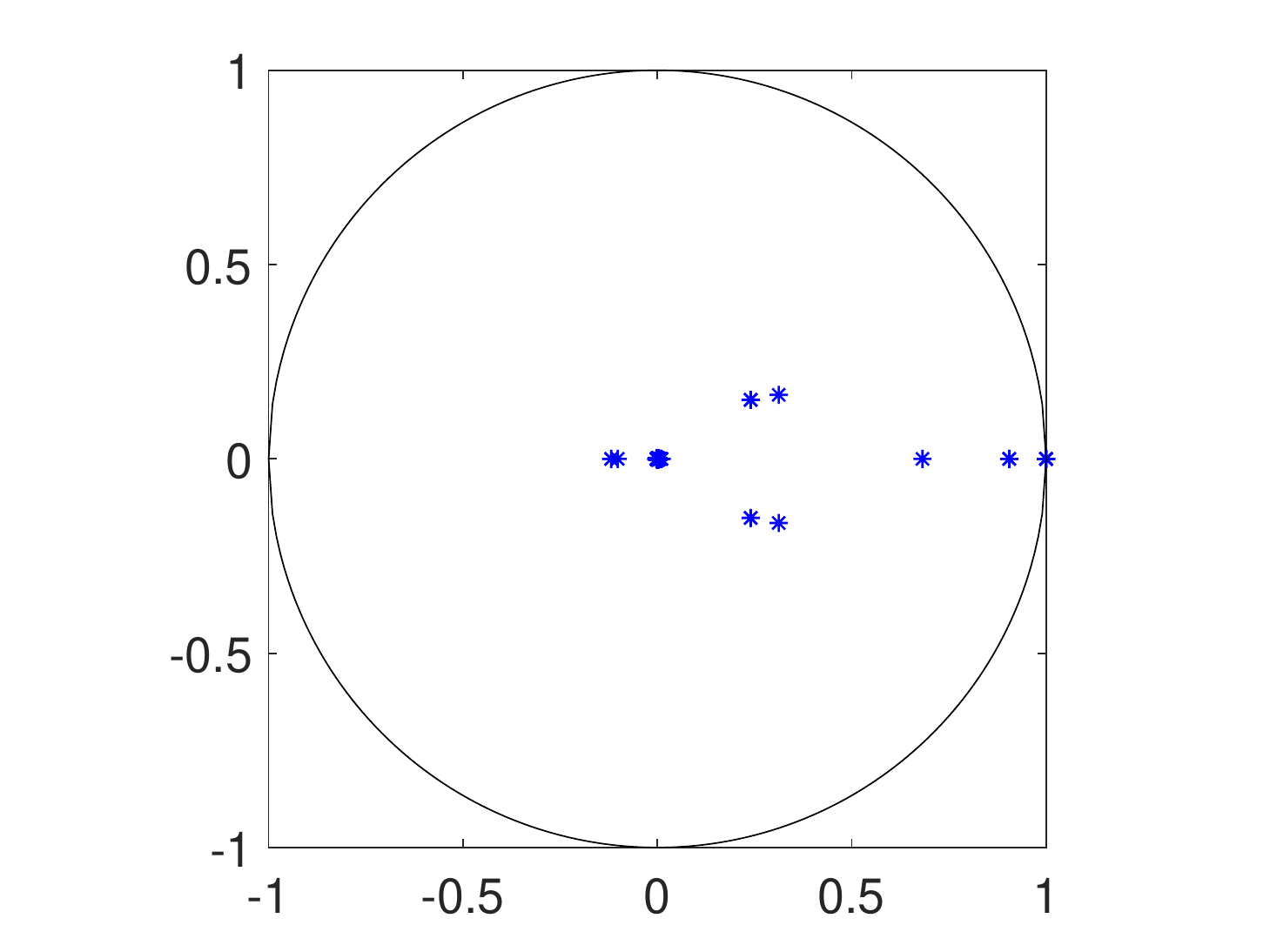}
\includegraphics[width=4.2cm]{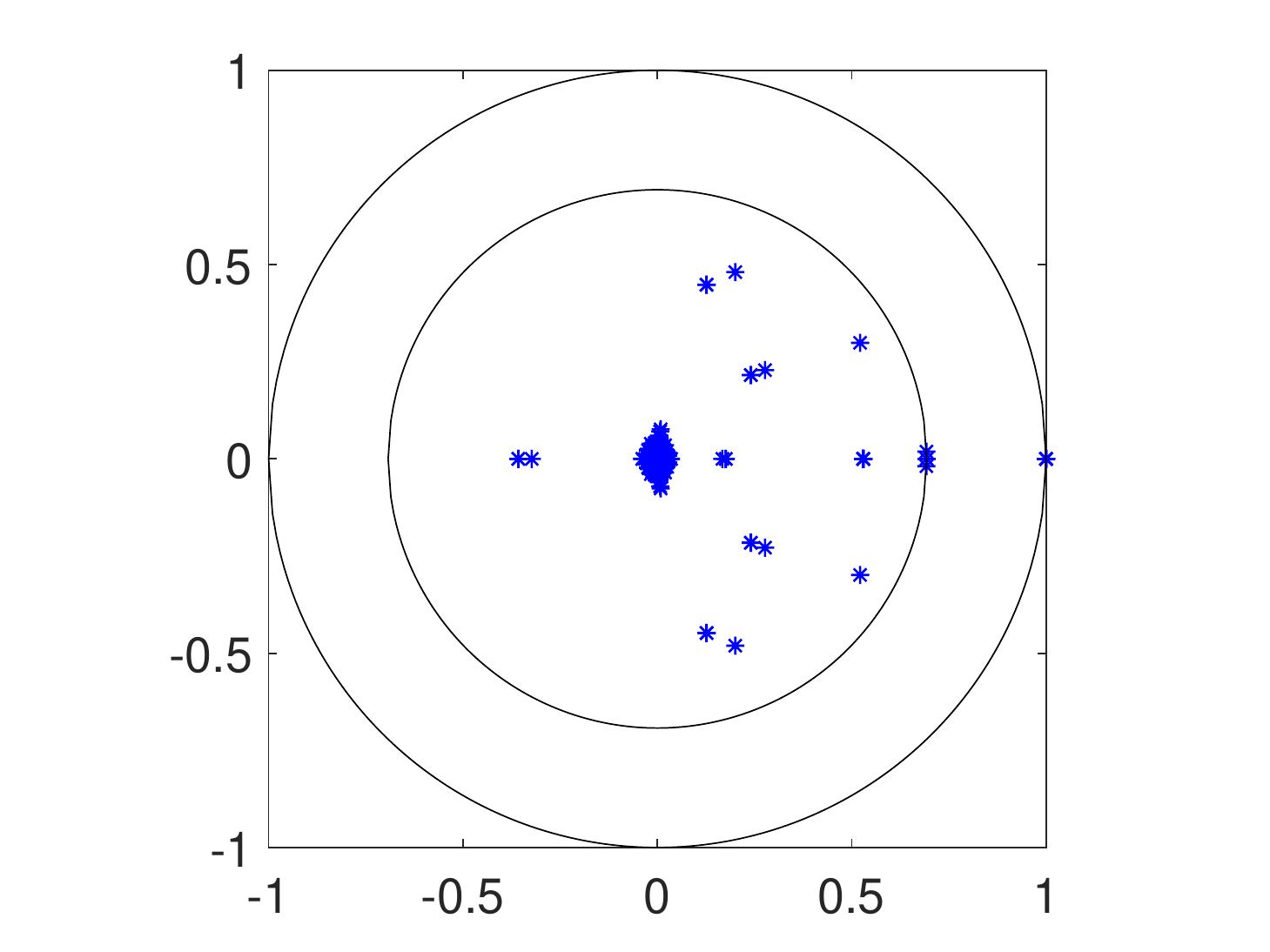}
\includegraphics[width=4.2cm]{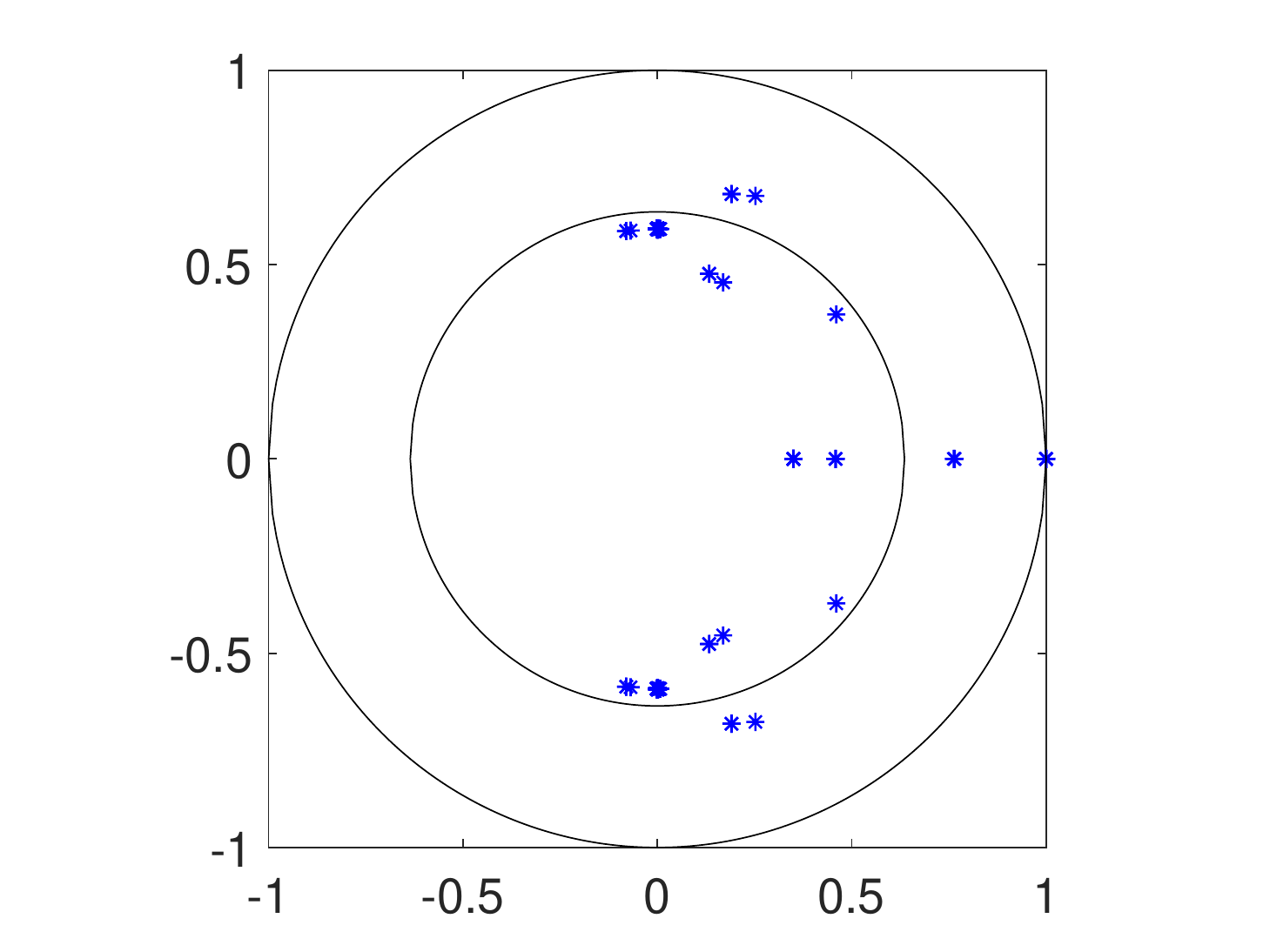}
\includegraphics[width=4.2cm]{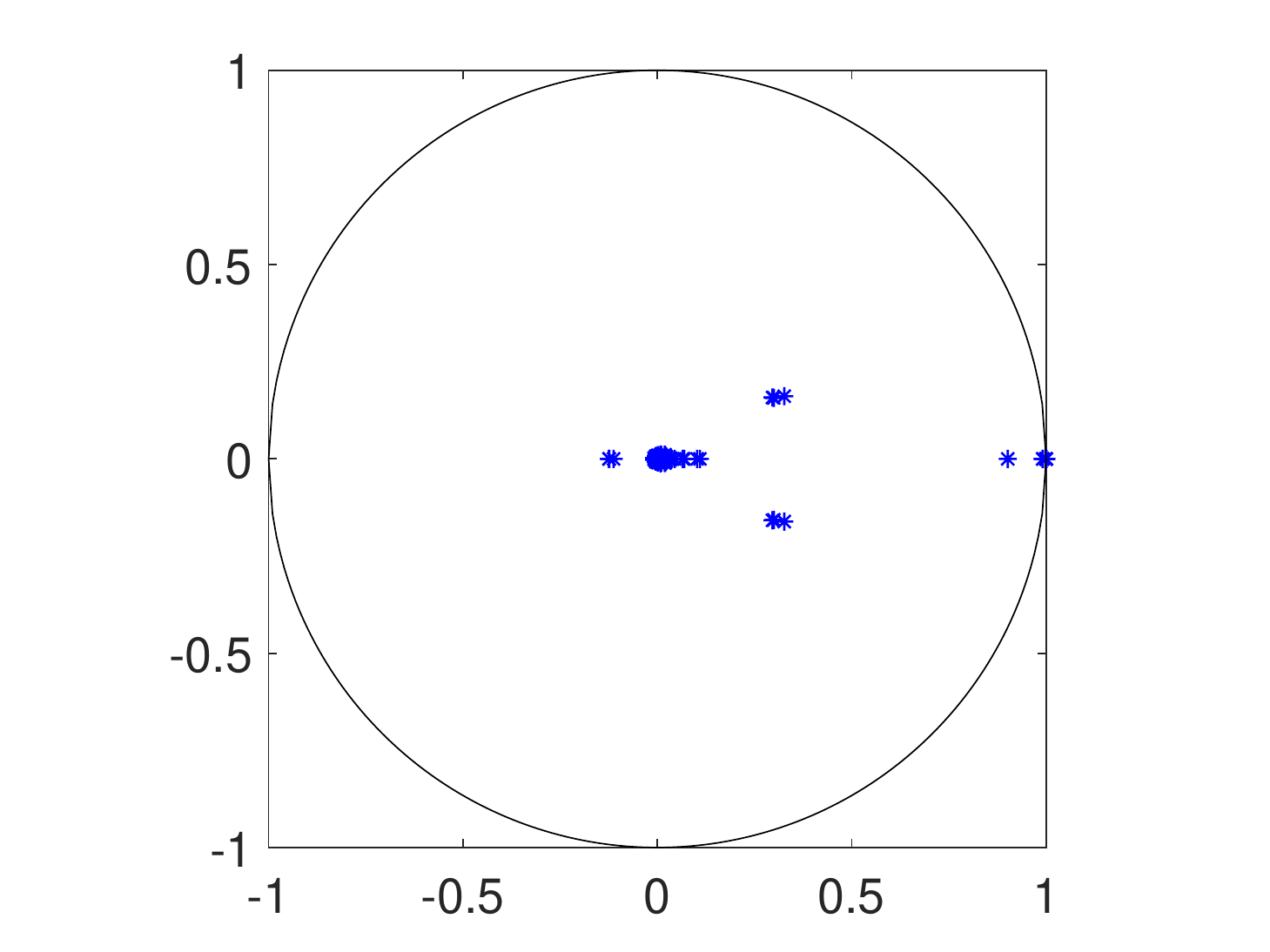}
\includegraphics[width=4.2cm]{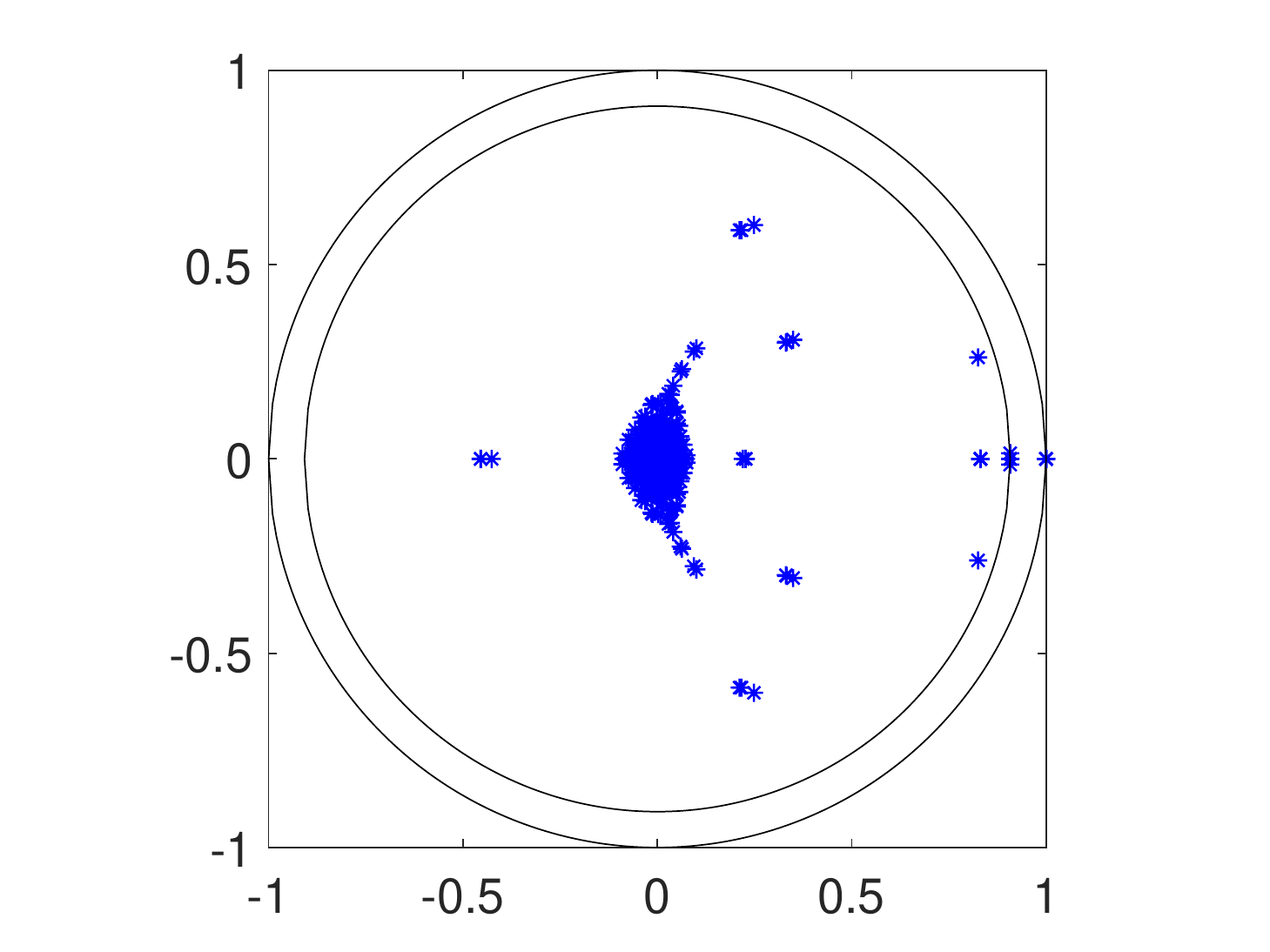}
\includegraphics[width=4.2cm]{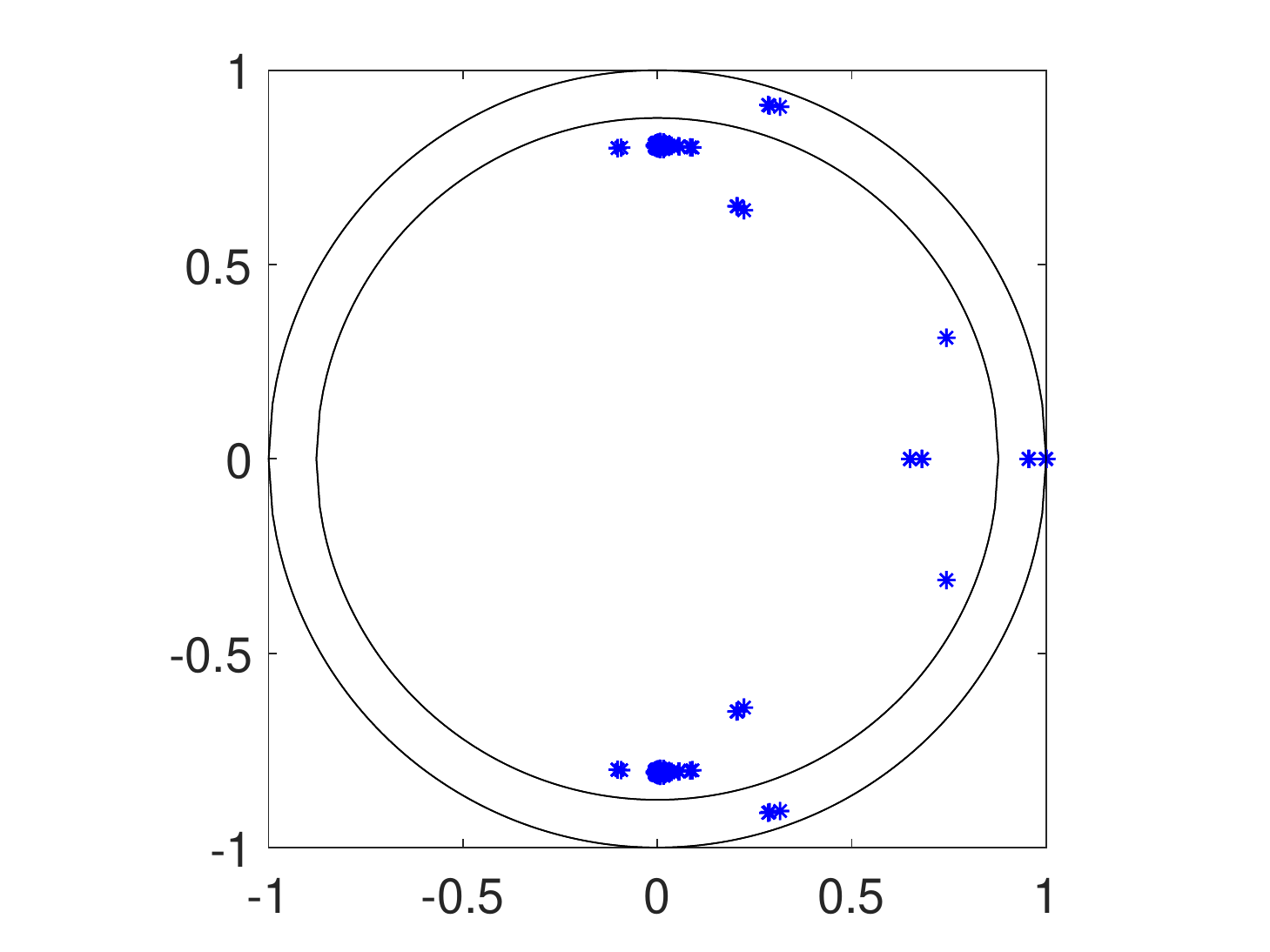}
\caption{
Eigenvalue distributions at $x^*$ for acceleration of ALS by sAA(1) (middle panels) and sNGMRES-R(1) (right panels), for tensor problems with $c=0.7$ (top pannels) and $c=0.9$ (bottom pannels).
(top left) Eigenvalues of $q'_{ALS}$ for $c=0.7$; $\rho(q')=0.906$.
(top middle) Eigenvalues of $T$ for sAA(1)-ALS using the predicted $\beta$ in \cref{opt-beta-real}; $\rho(T)=0.693$.
The radius of the inner circle is $\rho_p$ in \cref{eq:lower-bound-ALS}.
(top right) Eigenvalues of $T$ for sNGMRES-R(1)-ALS using the $\beta$ obtained by using brute-force optimization in \cref{convergence-table-AA-ALS}; $\rho(T)=0.7646$, and the radius of the inner circle is $\rho_{p,N}=0.6357$ from \cref{eq:similar-lower-bound-NGMRES-ALS}.
(bottom left) Eigenvalues of $q'_{ALS}$ for $c=0.9$; $\rho(q')=0.992$.
(bottom middle) Eigenvalues of $T$ for sAA(1)-ALS using the predicted $\beta$ in \cref{opt-beta-real}; $\rho(T)=0.908$.
The radius of the inner circle is $\rho_p$ in \cref{eq:lower-bound-ALS}.
(bottom right) Eigenvalues of $T$ for sNGMRES-R(1)-ALS using the $\beta$ obtained by using brute-force optimization in \cref{convergence-table-AA-ALS}; $\rho(T)=0.9594$, and the radius of the inner circle is $\rho_{p,N}=0.8772$ from \cref{eq:similar-lower-bound-NGMRES-ALS}.
Making abstraction of the eigenvalues one that correspond to the Hessian degeneracy,
the eigenvalues of $q'_{ALS}$ with the largest modulus are real (left panels), and the eigenvalues of $T$ for sAA(1)-ALS with the largest modulus lie on the inner circles (middle panels), in accordance with Conjectures \ref{conjec1-sAA-ALS} and \ref{conjec2-sAA-ALS}.
} \label{eigs-sAA-c07-c09-plot}
\end{figure}

Finally, \cref{eigs-sAA-c07-c09-plot} shows how sAA(1) and sNGMRES-R(1) accelerate ALS for increasingly ill-conditioned tensor problems
with $c=0.7$ and $c=0.9$, with Hessian condition numbers $\bar{\kappa}=123.90$ and $\bar{\kappa}=3837.90$ at
$x^*$. As $\bar{\kappa}$ increases, the ALS convergence factor rapidly deteriorates, to 0.906 and 0.992, and both sAA(1) and
sNGMRES-R(1) manage to improve the optimal asymptotic factors substantially, according to the theoretical results in
 \cref{thm:lower-bound-ALS}, \cref{conjec2-sAA-ALS} and \cref{thm:lowe-upper-bound}.
Although in  \cref{thm:lower-bound-ALS} we only give a lower bound on the optimal convergence factor for sAA(1)-ALS,
we see the bound is achieved for all our examples, in accordance with \cref{conjec2-sAA-ALS}.

\newpage
\subsection{Extending the numerical results of Section \ref{subsec:nonstationary}: asymptotic convergence of nonstationary AA and NGMRES}
\label{subsubsec:nonstationary-stationaryb}
\ \\
Here, we expand on the numerical results from Section \ref{subsubsec:nonstationary-GMRES} on
nonstationary AA and NGMRES. We first provide a remark on convergence speed for $f(x_k)-f(x^*)$.
\begin{remark}\label{rem:f2}
Note that  all the convergence factors $\rho$ discussed in this work are asymptotic for
convergence of $x_k$ to the true solution $x^{*}$:
\begin{equation*}
  \|x_k - x^*\|\approx\rho \|x_{k-1} - x^*\| \quad \text{as}\,\, k\rightarrow \infty.
\end{equation*}
Using the Taylor series for function $f(x)$ in \cref{eq:minf} and the fact $f'(x^{*})=0$ leads to
\begin{eqnarray*}
  f(x) &\approx&  f(x^{*}) + f'(x^{*})(x-x^{*}) + (x-x^{*})^Tf''(x^{*})(x-x^{*})\\
   &=&f(x^{*}) + (x-x^{*})^T H(x^{*})(x-x^{*}).
\end{eqnarray*}
From this we see that
\begin{equation*}
  \|f(x_k)-f(x^{*})\| \approx C\|x_k - x^*\|^2,
\end{equation*}
where $C$ is a constant that depends on the largest modulus of the eigenvalues of $H$.
Therefore,
\begin{equation}\label{eq:conv-order-f}
 \|f(x_k)-f(x^{*})\| \approx \rho^2 \|f(x_{k-1})-f(x^{*})\| \quad \,\,\text{as}\,\, k\rightarrow \infty.
\end{equation}
Relation \cref{eq:conv-order-f} is used to investigate $\rho$ in the numerical results of Section \ref{subsec:nonstationary}
and this Section.
\end{remark}
\begin{figure}[H]
\centering
\includegraphics[width=0.9\linewidth]{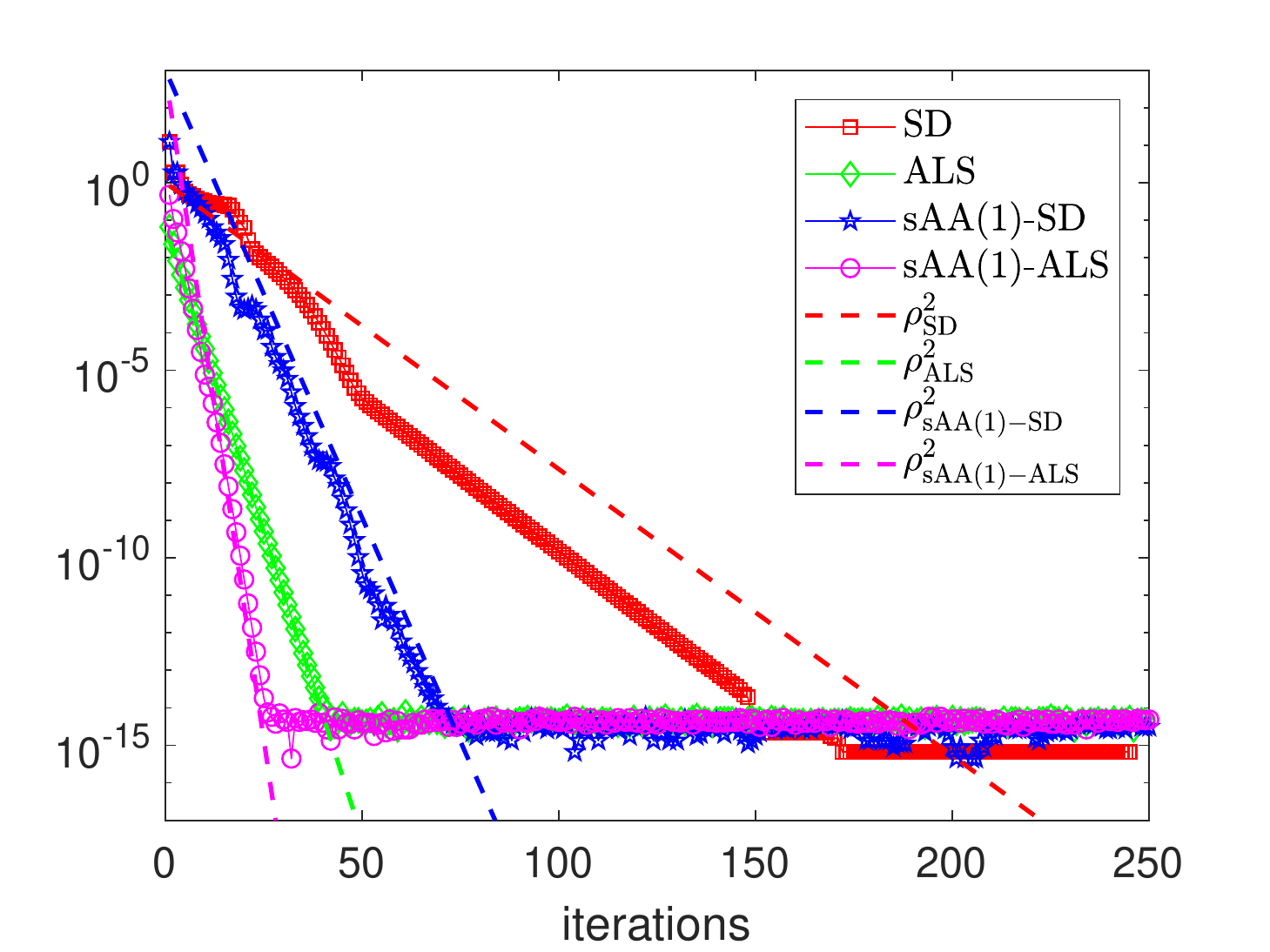}  
\caption{
Comparison of optimal sAA(1) methods for SD and ALS with theoretical asymptotic convergence factors, for a tensor problem
with $c=0.5$. The vertical axis represents $f(x_k)-f(x^*)$, the convergence towards the minimum value of $f(x)$.}
\label{c05-sAA-sNGMRES-plot}
\end{figure}

In \cref{c05-sAA-sNGMRES-plot} we compare convergence plots for nonlinear sAA(1)
iterations with optimal coefficients
for SD and ALS with the theoretical asymptotic convergence factors
$\rho_{\rm sAA(1)-SD}$ from \cref{Thm:positive-negative-AA-SD} and $\rho_{\rm sAA(1)-ALS}$ from \cref{thm:lower-bound-ALS,conjec2-sAA-ALS}, for a tensor problem with $c=0.5$.
For all simulations with SD steps in this section, we use the standard Mor\'{e}-Thuente cubic line search method of \cite{MR1367800} to determine the SD step length $\alpha_k$ in each iteration.
We observe that the nonlinear methods, with line searches for the SD steps and with a globalization mechanism that is based on the cubic line search, attain asymptotic convergence behavior that is consistent with the theoretical asymptotic convergence factors.

\begin{figure}[H]
\centering
\includegraphics[width=0.9\linewidth]{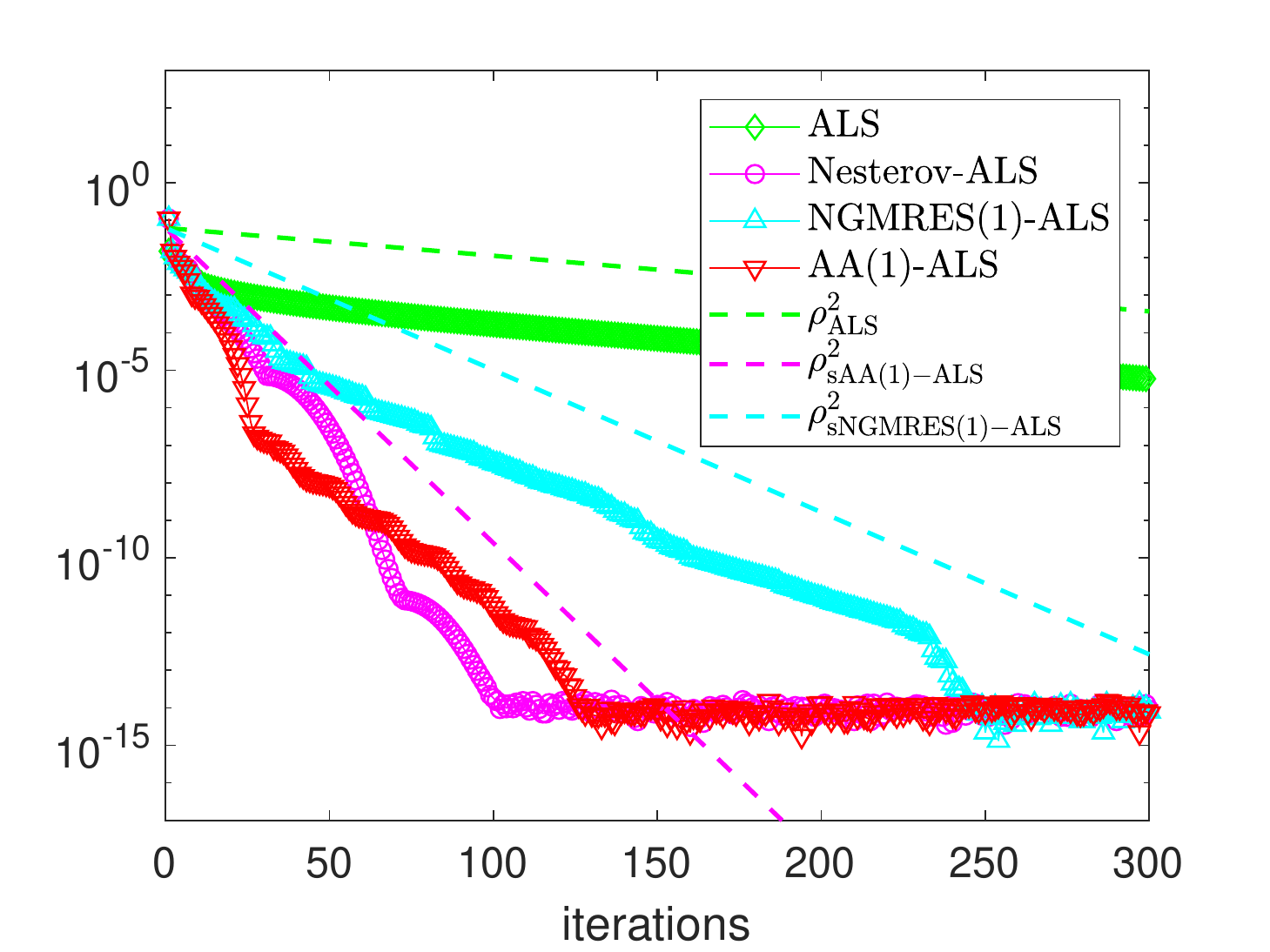}
\caption{
Comparison of the nonstationary AA(1)-ALS, NGMRES(1)-ALS, and Nesterov-ALS methods with theoretical asymptotic convergence factors for optimal stationary methods, for a tensor problem with $c=0.9$. The vertical axis represents $f(x_k)-f(x^*)$, the convergence towards the minimum value of $f(x)$.
}
\label{c05-AA-NGMRES-plotb}
\end{figure}

\cref{c05-AA-NGMRES-plotb} shows how the nonstationary AA, NGMRES and Nesterov methods applied to ALS
show convergence rates that are consistent with the predictions from optimal stationary methods, for an
ill-conditioned tensor problem with $c=0.9$, complementary to the results of \cref{c05-AA-NGMRES-plot}
for $c=0.5$ and $c=0.7$.

\cref{c05-inf-plotb} shows results for additional tensors with $c=0.5$ and $c=0.7$, using random seeds that
are different from \cref{c05-inf-plot}.
While the specific convergence traces for this nonconvex nonlinear problem depend substantially
on the random seed used, these results for additional random seeds confirm the general trends
of \cref{c05-inf-plot}.

\begin{figure}[H]
\centering
\includegraphics[width=0.49\linewidth]{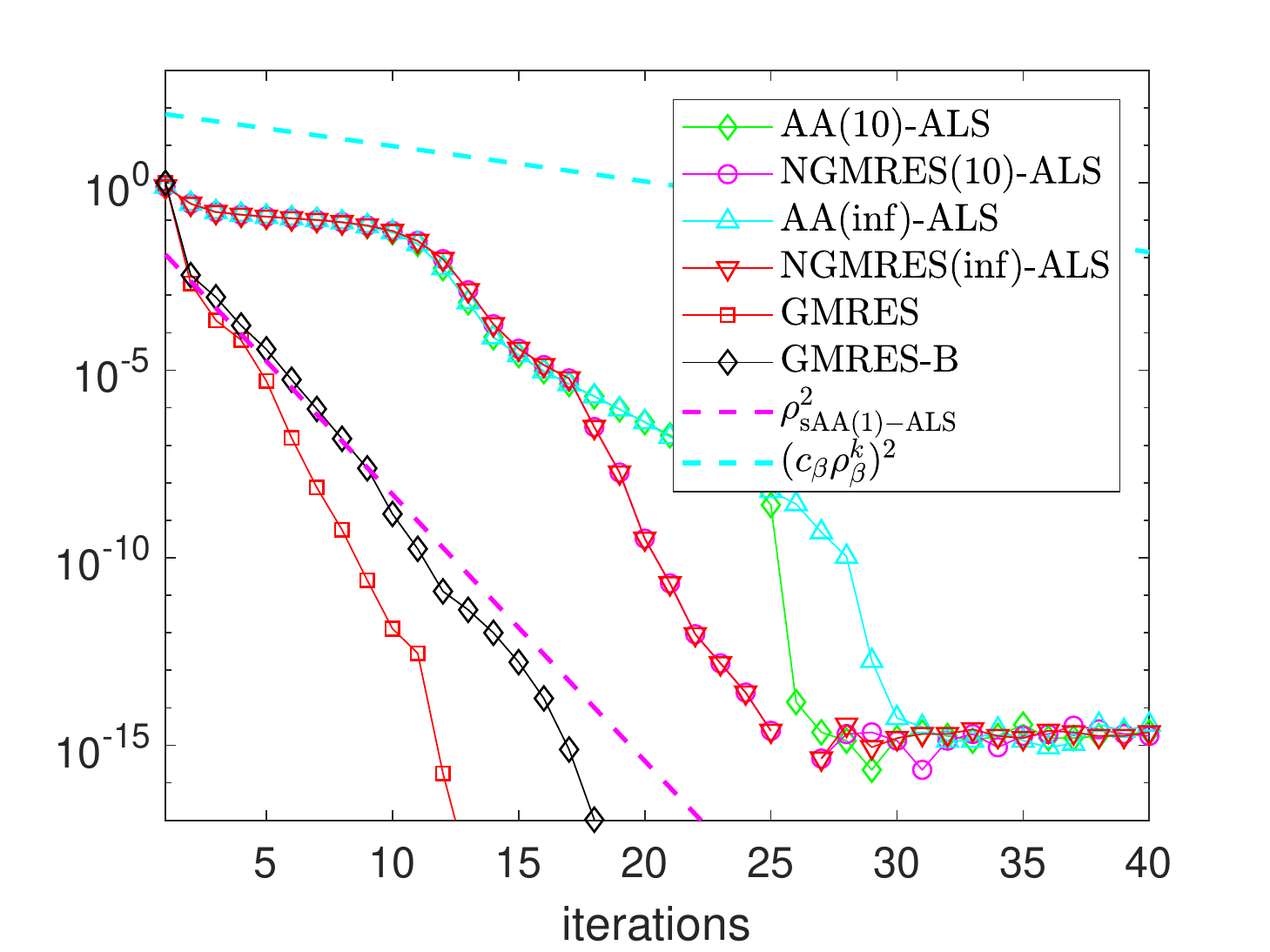}
\includegraphics[width=0.49\linewidth]{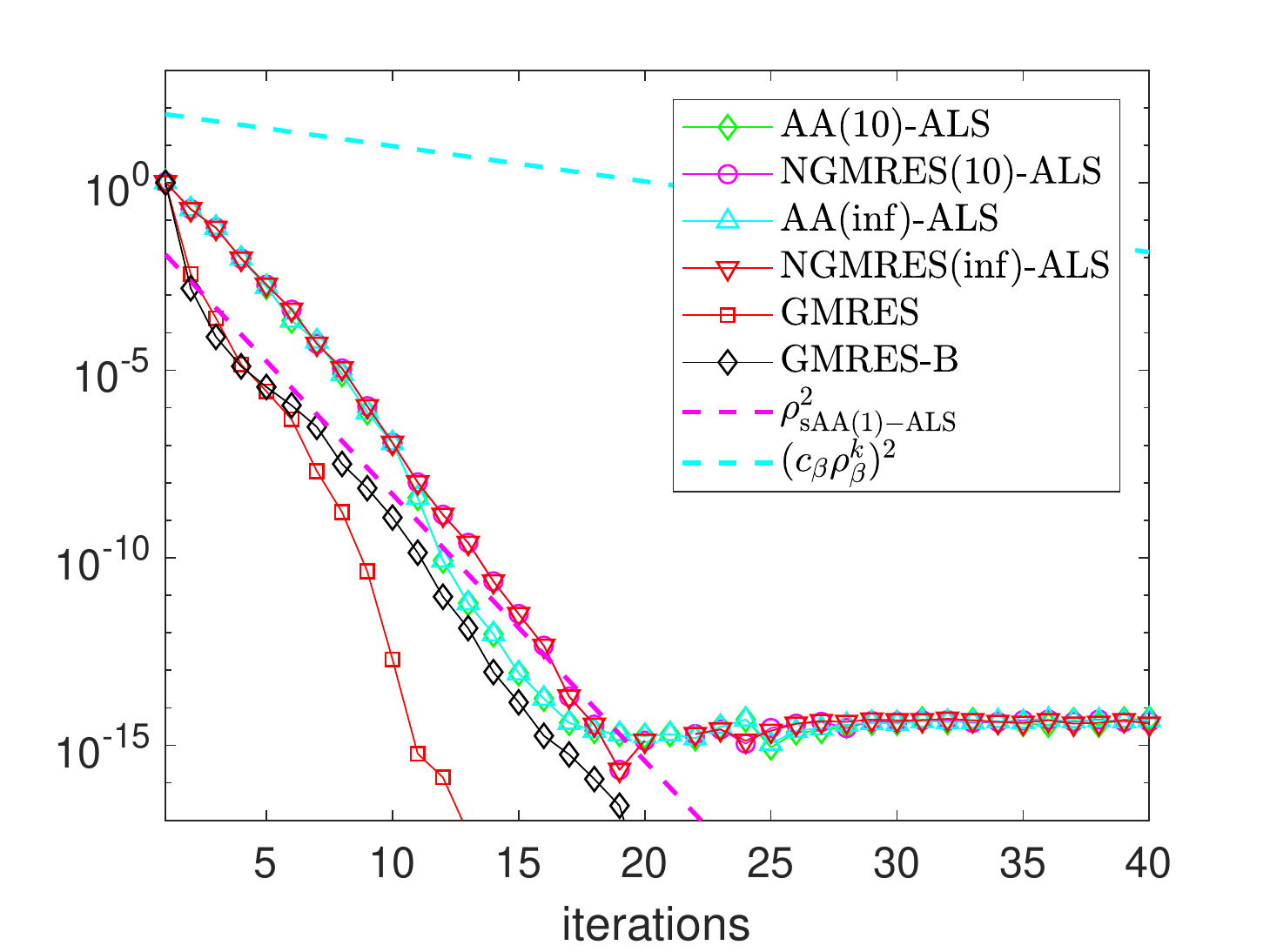}
\includegraphics[width=0.49\linewidth]{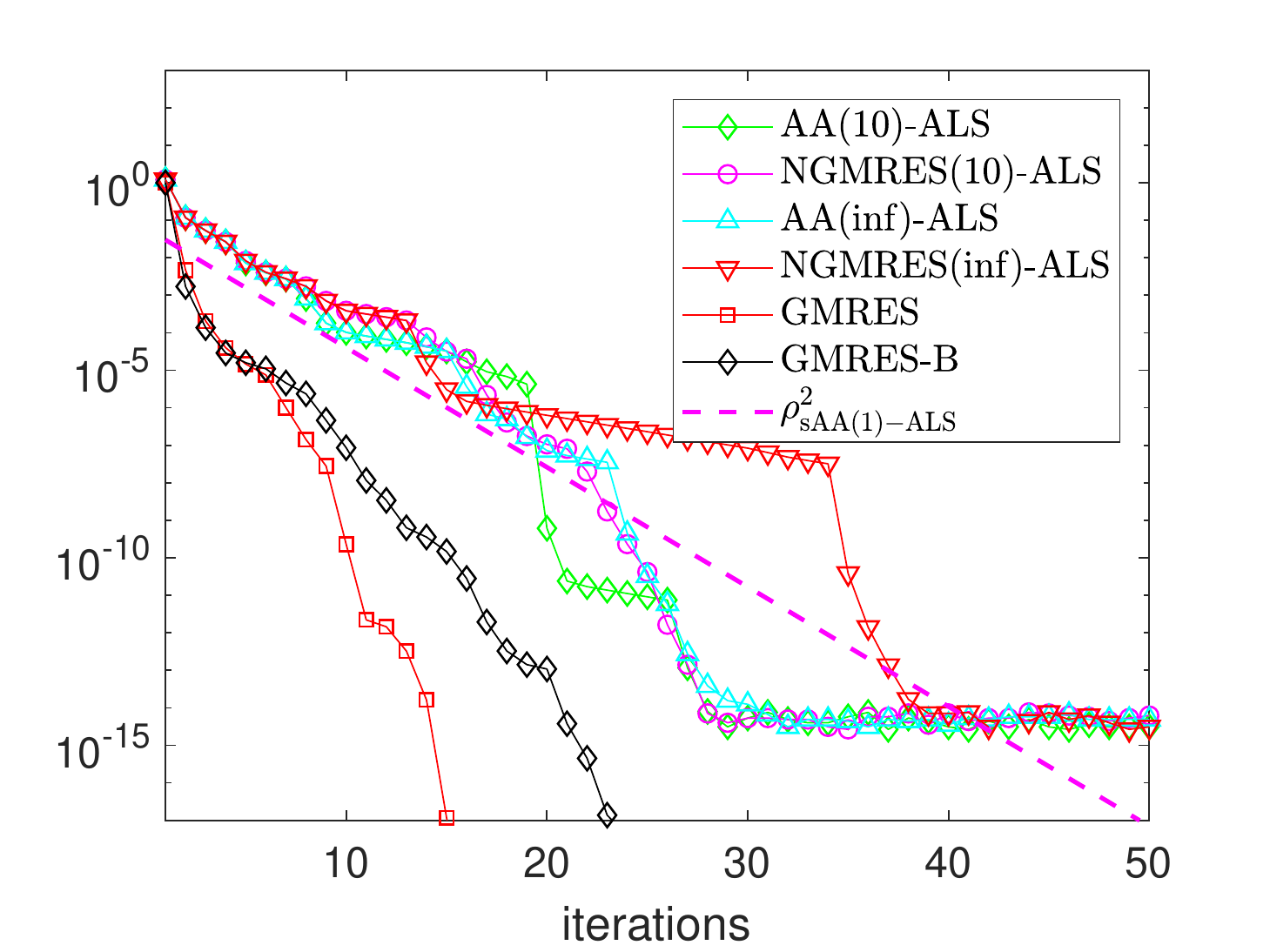}
\includegraphics[width=0.49\linewidth]{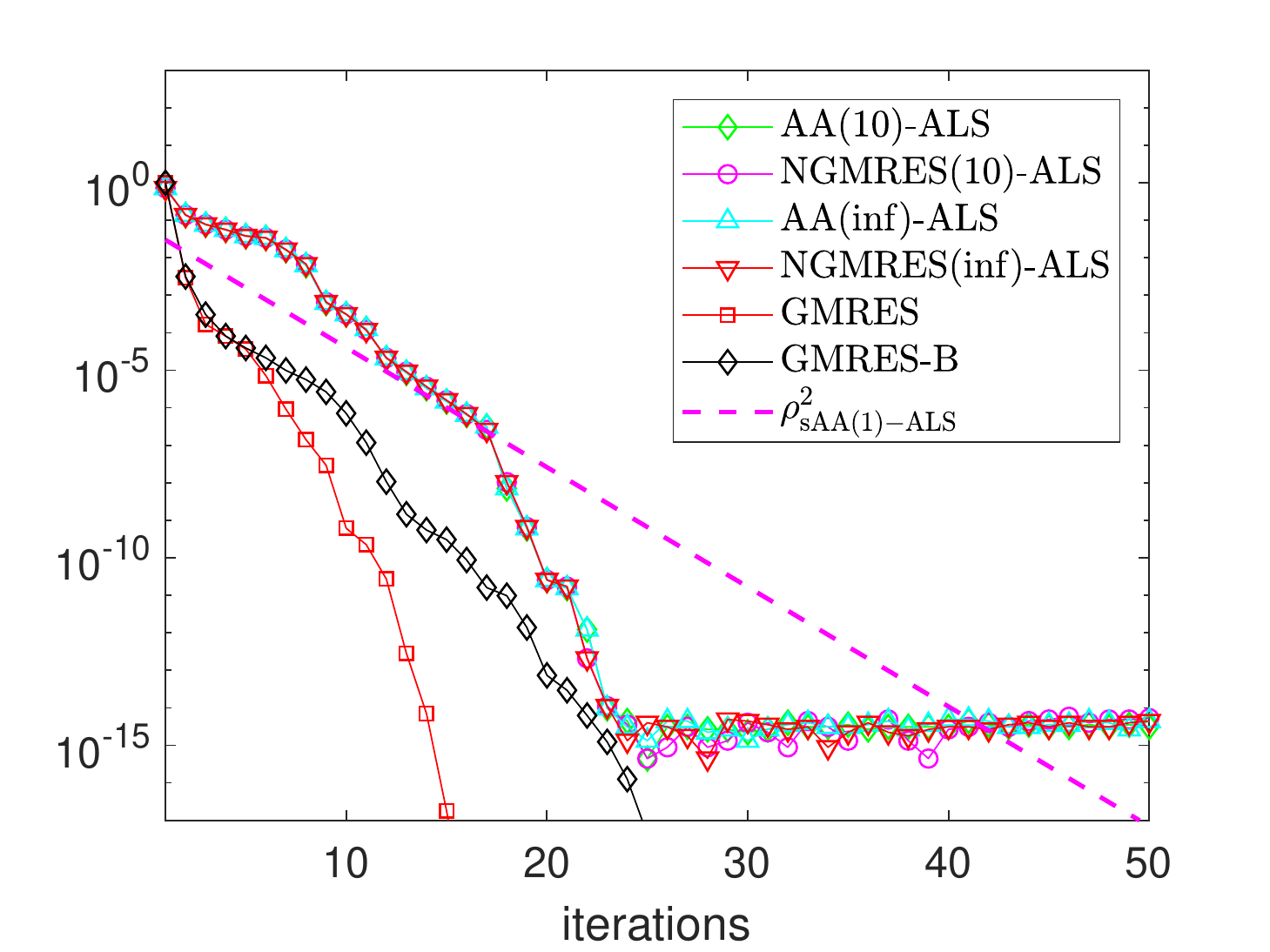}
\caption{Comparison of AA and NGMRES convergence curves for two random tensor problems with $c=0.5$
with different random seeds (top panels), and for two random tensor problems with $c=0.7$
with different random seeds (bottom panels).
The four nonlinear AA and NGMRES curves are compared with GMRES applied to
linearized equation \cref{eq:fixed-lin} and GMRES-B applied to projected
nonsingular linearized system \cref{eq:projected}. For the $c=0.5$ panels,
$c_\beta \rho_{\beta}^k$ computed based on the FOV of \cref{c05-FOV-plot}
provides a pessimistic upper bound.
Our new $\rho_{sAA(1)-ALS}$ from \cref{thm:lower-bound-ALS,conjec2-sAA-ALS} appears to provide
a useful indication of the convergence speed of the linear and nonlinear methods.
For the four nonlinear methods, the vertical axis represents $f(x_k)-f(x^*)$, the convergence towards the minimum value of $f(x)$.
For the GMRES runs, the vertical axis represents $\|r_k\|^2/\|r_0\|^2$.}
\label{c05-inf-plotb}
\end{figure}

\newpage
\subsection{Verifying \cref{conjec1-sAA-ALS} and \cref{conjec2-sAA-ALS} for real-world data}
\ \\

\begin{figure}[H]
\centering
\includegraphics[width=6.cm]{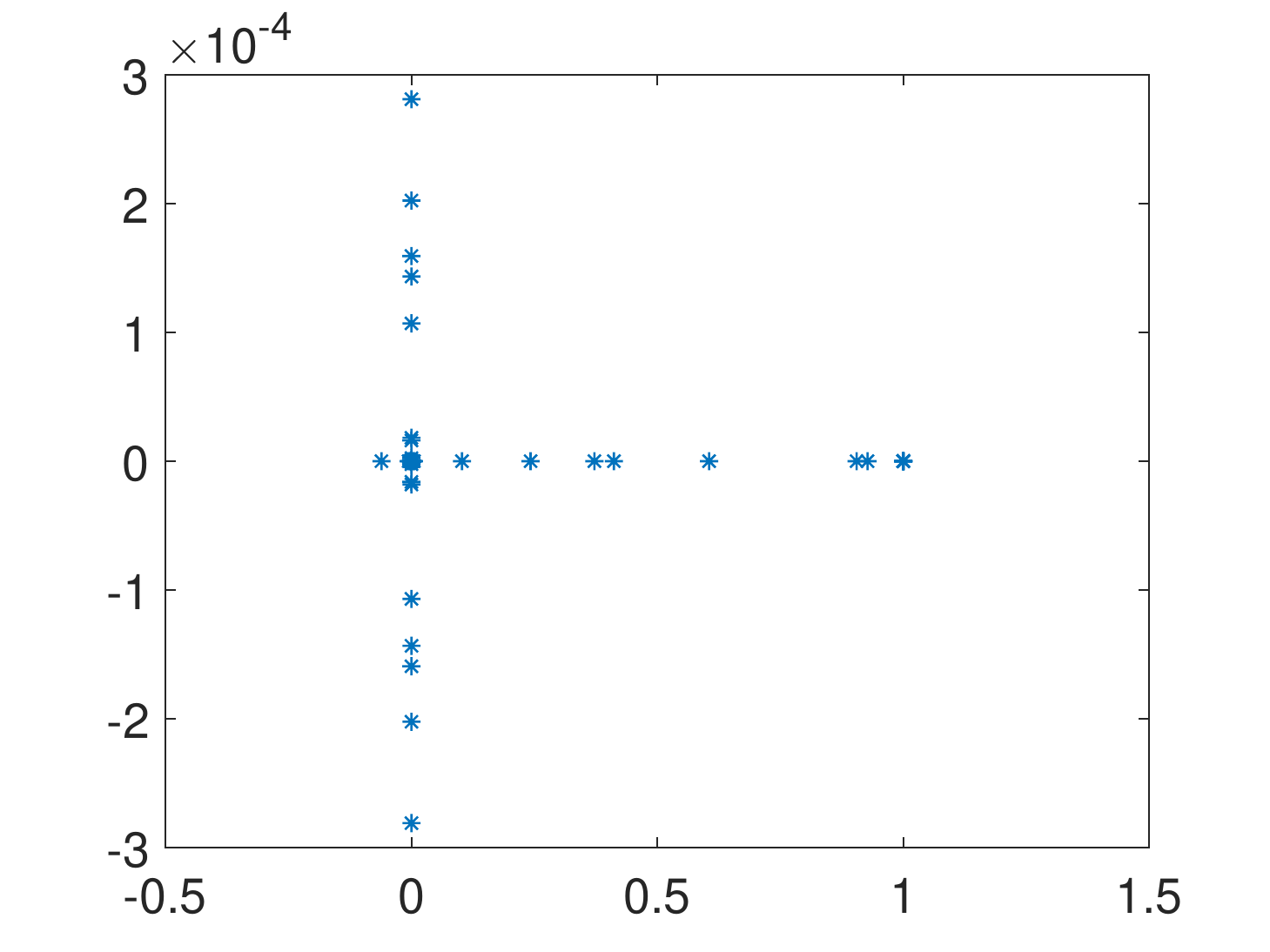}
\includegraphics[width=6.cm]{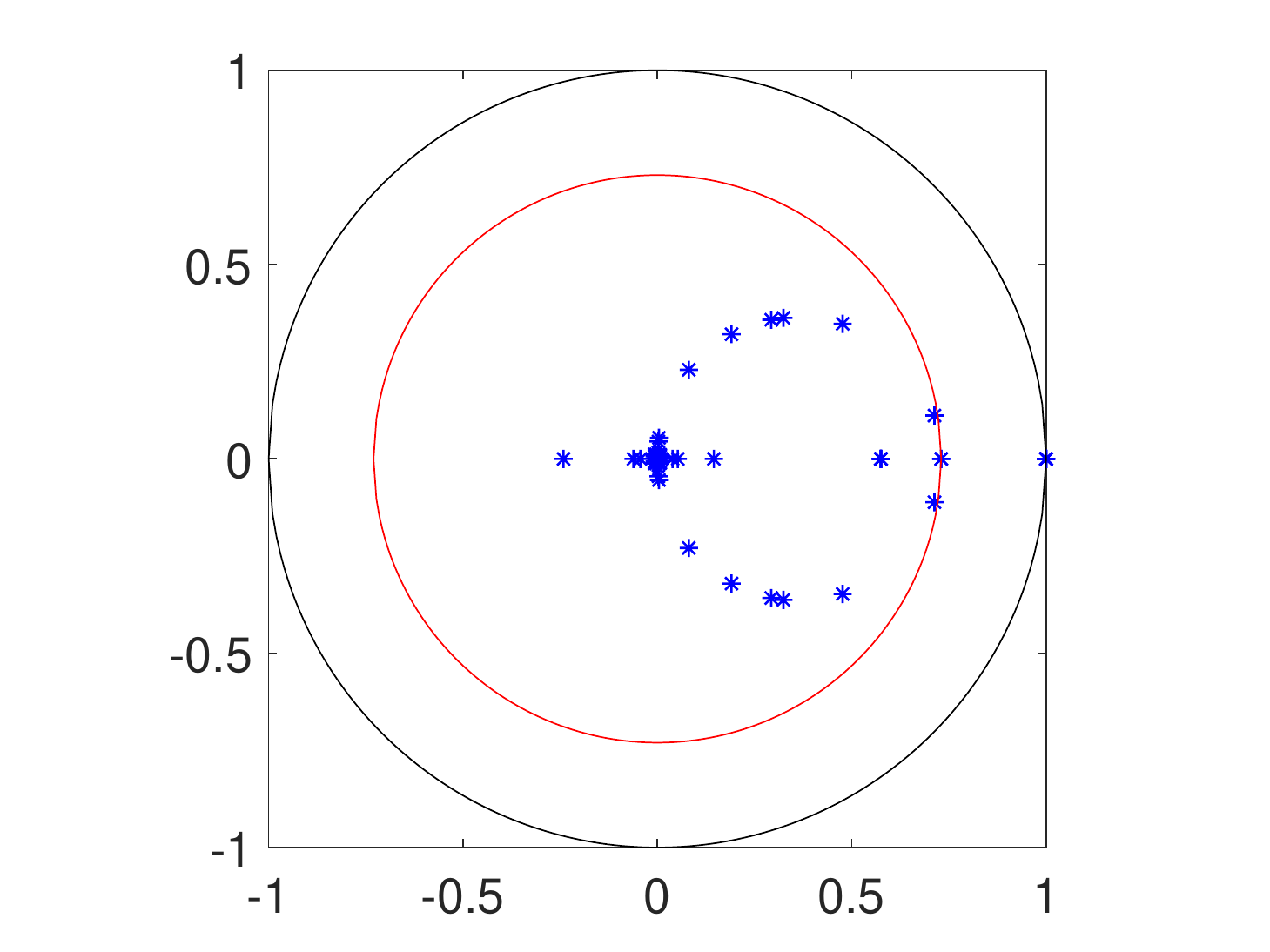}\\
\includegraphics[width=6cm]{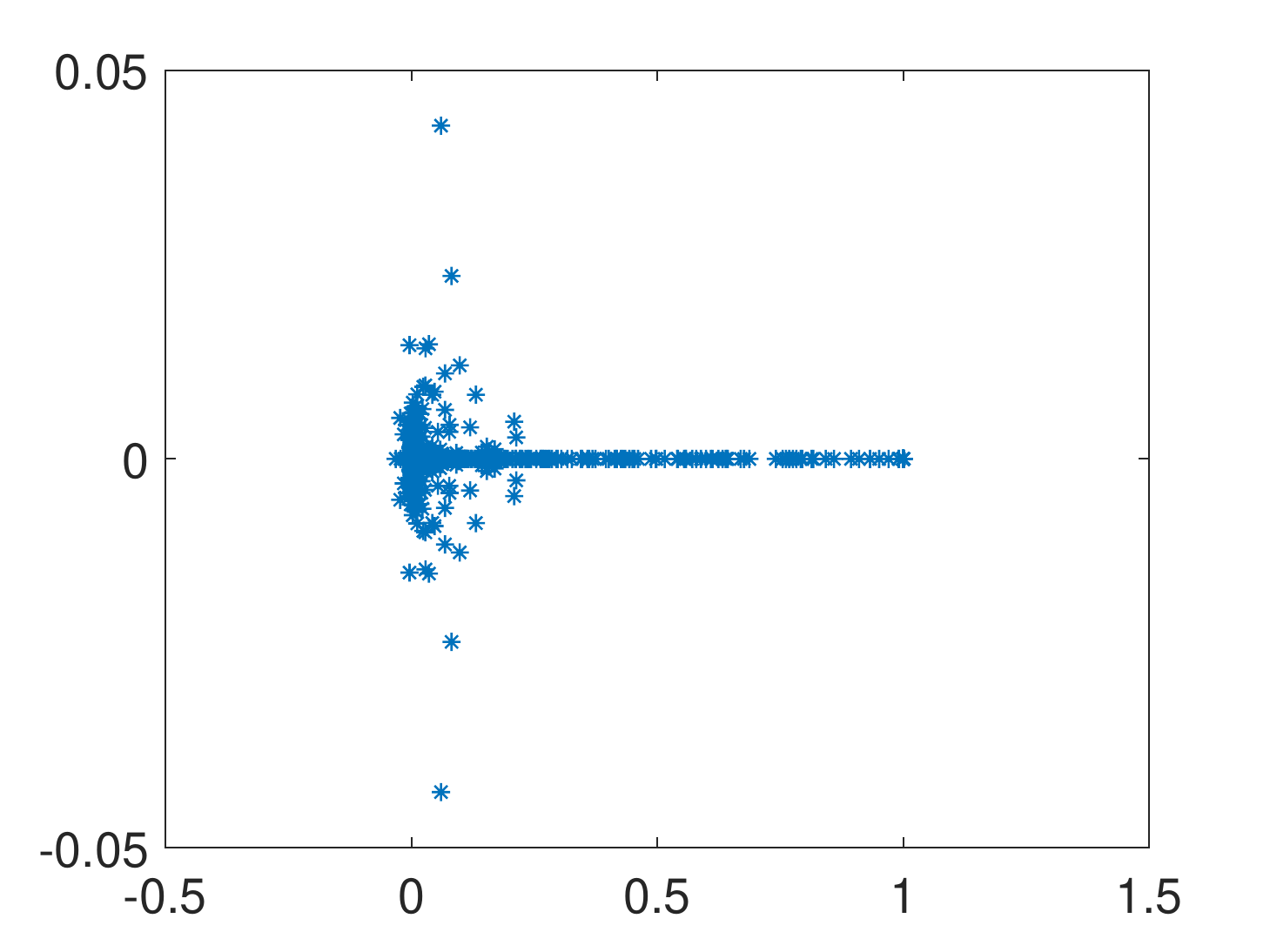}
\includegraphics[width=6cm]{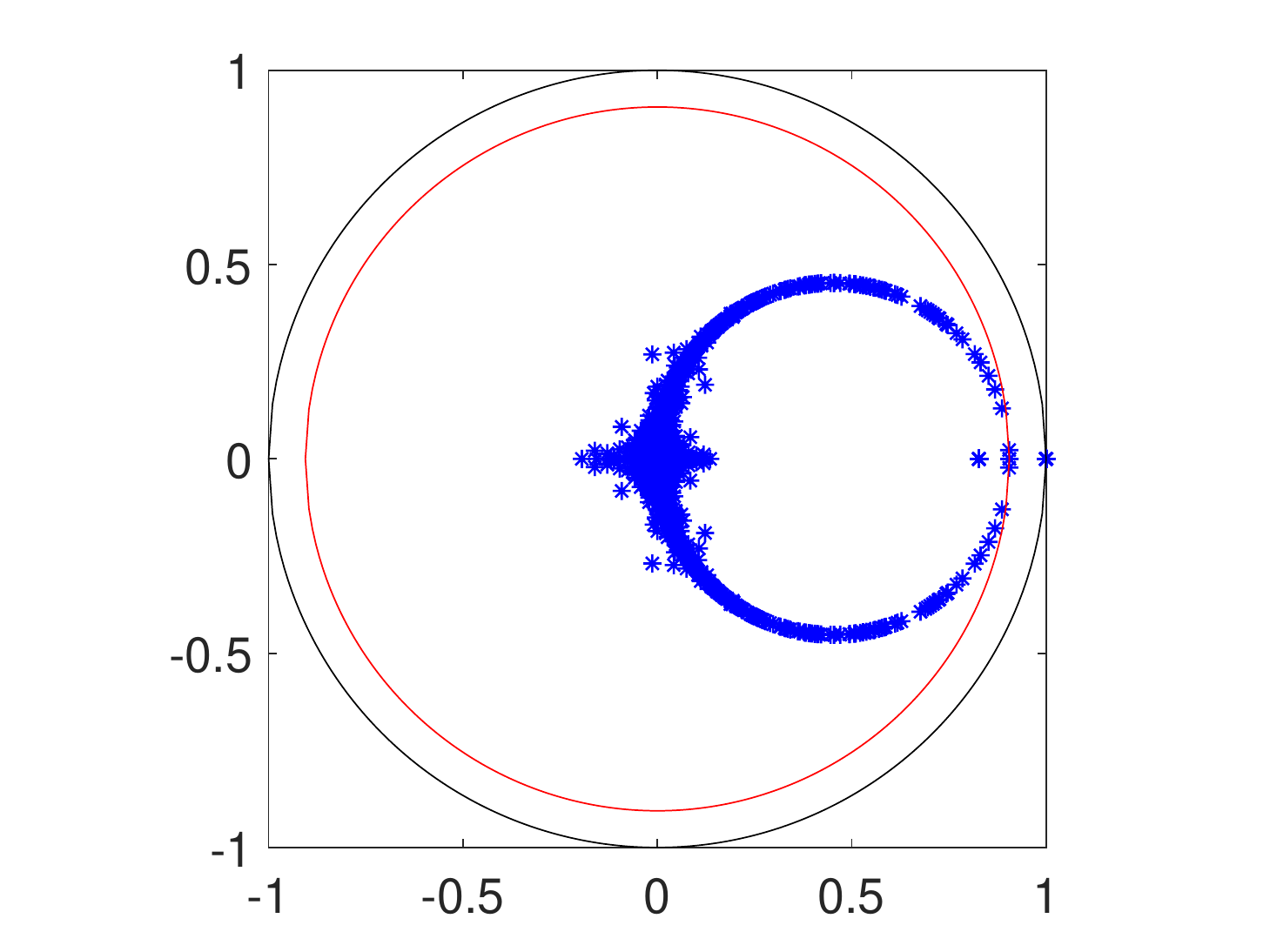}
\caption{
Eigenvalue distributions at $x^*$ for sAA(1) acceleration of  ALS for real-data tensor problems from \cite{mitchell2020nesterov}: Claus data (top row) and Enron data (bottom row).
(top left) Eigenvalues of $q'_{ALS}$ for Claus data.
(top right) Eigenvalues of $T$ for sAA(1)-ALS for Claus data using the predicted $\beta$ in \cref{opt-beta-real}. The radius of the inner circle is $\rho_p$ in \cref{eq:lower-bound-ALS}.
(bottom left) Eigenvalues of $q'_{ALS}$ for Enron data.
(bottom right) Eigenvalues of $T$ for sAA(1)-ALS for Enron data using the predicted $\beta$ in \cref{opt-beta-real}.
The radius of the inner circle is $\rho_p$ in \cref{eq:lower-bound-ALS}.
Making abstraction of the eigenvalues one that correspond to the Hessian degeneracy,
the eigenvalue of $q'_{ALS}$ with the largest modulus is real, and the eigenvalue of $T$ with the largest modulus lies on the inner circle, in accordance with Conjectures \ref{conjec1-sAA-ALS} and \ref{conjec2-sAA-ALS}.
}
\label{eigs-sAA-real-plot}
\end{figure}

\end{document}